\documentclass[12pt]{article}

\usepackage{amsmath,amstext,amssymb,amsthm}
\usepackage{appendix,color,graphicx,hyperref,multirow,lscape}
\usepackage{authblk}
\usepackage[utf8]{inputenc}
\usepackage{tikz}
\usetikzlibrary{arrows}
\hsize \textwidth
\advance \hsize by -\marginparwidth
\evensidemargin \oddsidemargin

\advance\hoffset by 5mm

\def\xjp{\vect x_{j'}}

\def\Ac{{\cal A}}
\def\Bc{{\cal B}}

\def\Ec{{\cal E}}
\def\Gc{{\cal G}}

\def\vect#1{\mbox{\boldmath{$#1$}}}
\def\mC{\mathbb{C}}
\def\mE{\mathbb{E}}
\def\mL{\mathbb{L}}
\def\mR{\mathbb{R}}

\def\wG{\widehat{G}}
\def\wU{\widehat{U}}
\def\wV{\widehat{V}}
\def\wf{\widehat{f}}
\def\wg{\widehat{g}}

\def\wP{\widehat P}
\def\hxi{\hat\psi}
\def\Rho{{X}}
\def\rmd{\mathrm{d}}
\def\rme{\mathrm{e}}
\def\rmi{\mathrm{i}}

\def\te{\rme}
\def\sinc{\mathrm{sinc}}

\def\bfrho{\mbox{\boldmath$\rho$}}
\def\bfgamma{\mbox{\boldmath$\gamma$}}
\def\bfPi{\mbox{\boldmath$\Pi$}}
\def\diag{\mathrm{diag}}
\def\rowsupp{\mathrm{rowsupp}}
\def\IP{\textit{Inverse Problems}}

\newcommand{\fl}{\hspace*{-6pc}}
\newcommand{\numberthis}{\addtocounter{equation}{1}\tag{\theequation}}
\newtheorem{theorem}{Theorem}[section]
\newtheorem{lemma}[theorem]{Lemma}
\newtheorem{remark}[theorem]{Remark}
\newtheorem{proposition}[theorem]{Proposition}
\newtheorem{corollary}[theorem]{Corollary}
\let\oldproofname=\proofname
\renewcommand{\proofname}{\rm\bf{\oldproofname}}
\renewcommand{\qedsymbol}{$\blacksquare$}
\renewenvironment{proof}[1][\proofname]{{\noindent\bfseries #1.\,}}{\qedsymbol}

\title{Array imaging of localized objects in homogeneous and heterogeneous media\footnote{This version: December 11, 2015.}}
\author[1]{Anwei Chai\thanks{anwei@math.stanford.edu}}
\author[2]{Miguel Moscoso\thanks{moscoso@math.uc3m.es}}
\author[3]{George Papanicolaou\thanks{papanico@math.stanford.edu}}
\affil[1]{Institute for Computational and Mathematical Engineering,
Stanford University, Stanford, CA 94305 USA}
\affil[2]{Gregorio Mill\'an Institute,
Universidad Carlos III de Madrid, Madrid 28911, Spain}
\affil[3]{Department of Mathematics,
Stanford University, Stanford, CA 94305 USA}
\date{}
\begin{document}
\maketitle
\begin{abstract}
    We present a comprehensive study of the resolution and stability properties of sparse promoting optimization theories 
    applied to narrow band array imaging of localized scatterers. We consider homogeneous and heterogeneous media, and multiple and
    single scattering situations. When the media is homogeneous with strong multiple scattering between scatterers,
    we give a non-iterative formulation to find the locations and reflectivities of the scatterers from a nonlinear
    inverse problem in two steps, using either single or multiple illuminations. We further introduce an approach that uses the top
    singular vectors of the response matrix as optimal illuminations, which improves the robustness of sparse promoting optimization
    with respect to additive noise. When multiple scattering is negligible, the optimization problem becomes linear and can be
    reduced to a hybrid-$\ell_1$ method when optimal illuminations are used. When the media
    is random, and the interaction with the unknown inhomogeneities can be primarily modeled by wavefront distortions,
    we address the statistical stability of these methods.  
    We analyze the fluctuations of the images obtained with the hybrid-$\ell_1$ method, and we show that it is stable with respect to 
    different realizations of the random medium provided the imaging array is large enough. 
    We compare the performance of the hybrid-$\ell_1$ method in random media to the
    widely used Kirchhoff migration and the multiple signal classification methods. 
\end{abstract}
\smallskip\noindent\textbf{AMS classification scheme numbers.} 34B27, 78A46, 78A48


\section{Introduction}

The recent mathematical theory of compressed sensing \cite{Donoho89,Donoho92,Donoho06,Candes06a,Candes06b} has 
been shown to be very promising in a number of areas as diverse as medicine \cite{Lustig07},
biomedicine \cite{Studer12}, geophysics \cite{Taylor79}, radar \cite{Baraniuk07}, astronomy \cite{Bobin08}, or microscopy \cite{Wu09}.
Most inverse problems in these areas are considered to be underdetermined, meaning that we do not have unique solutions 
and, therefore, it is apparently impossible to identify which one is indeed the correct one.
What makes compressed sensing at once interesting is that, often, the sought solution is known to be 
structured in the sense that it is sparse or compressible, which means that it depends upon a small number of parameters. 
This additional information changes the imaging problem dramatically because we can exploit the sparsity of the image
and look for the simplest one that tends to be the right one.

In this paper, we study narrow-band, active array imaging of a small number of localized scatterers using
both single and multiple illuminations.
The goal is to determine the positions and reflectivities of the scatterers from the echoes recorded at an array of sensors
when a few narrow band signals are sent to probe the medium.
By localized scatterers we mean scatterers whose diameter is small compared to the wavelength. Hence, the Foldy-Lax
approximation to the wave equation can be used to model wave propagation in the medium \cite{F45,L51,L52}. 
The number of scatterers is small because only a small portion of 
the region of interest is occupied by scatterers and, thus, the image we wish to recover is sparse.
We study the case in which the interaction between the scatterers 
is strong so that multiple scattering is important, and the case in which the interaction is small 
so that multiple scattering is negligible. We consider imaging in homogeneous media and imaging in randomly inhomogeneous media with significant
scattering from the inhomogeneities.
We restrict this study to the case in which the full waveform at the array is available for imaging,
which means that in the frequency domain both amplitudes and phases can be measured and recorded.
For the case in which the phases cannot be recorded we refer to \cite{CMP11,Fannjiang12a, Fannjiang12b, Candes13a, Candes13b,Novikov15,Moscoso15}.

In this work we consider narrow-band array systems and, therefore, the frequency diversity content of the data measured at the array is very limited.
There is an extensive literature on imaging techniques that deal with this problem.
Kirchhoff migration \cite{Biondi06}, matched field imaging \cite{Baggeroer93}, and
Multiple Signal Classification (MUSIC) \cite{Schmidt86} are among the most used techniques.
As in narrow-band array imaging the data is scarce, and hence, 
there are infinitely many configurations of scatterers that match the data set,
we formulate active array imaging as an optimization problem with
constraints \cite{CMP13,CMP14}. When multiple scattering between the scatterers is important,
the resulting problem is nonlinear, and therefore, it is apparently impossible
to solve the optimization problem non-iteratively \cite{CMP14}. We show, however,
that the nonlinearity can be avoided through a two-step
process that effectively linearizes the inverse problem. In the first step, we treat the scatterers as equivalent sources 
and we recover their locations and strengths. In the second step, once the locations of the scatterers
are fixed, we recover their true reflectivities using a known relationship between the source
strengths and the scatterer reflectivities. This is an explicit relation that comes from the Foldy-Lax
equations, given the scatterer locations and the illumination. 

When multiple scattering is significant some scatterers may be obscured due to screening effects. Therefore, not all the
scatterers may be recovered from data generated by a single illumination. Indeed, given an array illumination,
multiple scattering may reduce the effective illumination at certain locations due to destructive interferences
of secondary sources coming from all the scatterers. 
To mitigate this additional problem of multiple scattering, we will discuss the use of multiple illuminations. 
The resulting optimization problem with multiple illuminations will be formulated as a joint sparsity
recovery problem where we seek an unknown matrix whose columns share the same support.
Thus, we seek for solution vectors corresponding to different illuminations that have a common support
but have possibly different nonzero values.

The key point of the proposed two-step approach is the possibility of exact recovery of the locations of the equivalent sources 
in the first step.  We give conditions on the array imaging setup and the measurement noise level under which the locations of the sources can be 
recovered exactly. The uniqueness and stability of
the solution are analyzed, showing that the errors are proportional
to the amount of noise in the data with a proportionality factor that depends on the sparsity of the
solution and the mutual coherence of the sensing matrix
\cite{CMP14}. These conditions are given for general imaging configurations. Conditions on the resolution of the images 
that guarantee exact recovery in the paraxial regime are derived in \cite{FSY10} for scatterers whose range is known. 
A more general paraxial model with scatterers at different ranges from the array is considered in \cite{Borcea15}.
The interesting case of imaging scatterers with small off-grid displacements  are studied in \cite{FSY13,Borcea15}.
In  \cite{FSY13}, a simple perturbation method is proposed to reduce the gridding error for off-grid scatterers. The authors in
\cite{Borcea15} also present a very nice discussion on how to interpret the results obtained with $\ell_1$ minimization when
modeling errors due to off-grid displacements are significant.

For the case in which multiple scattering can be ignored,
we introduce a hybrid approach that combines the use of the singular value decomposition (SVD) of the data matrix
with $\ell_1$ minimization \cite{CMP13}. We use the top right singular vectors of this matrix as illumination vectors
to collect the data. Then, we project the data onto the subspace spanned by the top left singular vectors to filter out
the unnecessary data and the noise, and to reduce the dimension of the linear system. Finally, $\ell_1$ optimization is
applied to this reduced linear system  to obtain the sparsest solution. This hybrid-$\ell_1$ method turns out to be
very useful when imaging in random media.

Imaging in random media is fundamentally different from imaging in homogeneous or smoothly varying media.
When the medium is inhomogeneous we know, at best, the large scale, but we cannot known the small scale structure.
Hence, when the small structure of a medium is important we model it as a random spatial process.
In these cases, it is essential to take into consideration the statistical stability of the images,
which refers to the robustness of the imaging methods with respect to different realizations of the medium.
In fact, many of the usual imaging methods used in homogeneous  (or smoothly varying) media fail,
even for broadband signals and large arrays, because the images become noisy and change unpredictably
with the detailed features of the fluctuations of the medium. 
This is the case, for example, of the images obtained with Kirchhoff migration that depend on the particular
realization of the random medium, and thus, they become unstable. Statistical stability typically holds only in broadband,
we refer to  \cite{Borcea05,Borcea06,Borcea11} for details. Here, we consider narrow-band array systems.
We show that, in these cases, large arrays are essential to stabilize the images when the medium is random.
In particular, we show that the hybrid-$\ell_1$ method and MUSIC are efficient and robust when the arrays are large enough. 
We compare, using numerical simulations, the images obtained with these methods with those obtained with Kirchhoff migration.
The numerical simulations show that the hybrid-$\ell_1$ method becomes stable faster than MUSIC as the array size increases.
Kirchhoff migration is, as expected, unstable even for very large arrays.

The analysis of imaging in random media is done using a relatively simple  random phase model for the effects of the random medium.
This model characterizes wave propagation in the high-frequency regime in random media with weak fluctuations and small correlation lengths compared to the wavelength. 
It is widely used, for example, in adaptive optics to compensate for rapidly changing distortions in the received wavefronts due to
the atmospheric turbulence caused by changing temperature and wind conditions. 

The paper is organized as follows. In \S\ref{sec:formulation}, we formulate the array imaging problem in homogeneous media
when the multiple scattering is important. In \S\ref{sec:homogeneous medium}, we describe the optimization methods that
determine the locations and reflectivities of scatterers using a two-step non-iterative approach, with and without multiple illuminations.
In \S\ref{sec:single scattering}, we consider the case where multiple scattering is negligible, and we give a hybrid-$\ell_1$ method
that improves the resolution of the image. In \S\ref{sec:random medium}, we consider imaging in random media.
Using the simple random phase model, we show that the hybrid-$\ell_1$ and the MUSIC methods are statistically
stable provided the arrays are large. The effectiveness of all the methods is illustrated in various numerical
examples with comparisons to other imaging methods in each of the sections. Section \ref{sec:conclusions} contains our conclusions.
The proofs of the theoretical results are given in the appendix.

\section{Data model}\label{sec:formulation}
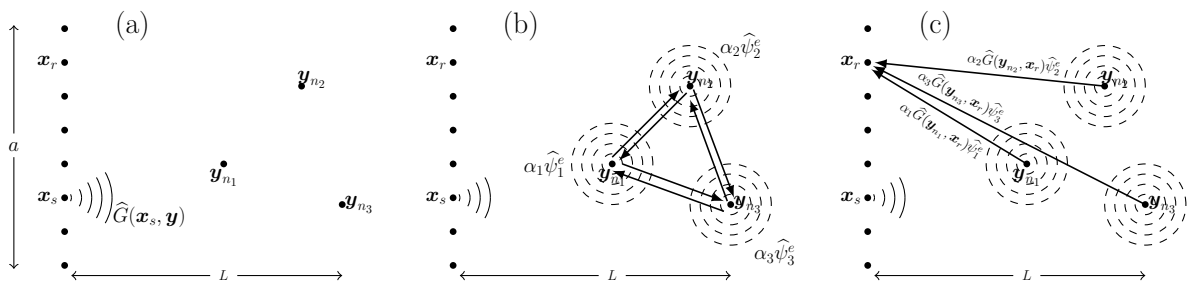
\begin{figure}[htbp]
    \centering
    \begin{tikzpicture}[scale=0.45, transform shape]
       \node at (3.5, 7.1) {\huge{(a)}};
        \draw [<-] (0, -0.1) -- (0, 3.2);
        \node at (0, 3.5) {\Large{$a$\Large}};
        \draw [->] (0, 3.9) -- (0, 7.1);
        \fill[fill=black] (1.5, 0) circle [radius=0.1];
        \fill[fill=black] (1.5, 1) circle [radius=0.1];
        \fill[fill=black] (1.5, 2) circle [radius=0.1];
        \node at (1, 2) {\Large{$\vect x_s$}};
        \fill[fill=black] (1.5, 3) circle [radius=0.1];
        \fill[fill=black] (1.5, 4) circle [radius=0.1];
        \fill[fill=black] (1.5, 5) circle [radius=0.1];
        \fill[fill=black] (1.5, 6) circle [radius=0.1];
        \node at (1, 6) {\Large{$\vect x_r$}};
        \fill[fill=black] (1.5, 7) circle [radius=0.1];
        \draw [<-] (1.7, -0.3) -- (5.8, -0.3);
        \node at (6.2, -0.3) {$L$};
        \draw [->] (6.6, -0.3) -- (9.7, -0.3);
        \fill[fill=black] (6.2, 3) circle [radius=0.1];
        \node at (6.2, 2.6) {\Large{$\vect y_{n_1}$}};
        \fill[fill=black] (8.5, 5.3) circle [radius=0.1];
        \node at (8.8, 5.5) {\Large{$\vect y_{n_2}$}};
        \fill[fill=black] (9.7, 1.8) circle [radius=0.1];
        \node at (10.2, 1.8) {\Large{$\vect y_{n_3}$}};
        \draw (1.7, 2) arc[radius=0.2, start angle=0, end angle=30];
        \draw (1.7, 2) arc[radius=0.2, start angle=0, end angle=-30];
        \draw (2, 2) arc[radius=0.5, start angle=0, end angle=30];
        \draw (2, 2) arc[radius=0.5, start angle=0, end angle=-30];
        \draw (2.3, 2) arc[radius=0.9, start angle=0, end angle=30];
        \draw (2.3, 2) arc[radius=0.9, start angle=0, end angle=-30];
        \draw (2.6, 2) arc[radius=1.2, start angle=0, end angle=30];
        \draw (2.6, 2) arc[radius=1.2, start angle=0, end angle=-30];
        \draw (2.9, 2) arc[radius=1.5, start angle=0, end angle=30];
        \draw (2.9, 2) arc[radius=1.5, start angle=0, end angle=-30];
        \node at (4, 1.5) {\Large{$\wG(\vect x_s, \vect y)$}};
    \end{tikzpicture}
\hspace{0.3cm}
\begin{tikzpicture}[scale=0.45, transform shape]
       \node at (3.5, 7.1) {\huge{(b)}};
        \fill[fill=black] (1.5, 0) circle [radius=0.1];
        \fill[fill=black] (1.5, 1) circle [radius=0.1];
        \fill[fill=black] (1.5, 2) circle [radius=0.1];
        \node at (1, 2) {\Large{$\vect x_s$}};
        \fill[fill=black] (1.5, 3) circle [radius=0.1];
        \fill[fill=black] (1.5, 4) circle [radius=0.1];
        \fill[fill=black] (1.5, 5) circle [radius=0.1];
        \fill[fill=black] (1.5, 6) circle [radius=0.1];
        \node at (1, 6) {\Large{$\vect x_r$}};
        \fill[fill=black] (1.5, 7) circle [radius=0.1];
        \draw [<-] (1.7, -0.3) -- (5.8, -0.3);
        \node at (6.2, -0.3) {$L$};
        \draw [->] (6.6, -0.3) -- (9.7, -0.3);
        \fill[fill=black] (6.2, 3) circle [radius=0.1];
        \node at (6.2, 2.6) {\Large{$\vect y_{n_1}$}};
        \fill[fill=black] (8.5, 5.3) circle [radius=0.1];
        \node at (8.8, 5.5) {\Large{$\vect y_{n_2}$}};
        \fill[fill=black] (9.7, 1.8) circle [radius=0.1];
        \node at (10.2, 1.8) {\Large{$\vect y_{n_3}$}};
        \draw (1.7, 2) arc[radius=0.2, start angle=0, end angle=30];
        \draw (1.7, 2) arc[radius=0.2, start angle=0, end angle=-30];
        \draw (2, 2) arc[radius=0.5, start angle=0, end angle=30];
        \draw (2, 2) arc[radius=0.5, start angle=0, end angle=-30];
        \draw (2.3, 2) arc[radius=0.9, start angle=0, end angle=30];
        \draw (2.3, 2) arc[radius=0.9, start angle=0, end angle=-30];
        \draw (2.6, 2) arc[radius=1.2, start angle=0, end angle=30];
        \draw (2.6, 2) arc[radius=1.2, start angle=0, end angle=-30];
        \draw[dashed] (6.2, 3) circle [radius=0.3];
        \draw[dashed] (6.2, 3) circle [radius=0.6];
        \draw[dashed] (6.2, 3) circle [radius=0.9];
        \draw[dashed] (6.2, 3) circle [radius=1.2];
        \node at (4.2, 3) {\Large{$\alpha_1\widehat{\psi}^e_1$}};
        \draw[dashed] (8.5, 5.3) circle [radius=0.3];
        \draw[dashed] (8.5, 5.3) circle [radius=0.6];
        \draw[dashed] (8.5, 5.3) circle [radius=0.9];
        \draw[dashed] (8.5, 5.3) circle [radius=1.2];
        \node at (10, 6.5) {\Large{$\alpha_2\widehat{\psi}^e_2$}};
        \draw[dashed] (9.7, 1.8) circle [radius=0.3];
        \draw[dashed] (9.7, 1.8) circle [radius=0.6];
        \draw[dashed] (9.7, 1.8) circle [radius=0.9];
        \draw[dashed] (9.7, 1.8) circle [radius=1.2];
        \node at (11, 0.4) {\Large{$\alpha_3\widehat{\psi}^e_3$}};
        \draw [-latex,line width=0.6pt] (6.2, 3.2) -- (8.2, 5.2);
        \draw [-latex,line width=0.6pt] (8.4, 5.1) -- (6.4, 3.1);
        \draw [-latex,line width=0.6pt] (6.5, 3) -- (9.5, 1.9);
        \draw [-latex,line width=0.6pt] (9.5, 1.6) -- (6.2, 2.8);
        \draw [-latex,line width=0.6pt] (9.6, 1.9) -- (8.5, 4.9);
        \draw [-latex,line width=0.6pt] (8.6, 5.2) -- (9.8, 2.0);
    \end{tikzpicture}
\hspace{0.2cm}
    \begin{tikzpicture}[scale=0.45, transform shape]
       \node at (3.5, 7.1) {\huge{(c)}};
        \fill[fill=black] (1.5, 0) circle [radius=0.1];
        \fill[fill=black] (1.5, 1) circle [radius=0.1];
        \fill[fill=black] (1.5, 2) circle [radius=0.1];
        \node at (1, 2) {\Large{$\vect x_s$}};
        \fill[fill=black] (1.5, 3) circle [radius=0.1];
        \fill[fill=black] (1.5, 4) circle [radius=0.1];
        \fill[fill=black] (1.5, 5) circle [radius=0.1];
        \fill[fill=black] (1.5, 6) circle [radius=0.1];
        \node at (1, 6) {\Large{$\vect x_r$}};
        \fill[fill=black] (1.5, 7) circle [radius=0.1];
        \draw [<-] (1.7, -0.3) -- (5.8, -0.3);
        \node at (6.2, -0.3) {$L$};
        \draw [->] (6.6, -0.3) -- (9.7, -0.3);
        \fill[fill=black] (6.2, 3) circle [radius=0.1];
        \node at (6.2, 2.6) {\Large{$\vect y_{n_1}$}};
        \fill[fill=black] (8.5, 5.3) circle [radius=0.1];
        \node at (8.8, 5.5) {\Large{$\vect y_{n_2}$}};
        \fill[fill=black] (9.7, 1.8) circle [radius=0.1];
        \node at (10.2, 1.8) {\Large{$\vect y_{n_3}$}};
        \draw (1.7, 2) arc[radius=0.2, start angle=0, end angle=30];
        \draw (1.7, 2) arc[radius=0.2, start angle=0, end angle=-30];
        \draw (2, 2) arc[radius=0.5, start angle=0, end angle=30];
        \draw (2, 2) arc[radius=0.5, start angle=0, end angle=-30];
        \draw (2.3, 2) arc[radius=0.9, start angle=0, end angle=30];
        \draw (2.3, 2) arc[radius=0.9, start angle=0, end angle=-30];
        \draw (2.6, 2) arc[radius=1.2, start angle=0, end angle=30];
        \draw (2.6, 2) arc[radius=1.2, start angle=0, end angle=-30];
        \draw[dashed] (6.2, 3) circle [radius=0.3];
        \draw[dashed] (6.2, 3) circle [radius=0.6];
        \draw[dashed] (6.2, 3) circle [radius=0.9];
        \draw[dashed] (6.2, 3) circle [radius=1.2];
        \draw[dashed] (8.5, 5.3) circle [radius=0.3];
        \draw[dashed] (8.5, 5.3) circle [radius=0.6];
        \draw[dashed] (8.5, 5.3) circle [radius=0.9];
        \draw[dashed] (8.5, 5.3) circle [radius=1.2];
        \draw[dashed] (9.7, 1.8) circle [radius=0.3];
        \draw[dashed] (9.7, 1.8) circle [radius=0.6];
        \draw[dashed] (9.7, 1.8) circle [radius=0.9];
        \draw[dashed] (9.7, 1.8) circle [radius=1.2];
        \draw[-latex,line width=0.6pt] (6.2, 3) -- (1.6, 5.8) node[pos=0.5, sloped, below]{$\alpha_1\wG(\vect y_{n_1}, \vect x_r)\widehat{\psi}^e_1$};
        \draw[-latex,line width=0.6pt] (8.5, 5.3) -- (1.7, 6) node[pos=0.4, sloped, above]{$\alpha_2\wG(\vect y_{n_2}, \vect x_r)\widehat{\psi}^e_2$};
        \draw[-latex,line width=0.6pt] (9.7, 1.8) -- (1.65, 5.9) node[pos=0.7, sloped, above]{$\alpha_3\wG(\vect y_{n_3}, \vect x_r)\widehat{\psi}^e_3$};
    \end{tikzpicture}
    \caption{Schematic. $(a)$ A linear array with $N=8$ transducers probe a medium sending a spherical wave from $\vect x_s$; the illumination vector is $\vect \wf=[0,0,0,0,0,\wf_6,0,0]$. The medium is either homogeneous or inhomogeneous and the corresponding Green's function is $\wG(\vect x, \vect y)$. (b) There are $M=3$ point-like scatterers with reflectivities $\alpha_j$ at positions $\vect y_{n_j}$, $j=1,2,3$. The interaction between them is strong creating effective (or secondary) sources $\gamma_{n_j}=\alpha_j\widehat{\psi}^e_j$ that depend on all the scatterers' positions and their reflectivities. If those are known, the effective fields $\widehat{\psi}^e_j$ can be computed by solving the Foldy-Lax equations described in Appendix~\ref{appendix:FLmodel}. (c) The response received at $\vect x_r$ is the superposition of all the scattered waves $\widehat{\psi}^s_j=\alpha_j\wG(\vect y_{n_j}, \vect x_r)\widehat{\psi}^e_j$ from the scatterers.}
    \label{fig:schematic}
\end{figure}
Probing of the medium can be done with many different types of arrays, transmitters and recording devices.
Also transducers that can both sense and transmit are usually employed. Besides, the geometric layout of the arrays
depend on the application (acoustics, seismology, radar, ...) and they may be arranged in a transmission or
a backscattering configuration. To fix ideas, we will consider an active array consisting of $N$ transducers
located at positions $\vect x_i$, $i=1,\dots,N$, placed in front of the medium to be probed (see Fig. \ref{fig:schematic} (a)). Here, and in the
rest of the paper, we use boldface lower case letters for vectors, capital letters in boldface for matrices,
and the correponding letters, without boldface, for the entries of the matrices.
To ensure that the transducers behave like an array of aperture $a = (N - 1) h$, and not like separate entities, they are
separated by a distance $h$ of the order of wavelength $\lambda=\frac{2\pi c_0}{\omega}$ of the probing signals,
where $c_0$ is the wave speed in homogeneous medium and $\omega$ is the corresponding frequency.

We will assume that the object we wish to image consists of $M$ randomly positioned point-like scatterers.
The medium can be homogeneous or inhomogeneous. Multiple scattering among the scatterers may or may not be important. 
All the scatterers, with unknown reflectivities $\alpha_j\in\mC$, where $\mC$ stands for the complex field,
and positions $\vect y_{n_j}$, $j=1,\ldots,M$, are within a region of interest called the image window (IW),
which is centered at a distance $L$ from the array. We discretize the IW using a uniform grid of points
$\vect y_j$, $j=1,\ldots,K$, and we introduce the true {\it reflectivity vector}
$$\vect\rho_0=[\rho_{01},\ldots,\rho_{0K}]^T\in\mC^K\, ,$$ 
such that
$\rho_{0k}=\sum_{j=1}^M\alpha_j\delta_{\vect y_{n_j}\vect y_k},\,\, k=1,\ldots,K,$ where
$\delta_{\cdot\cdot}$ is the classical Kronecker delta and $\cdot^T$ is the transpose only operation,
while $\cdot^\ast$ stands for the conjugate transpose.
We further assume that each scatterer is located at one of the $K$ grid points, so
$\{\vect y_{n_1},\ldots,\vect y_{n_M}\}\subset\{\vect y_1,\ldots,\vect y_K\}$. 
For a study of off-grid scatterers we refer to \cite{FSY13,Borcea15}.

To write the data received on the array in a compact form, we define the Green's function vector   
\begin{equation}\label{eq:GreenFuncVec}
\vect \wg(\vect y;\omega)=[\wG(\vect x_{1},\vect y;\omega), \wG(\vect x_{2},\vect y;\omega),\ldots,
\wG(\vect x_{N},\vect y;\omega)]^T\,
\end{equation}
at location $\vect y$ in the IW, where $\wG(\vect x,\vect y;\omega)$ denotes the free-space Green's function of
the (homogeneous or inhomogeneous) medium that characterizes the propagation of a 
signal of angular frequency $\omega$ from point $\vect y$ to point $\vect x$.
When the medium is homogeneous,
\begin{equation}\label{greenfunc}
\wG(\vect x,\vect y;\omega)=\wG_0(\vect x,\vect y;\omega)=\frac{\exp(\rmi\kappa|\vect x-\vect y|)}{4\pi|\vect x-\vect y|}\, ,\quad
 \kappa=\frac{\omega}{c_0}\, ,
\end{equation}
and we have the Green's function vector in a homogeneous medium as
\begin{equation*}
    \vect\wg_0(\vect y; \omega)=[\wG_0(\vect x_1,\vect y;\omega),\wG_0(\vect x_2, \vect y;\omega),\ldots,
    \wG_0(\vect x_N,\vect y;\omega)]^T.
\end{equation*}
If $\vect\wf=[\wf_1,\ldots,\wf_N]^T$ is the illumination vector whose entries are the signals sent from the transmitters in the array, then 
$\vect\wg(\vect y; \omega)\vect\wf$ gives the field at position $\vect y$ in a free-space.

We further introduce the $N\times K$ sensing matrix
\begin{equation}\label{eq:sensingmatrix}
\vect{\Gc}(\omega)=[\vect\wg(\vect y_1;\omega)\,\cdots\,\vect\wg(\vect y_K;\omega)]\, 
\end{equation}
that maps a distribution of sources in the IW to the data received on the array.
With this notation, the full response matrix, can be written as 
\begin{equation}\label{eq:responsematrix0}
\vect\wP=\vect\Gc\diag(\bfrho_0)\mathbf{Z}^{-1}(\bfrho_0)\vect\Gc^T=\vect\Gc\diag(\bfrho_0)\vect\Gc_{FL}^T(\vect\rho_0).
\end{equation}
Here, and in all that follows, we drop the dependence of waves and measurements on the frequency $\omega$.
In \eqref{eq:responsematrix0}, $\mathbf{Z}^{-1}(\bfrho_0)$ denotes the inverse of the Foldy-Lax matrix
which depends on the unknown reflectivity vector $\bfrho_0$ (see Appendix~\ref{appendix:FLmodel}). 
To motivate \eqref{eq:responsematrix0}, consider an illumination vector $\vect\wf=[\wf_1,\ldots,\wf_N]^T$ (see Fig. \ref{fig:schematic}).
Then, $\vect\Gc_{FL}^T(\vect\rho_0)\vect\wf$ gives the total field at each grid point of the IW, including multiple scattering between the scatterers and the interaction with the unknown inhomogeneities of the medium. The total field  $\vect\Gc_{FL}^T(\vect\rho_0)\vect\wf$ is reflected by the scatterers on the grid that have reflectivities given by the vector $\bfrho_0$, and then it is
backpropagated to the array by the matrix $\vect\Gc$. All the available information for imaging is contained in the array response matrix \eqref{eq:responsematrix0}. If transmitters and receivers are located at the same positions, then \eqref{eq:responsematrix0}  is symmetric.

For a fixed array configuration, wave propagation is completely described by the full response matrix $\vect\wP$ \eqref{eq:responsematrix0}.
Indeed, the data received on the array due to an illumination vector $\vect\wf$ is given by
\begin{equation}\label{eq:data}
\vect b=\vect\wP\vect\wf. 
\end{equation}

\section{Active array imaging in homogeneous media}\label{sec:homogeneous medium}
In this section, we formulate the inverse problem of active array imaging  when the medium is homogeneous and, therefore, the wavefronts are not distorted.
In this case, the response matrix can be written as
\begin{equation}\label{eq:responsematrix3}
\vect\wP=\vect\Gc_0\diag(\bfrho_0)\mathbf{Z}^{-1}(\bfrho_0)\vect\Gc_0^T=\vect\Gc_0\diag(\bfrho_0)\vect\Gc_{0FL}^T(\vect\rho_0)\,,
\end{equation}
where $\vect{\Gc}_0=[\vect\wg_0(\vect y_1)\,\cdots\,\vect\wg_0(\vect y_K)]$ denotes the sensing matrix in a homogeneous medium, and
$\mathbf{Z}^{-1}(\bfrho_0)$ is the inverse of the Foldy-Lax matrix \eqref{eq:Z} with $\wG(\vect y_i,\vect y_j)=\wG_0(\vect y_i,\vect y_j)$ 
(see Appendix~\ref{appendix:FLmodel}).
The object to be imaged is an ensemble of small but strong scatterers whose mutual interaction cannot be ignored.
To determine their positions and reflectivities we use the collected data  using a single illumination
$\vect\wf$ in subsection \ref{subsec:single illumination}, and using multiple illuminations $\vect\wf^j$, $j=1,\dots,\upsilon$,
in subsections \ref{subsec:multiple arbitrary illuminations} and \ref{subsec:optimal illuminations}.
In signal processing literature, the corresponding problems are called {\it Single Measurement Vector (SMV)} and
{\it Multiple Measurement Vector (MMV)} problems, respectively.

\subsection{Imaging using single illumination}\label{subsec:single illumination}
For a given illumination vector $\vect\wf$, we define the operator $\vect\Ac_{\wf}$ via
\begin{equation*}
    \vect\Ac_{\wf}[\bfrho_0]\bfrho_0 = \vect\wP\vect\wf,
\end{equation*}
which maps the reflectivity vector $\bfrho_0$ to the data \eqref{eq:data}.
From \eqref{eq:responsematrix3} it follows that
\begin{equation}\label{eq:Af}
    \vect\Ac_{\wf}[\bfrho]=[\wg_{\wf}(\vect y_1)\vect\wg_0(\vect y_1)\,\cdots\,\wg_{\wf}(\vect y_K)\vect\wg_0(\vect y_K)],
\end{equation}
where $\wg_{\wf}(\vect y_j)=\vect\wg^T_{FL}(\vect y_j)\vect\wf$, $j=1,\ldots,K$,
are scalars meaning the total field at the grid points $\vect y_j$ due to
the illumination $\vect\wf$, with $\vect\wg_{FL}(\vect y_j)$ being the $j^\mathrm{th}$
column of the matrix $\vect\Gc_{0FL}$. Using this notation, active array imaging with a single illumination amounts to finding
the unknown reflectivity vector $\bfrho_0\in\mC^K$ from the system of $N$ equations
\begin{equation}\label{eq:single}
    \vect\Ac_{\wf}[\bfrho] \bfrho=\vect b.
\end{equation}
In a typical array imaging configuration, the number of transducers $N$ is much smaller than the number of the grid points $K$
in the IW and, hence, \eqref{eq:single} is underdetermined.
Furthermore, due to the multiple scattering, $\wg_{\wf}(\vect y_j)$, $j=1,\ldots,K$, depend on the unknown
reflectivity vector $\bfrho_0$, which makes \eqref{eq:single} nonlinear with respect to $\bfrho$.
Such nonlinearity makes us think that non-iterative inversion is inapplicable to solve \eqref{eq:single}.
However, by rearranging the terms in these equations, we can reformulate the problem to solve for
the locations of scatterers directly, without any iteration, and then to recover the reflectivities
of each scatterer in a second single step, as we explain next.

\subsubsection{Support recovery}
The localization problem is by far much more difficult than the estimation of reflectivities, which is a straightforward inversion if the former is exact. To localize the scatterers without any iteration, we introduce the {\em effective source vector}
\begin{equation}\label{eq:source vector}
    \bfgamma_{\wf} =\diag(\bfrho)\mathbf{Z}^{-1}(\bfrho)\vect\Gc_0^T\vect\wf\, ,
\end{equation}
and seek for its support. Then, according to \eqref{eq:responsematrix3} $\vect\Ac_{\wf}[\bfrho]\bfrho=\vect\Gc_0\bfgamma_{\wf}$, and
\eqref{eq:single} becomes
\begin{equation}\label{eq:linearsystemsingleillum}
    \vect\Gc_0\bfgamma_{\wf}=\vect b\, ,
\end{equation}
which is linear for the new unknown $\bfgamma_{\wf}$. Note that in the formulation \eqref{eq:single}
the operator $\vect\Ac_{\wf}[\bfrho]$ depends on the illumination $\vect\wf$,
whereas in \eqref{eq:linearsystemsingleillum} the unknown $\bfgamma_{\wf}$ is the one which depends on $\vect\wf$.

It is important to emphasize that due to the existence of  multiple scattering, the solution of \eqref{eq:linearsystemsingleillum} may
not give all the support of $\bfrho_0$. This is not a flaw of the new formulation, but an implicit problem of array
imaging when multiple scattering is important. Indeed, it is possible
that one or several scalars $\wg_{\wf}(\vect y_j)$, $j=1,\ldots,K$, are very small or even zero, and hence,
the corresponding scatterers become dark. This is the well-known {\em screening effect} which makes
scatterers undetectable, and that is manifested in our formulation making some components of the effective
source vector $\bfgamma_{\wf}$ arbitrarily small.

Since \eqref{eq:linearsystemsingleillum} is underdetermined and the effective source vector $\bfgamma_{\wf}$
is sparse because $M\ll N\ll K$, we solve the $\ell_1$ minimization problem
\begin{equation}\label{eq:l1singleillum}
    \min\|\bfgamma_{\wf}\|_{\ell_1}\qquad\text{s.t.}\qquad\vect\Gc_0\bfgamma_{\wf}=\vect b
\end{equation}
when data is noiseless. When the data is contaminated by a noise vector $\vect e$ with finite energy, 
we then solve the relaxed problem
\begin{equation}\label{eq:l1singleillumnoise}
    \min\|\bfgamma_{\wf}\|_{\ell_1}\qquad\text{s.t.}\qquad\|\vect\Gc_0\bfgamma_{\wf}-\vect b\|_{\ell_2}<\delta\, ,
\end{equation}
for some given positive constant $\delta$. In \eqref{eq:l1singleillum} and \eqref{eq:l1singleillumnoise},
$\|\bfgamma\|_{\ell_1}=\sum_{i=1}^K|\gamma_i|$ and $\|\bfgamma\|_{\ell_2}=\sqrt{\sum_{i=1}^K|\gamma_i|^2}$.

The following theorem gives conditions under which \eqref{eq:l1singleillum} recovers the positions
and the strengths of the effective sources exactly if the data is noiseless. The proof follows that
in \cite{CMP13,CMP14}.

\begin{theorem}\label{thm0}
Assume that the resolution of the IW is such that
\begin{equation}\label{eq:mutualcoherence}
        \max_{i\neq j}\left|\frac{\vect\wg_0^\ast(\vect y_i)\vect\wg_0(\vect y_j)}{\|\vect\wg_0(\vect y_i)\|_{\ell_2}\|\vect\wg_0(\vect y_j)\|_{\ell_2}}\right|<\epsilon\,.
\end{equation}
If the number of effective sources $M$ is such that $\epsilon\, M < 1/2$,
then $\bfgamma_{0\wf}=\diag(\bfrho_0)\mathbf{Z}^{-1}(\bfrho_0)\vect\Gc_0^T\vect\wf$
is the unique solution to \eqref{eq:l1singleillum} with support fully contained by that of $\bfrho_0$.
\end{theorem}

\begin{remark}
The condition \eqref{eq:mutualcoherence} of the sensing matrix $\vect\Gc_0$ is determined by the
array imaging configuration. It is a measure of how linearly independent the columns of the sensing matrix are and it is related to the array geometry as shown in \cite{CMP14}.
\end{remark}

\begin{remark}
It turns out that to prove the result in Theorem \ref{thm0}, the condition \eqref{eq:mutualcoherence} has to be satisfied only on
the support of $\bfgamma_{0\wf}$. This means that if the distance between the effective sources is
known a priori to be large so \eqref{eq:mutualcoherence} holds for the set of indices corresponding
to its support, the discretization of the IW can be as small as we want. 
\end{remark}

The next theorem provides an important stability result for the problem \eqref{eq:l1singleillumnoise},
the proof of which is given by Theorem~$4.3$ in \cite{CMP14}.

\begin{theorem}\label{theorem:smv}
    For a given array imaging configuration, with resolution condition given by \eqref{eq:mutualcoherence}, the solution $\bfgamma_{\star\wf}$ to \eqref{eq:l1singleillumnoise} satisfies
    \begin{equation}\label{eq:SMVstability}
        \|\bfgamma_{\star\wf}-\bfgamma_{0\wf}\|_{\ell_2}\le\frac{\delta}{\sqrt{1-(M-1)\epsilon}},
    \end{equation}
    provided $\delta\ge\|\vect e\|_{\ell_2}\sqrt{1+\frac{M(1-(M-1)\epsilon)}{(1-2M\epsilon+\epsilon)^2}}$, where $\vect e$ is the noise vector added to the data.
    Moreover, the support of $\bfgamma_{\star\wf}$ is fully contained in that of $\bfgamma_{0\wf}$, and all the components satisfying 
    \begin{equation}\label{eq:SMVbound}
        |(\bfgamma_{0\wf})_j|>\delta/\sqrt{1-(M-1)\epsilon}
    \end{equation}
    are within the support of $\bfgamma_{\star\wf}$.
\end{theorem}

\subsubsection{Reflectivity estimation}
Optimization \eqref{eq:l1singleillum} (or \eqref{eq:l1singleillumnoise}) gives the effective
source vector $\bfgamma_{\wf}$. In a second step, we compute the true reflectivities from the solution of this problem at once. 
Let $\Lambda_\star$ be the support of the recovered solution such that $|\Lambda_\star|=M'\leq M$,
and $\bfgamma_{\wf,M'}$ be the solution vector on $\Lambda_\star$. From
\eqref{eq:source vector}, we have
\begin{equation*}
    \bfgamma_{\wf,M'}=\diag(\mathbf{Z}^{-1}(\bfrho_{M'})\vect\Gc_0^T\vect\wf)\bfrho_{M'}=\diag(\wg_{\wf}(\vect y_{n_1}),\ldots,\wg_{\wf}(\vect y_{n_{M'}}))\bfrho_{M'}\, ,
\end{equation*}
where  $\wg_{\wf}(\vect y_{n_j})=\vect\wg^T_{0FL}(\vect y_{n_j})\vect\wf$.
Note that the scalars $\wg_{\wf}(\vect y_{n_j})$ are the exciting fields at the scatterers'
positions, 
and that the effective sources $\gamma_{n_j}$ are the true reflectivities $\rho_{n_j}$ multiplied by the exciting fields (see Fig. \ref{fig:schematic} (b)). Hence, using \eqref{eq:exciting-field}
we can compute $\wg_{\wf}(\vect y_{n_j})$ explicitly as follows
\begin{equation}\label{eq:invertless}
    \wg_{\wf}(\vect y_{n_j})=\vect\wg_0^T(\vect y_{n_j})\vect\wf+\sum_{k=1,k\neq j}^{M'}\gamma_{k}\wG_0(\vect y_{n_j},\vect y_{n_k}),\quad j=1,\ldots,{M'}.
\end{equation}
Then, the true reflectivities of the scatterers are recovered by
\begin{equation}\label{eq:truereflectivities}
    \rho_{n_j}=\gamma_{n_j}/\wg_{\wf}(\vect y_{n_j}),\quad j=1,\ldots,{M'}.
\end{equation}
When the data contains additive noise, we choose the support $\Lambda_\star$ 
of the solution recovered by \eqref{eq:l1singleillumnoise} such that all the components of $\bfgamma_{\wf,M'}$
satisfy \eqref{eq:SMVbound}.

\subsection{Imaging using multiple arbitrary illuminations}\label{subsec:multiple arbitrary illuminations}
Imaging with a single illumination can be very sensitive to additive noise, especially when the noise level
is high. Moreover, the {\em screening effect} can cause the failure of recovering some of the scatterers.
Note that, for a fixed imaging configuration, the {\em screening effect} depends on the illumination
vector $\vect\wf$ and the amount of noise in the data. When the effective source at $\vect y_j$ is
below the noise level because $\wg_{\wf}(\vect y_j)$ is small, the corresponding scatterer cannot be detected.
This motives us to consider active array imaging with multiple illuminations. In this case, active
array imaging is modeled as a joint sparsity recovery problem, in which we seek for a matrix solution $\vect{\Rho}$ whose
columns share the same support. By increasing the diversity of illuminations, we are able to minimize the
{\em screening effect} and have higher chance of locating all the scatterers more stably.

\subsubsection{Support recovery}
Instead of simply stacking multiple data vectors $\vect b^j$ obtained from different illuminations $\vect\wf^j$,
$j=1,\ldots,\upsilon$, and solving the corresponding augmented linear system as in \S\ref{subsec:single illumination}, we formulate the problem using 
the {\it MMV} approach where the unknown
vectors $\bfgamma^j$ corresponding to each illumination $\vect\wf^j$ are arranged into a matrix.
More precisely, let $\mathbf{B}=[\vect b^1,\ldots,\vect b^\upsilon]$ be the matrix whose columns
are the data vectors generated by all the illuminations, and $\vect{\Rho}=[\bfgamma^1,\ldots,\bfgamma^\upsilon]$
be the unknown matrix whose $j^\mathrm{th}$ column corresponds to the {\em effective source vector}
$\bfgamma^j$ under illumination $\vect\wf^j$. Thus, we formulate the problem of active array imaging with multiple
measurements as solving the matrix-matrix equation
\begin{equation}\label{eq:linearsystemmmv}
    \vect\Gc_0\vect\Rho=\mathbf{B}\,
\end{equation}
for $\vect\Rho$. The sparsity of the matrix variable $\vect\Rho$ is characterized by the number of nonzero
rows. Thus, we define the row-support of a given matrix $\vect\Rho$ by
\begin{equation*}
    \rowsupp(\vect\Rho)=\{i:\,\,\exists\, j\,\,\text{s.t.}\,\, \Rho_{ij}\neq0\}.
\end{equation*}
When the matrix $\vect\Rho$ degenerates
to a column vector, the row-support reduces to the support of that vector.
Similar to the $\ell_1$ norm relaxation used in the SMV problem in \S\ref{subsec:single illumination},
the sparsest solution  using multiple illuminations is given by the solution to the convex problem
\begin{equation}\label{eq:MMV.convex}
    \min\Xi_1(\vect\Rho)\qquad\text{s.t.}\qquad\vect\Gc_0\vect\Rho=\mathbf{B},
\end{equation}
where $\Xi_1$ is a convenient convex relaxation of the size of $\rowsupp(\vect\Rho)$.
As in \cite{CMP14}, we take 
\begin{equation} \label{eq:Jp1}
\Xi_1(\vect\Rho)=\sum_{i=1}^N\|\Rho_{i\cdot}\|_{\ell_2},
\end{equation}
where $\Rho_{i\cdot}$ is the $i^\mathrm{th}$ row of the matrix.
When the data vectors are contaminated with additive noise vectors $\vect{e}^j$,
$j=1,\ldots,\upsilon$, we have the matrix-matrix equation
\begin{equation}\label{eq:linearsystemmmvnoise}
    \vect\Gc_0\vect\Rho = \mathbf{B}+\vect\Ec,
\end{equation}
where $\vect\Ec=[\vect{e}^1\cdots\vect{e}^\upsilon]$, and we seek a solution to
\begin{equation}\label{eq:MMV21noise}
    \min\Xi_1(\vect\Rho)\qquad\text{s.t.}\qquad\|\vect\Gc_0\vect\Rho-\mathbf{B}\|_F<\delta,
\end{equation}
for a pre-specified constant $\delta$, where the Frobenius norm is given by
\[\|\vect\Rho\|_F=\left(\sum_{i=1}^N\sum_{j=1}^K|\Rho_{ij}|^2\right)^{\frac{1}{2}}.\]

\begin{remark}\label{rem:MMV}
    Similar to Theorems~\ref{thm0} and \ref{theorem:smv} for array imaging with a single illumination,
    we have results regarding the exact recovery and stability of the solution to \eqref{eq:MMV.convex}
    and \eqref{eq:MMV21noise} for imaging using multiple illuminations. These results are proved in \cite{CMP14}.
    We note, however, that these results do not provide a quantitative improvement when the number of measurements
    $\upsilon$ increases. Intuitively, this lack of improvement is justifiable since the measurements obtained from
    randomly chosen illuminations could be rather ineffective. There is no guarantee that randomly picked illuminations
    bring more information useful for imaging. In practive, however, we do observe a general improvement in the images
    formed with multiple random illuminations. In \S\ref{subsec:optimal illuminations}, we use selective illuminations,
    obtained from the SVD of the array response matrix $\vect\wP$ to increase the efficiency of the MMV approach.
\end{remark}

\subsubsection{Reflectivity estimation} 

Once we obtain the matrix $\vect\Rho_\star$ from \eqref{eq:MMV.convex} (or \eqref{eq:MMV21noise}),
whose columns are the effective sources corresponding to the different illuminations, we compute
in a second step the true reflectivities as follows. For each component $i$ in the
support such that
the stability condition in Theorem~$4.3$ in \cite{CMP14} is satisfied, we compute the reflectivities $\rho^j_{i}$
corresponding to each illumination $j$ by applying \eqref{eq:invertless} and
\eqref{eq:truereflectivities}. Then, we take the average $\frac{1}{\upsilon}\sum_{j=1}^\upsilon\rho_{i}^j$
as the estimated reflectivity.

\subsection{Imaging using optimal illuminations}\label{subsec:optimal illuminations}

To further improve the efficiency, and especially the robustness of \eqref{eq:MMV21noise}, we propose 
to use selective illuminations as multiple illumination vectors  \cite{CMP14}.
The optimal set of illumination vectors can be obtained systematically from the SVD of the full response matrix $\vect\wP$. If the full array response matrix is not available, they can be obtained from an iterative time
reversal process as discussed in \cite{CMP13}.
Let the SVD of $\vect\wP$ be
\begin{equation*}
    \vect\wP=\vect\wU\vect\Sigma\vect\wV^\ast=\sum_{j=1}^{\tilde M}\sigma_j\wU_{\cdot j}\wV_{\cdot j}^\ast,
\end{equation*}
where $\wU_{\cdot j}$ and $\wV_{\cdot j}$ are the left and right singular vectors, respectively,
and the nonzero singular values $\sigma_j$ are given in descending order as
$\sigma_1\ge\sigma_2\ge\cdots\ge\sigma_{\tilde M}>0$, with $\tilde M\ge M$.
When there is no  noise in the data, $\tilde M=M$. Now, let the illumination vectors
be the top right singular vectors, i.e. $\vect\wf^j=\wV_{\cdot j}$, $j=1,\ldots,\upsilon\le\tilde M$. Then, we have
\begin{equation}\label{eq:linearsystemmmv optimal}
    \vect\Gc_0\vect\Rho_V=\mathbf{B}_{opt}=\vect\wP\vect\wV_{\cdot,1:\upsilon}=[\sigma_1\wU_{\cdot 1}\cdots\sigma_\upsilon\wU_{\cdot\upsilon}]+\vect{\widetilde{\Ec}}\, ,
\end{equation}
where the data matrix $\mathbf{B}_{opt}$ contains all the essential information for imaging the scatterers.
The improvement of the efficiency, when using optimal illuminations, comes from the fact that
illuminations using singular vectors deliver most of the energy around the scatterers,
even when multiple scattering is non-negligible. Therefore, taking a few top singular vectors
is enough to focus around the scatterers that contribute to the data received on the imaging array.

\subsection{Numerical experiments} \label{sec:numerical simulation}
We now present numerical simulations in two dimensions. 
In all the simulations shown below with a single illumination, we use the iterative shrinkage-thresholding algorithm GelMa
proposed in \cite{Moscoso12} due to its flexibility with respect to the choice of the regularization parameter
used in the algorithm. In the simulations with multiple illuminations we use an extension of this algorithm described 
in \cite{CMP14}.

Figure~\ref{original_mm} shows five scatterers placed within an IW  of size 
$41\lambda\times41\lambda$, which is at a distance $L=100\lambda$ from a linear array with $100$ transducers that are one wavelength $\lambda$ apart
(the spatial units of all the images is $\lambda$).
The amplitudes of the reflectivities of the scatterers $|\alpha_j|$ are $2.96$, $2.76$, $2.05$, $1.54$ and $1.35$.  
Their phases are set randomly in each realization. Note that for a  given illumination $\vect\wf$ and a scatterer configuration $\vect\rho_0$ with fixed amplitudes,
the amount of multiple scattering 
depends on the realization of the phases in $\vect\rho_0$. For the amplitudes of the reflectivities chosen here, the amount of multiple scattering over single 
scattering ranges  between $50\%$  and $100\%$. 

\begin{figure}[htbp]
\begin{center}
\includegraphics[scale=0.25]{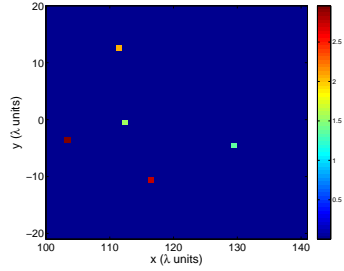}
\caption{Original configuration of the scatterers in a $41\times41$ image window with grid points separated by $1$.}
\label{original_mm}
\end{center}
\end{figure}

Figure~\ref{SVM} shows the images obtained by $\ell_1$ norm minimization with $0\%$ (left), $10\%$ (middle), and $20\%$ noise (right)
when a single illumination probes the medium. When there is no noise in the data, $\ell_1$ norm minimization recovers the positions and
reflectivities of the scatterers exactly (the true locations are indicated with small white dots).
However, when the data is corrupted by $10\%$ and $20\%$ of additive noise (middle and right images of Figure~\ref{SVM}),
$\ell_1$ norm minimization fails. Some ghosts appear in the images and some scatterers are missing because of {\it screening effect}.

\begin{figure}
\centering
\begin{tabular}{ccc}
\includegraphics[scale=0.25]{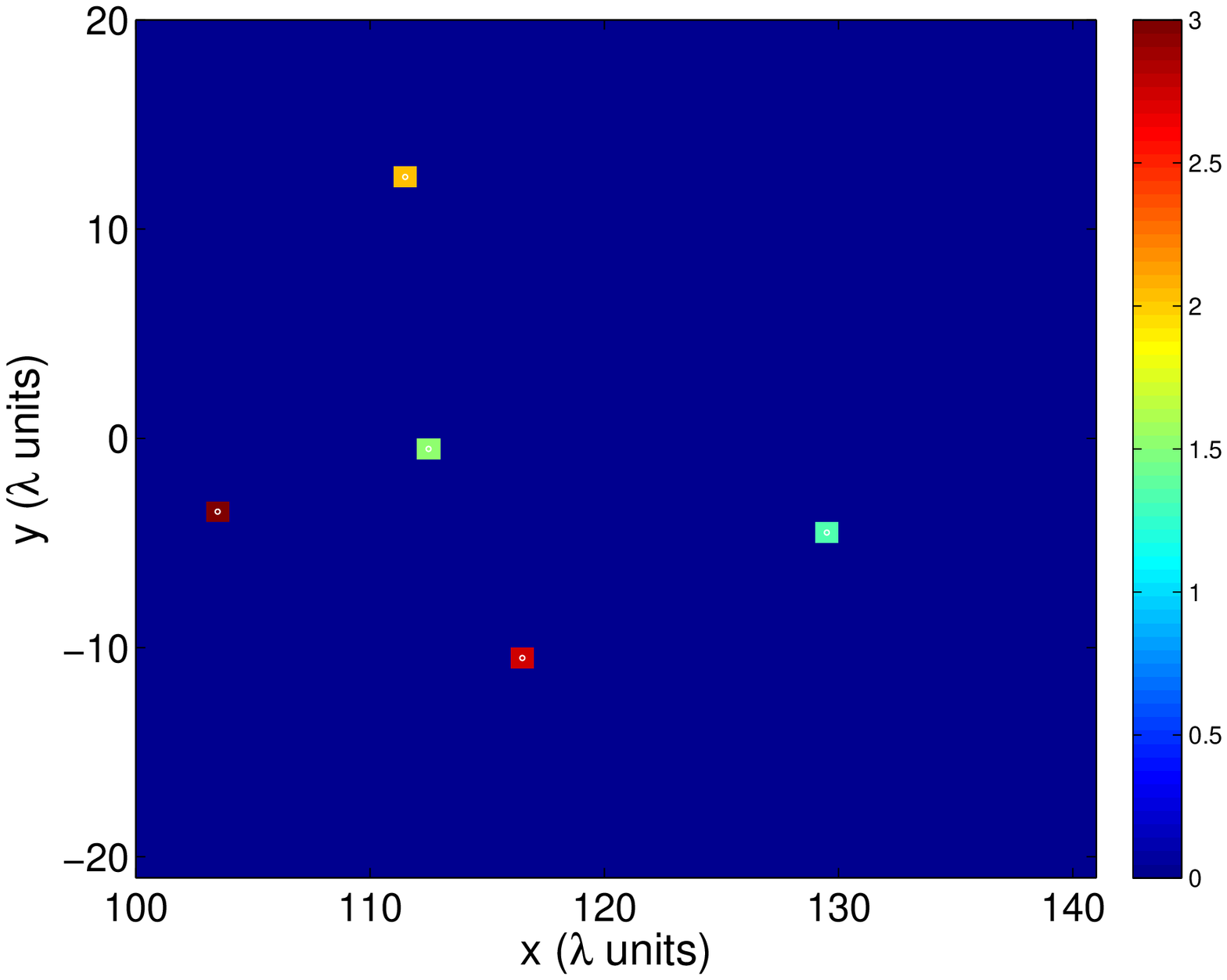} & 
\includegraphics[scale=0.25]{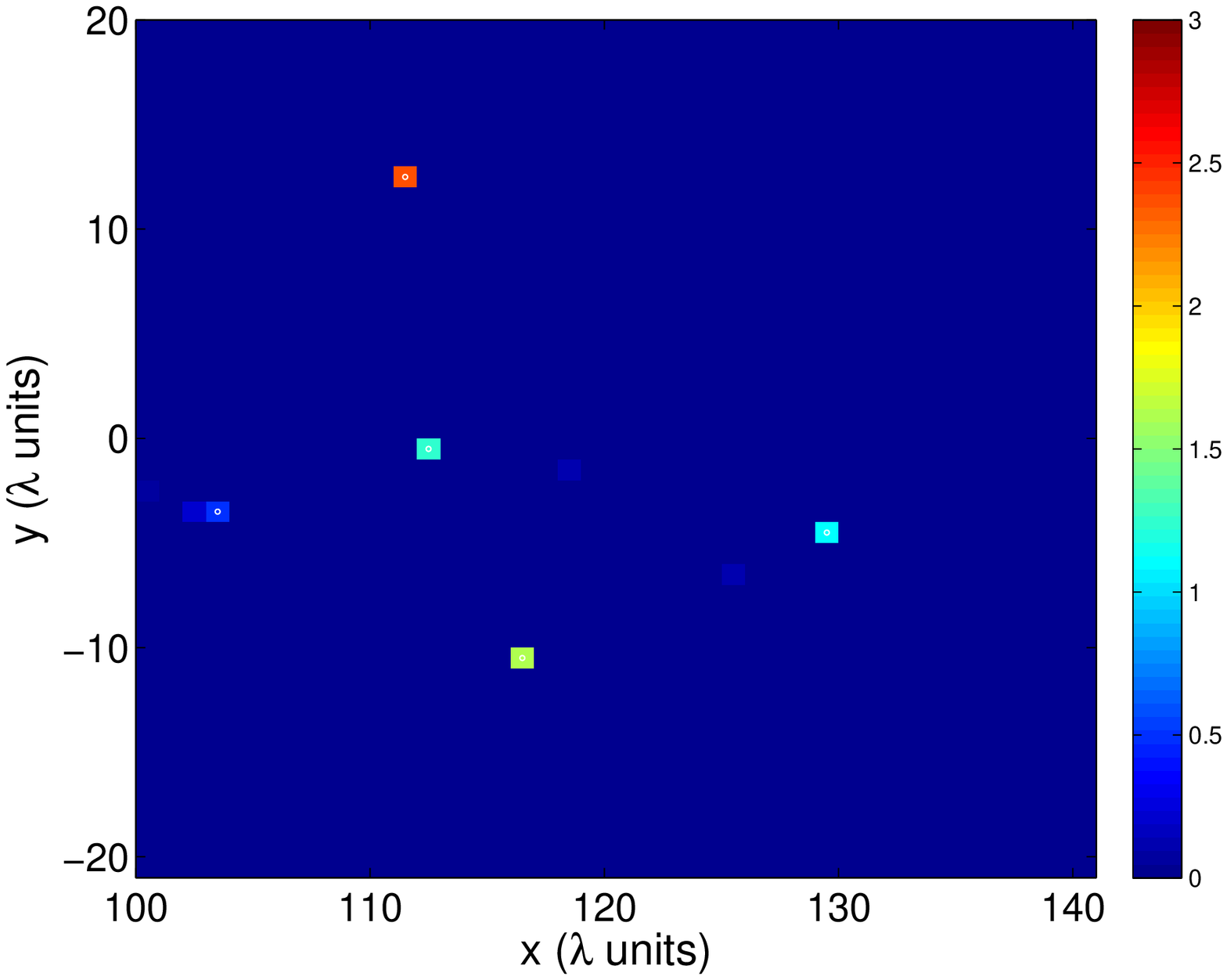} & 
\includegraphics[scale=0.25]{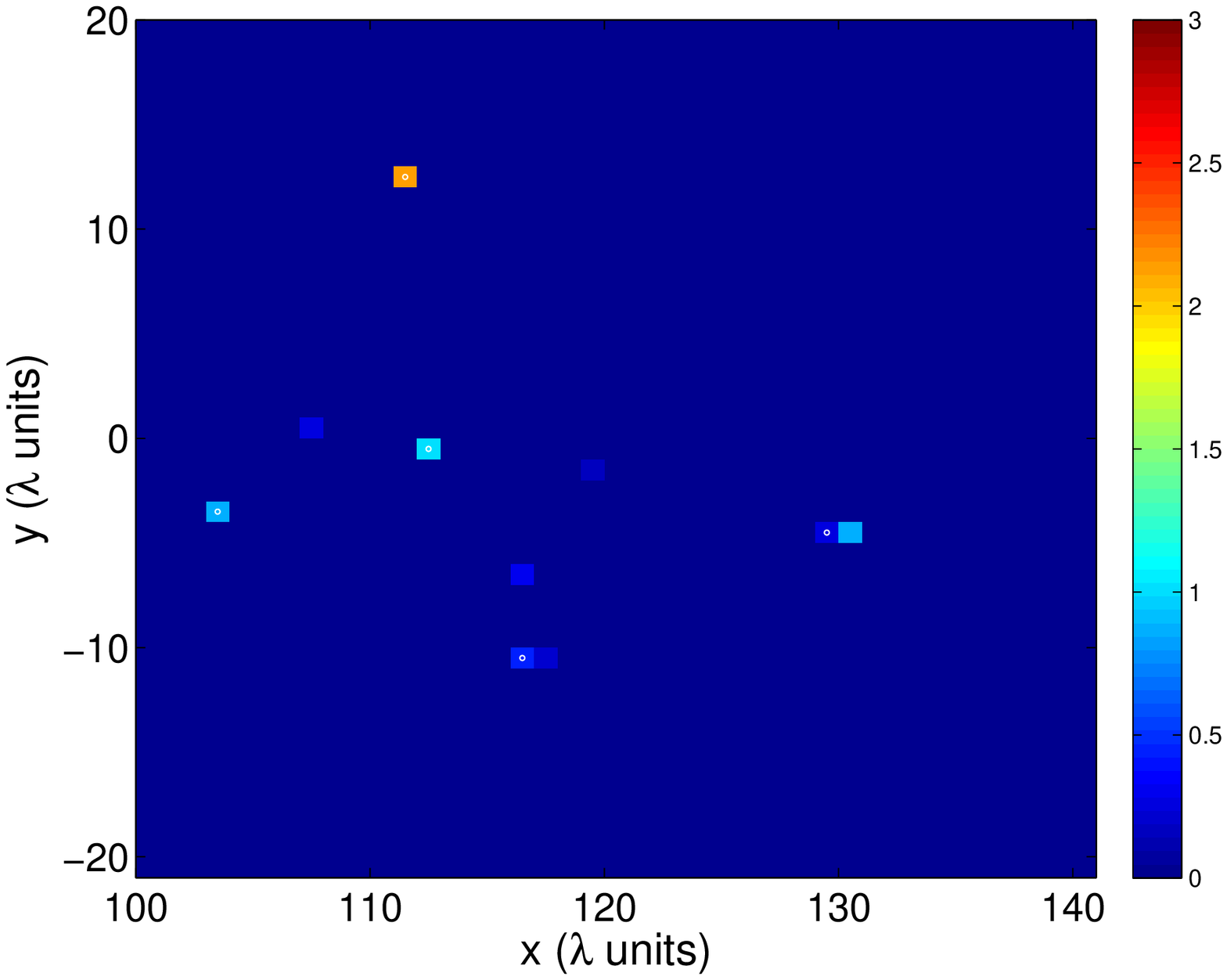}
\end{tabular}
\caption{Images reconstructed by solving \eqref{eq:l1singleillum} and \eqref{eq:l1singleillumnoise} when single illumination is used.
From left to right, there is $0\%$, $10\%$, and $20\%$ additive noise in the data.}
\label{SVM}
\end{figure}

When multiple illuminations are used we expect a general improvement in the images. 
Figure~\ref{MMV} shows the results of the MMV algorithm when $5$ random illuminations are used.
By multiple random illuminations we mean a set of illuminations coming, each one,
from only one of the transducers on the array at a time, i.e., $\wf_{p}=1$ and  $\wf_q=0$ for $q\neq p$,
with $p$ chosen randomly at a time. The set of illuminations used in each image of Figure~\ref{MMV} is different.
The data contain $10\%$ (left), $20\%$ (middle) and $50\%$ (right) of additive noise.
We observe that the images obtained with multiple illuminations are more stable with respect to additive noise. 
However, we still see some missing scatterers in the right image ($50\%$ of noise), and therefore,
for this particular set of randomly chosen illuminations, imaging is still affected by {\it screening effect}.

\begin{figure}
\centering
\begin{tabular}{ccc}
\includegraphics[scale=0.25]{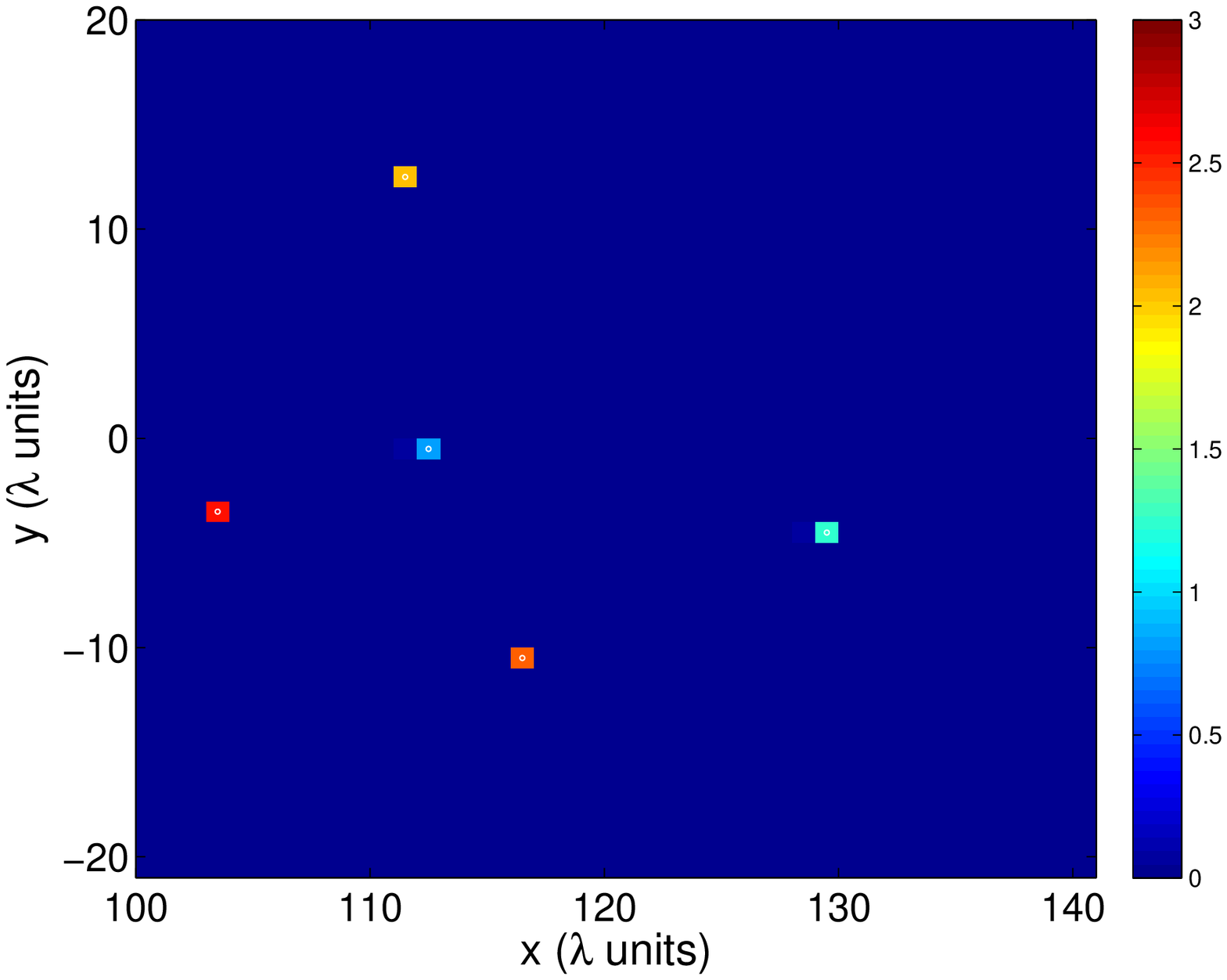} & 
\includegraphics[scale=0.25]{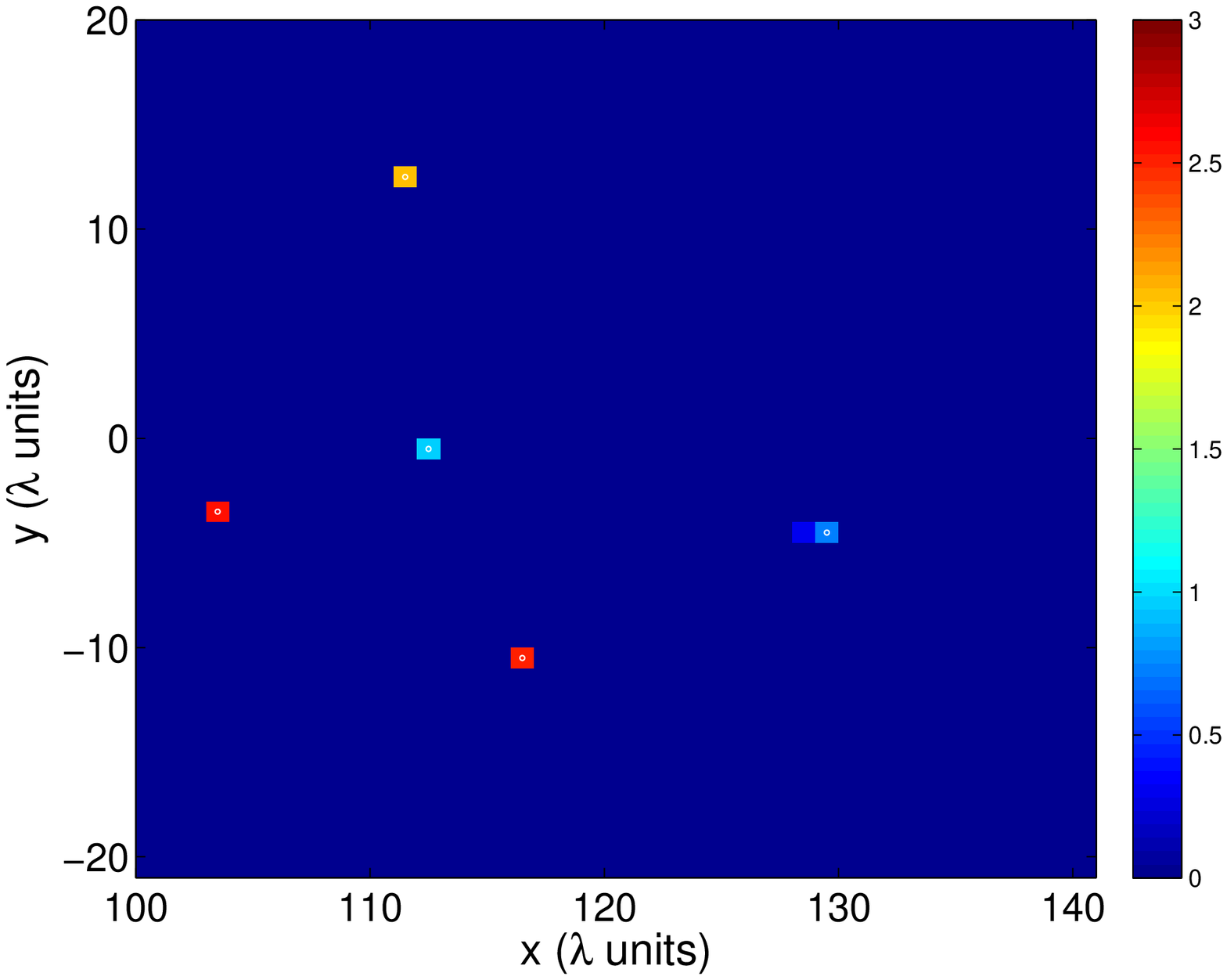} & 
\includegraphics[scale=0.25]{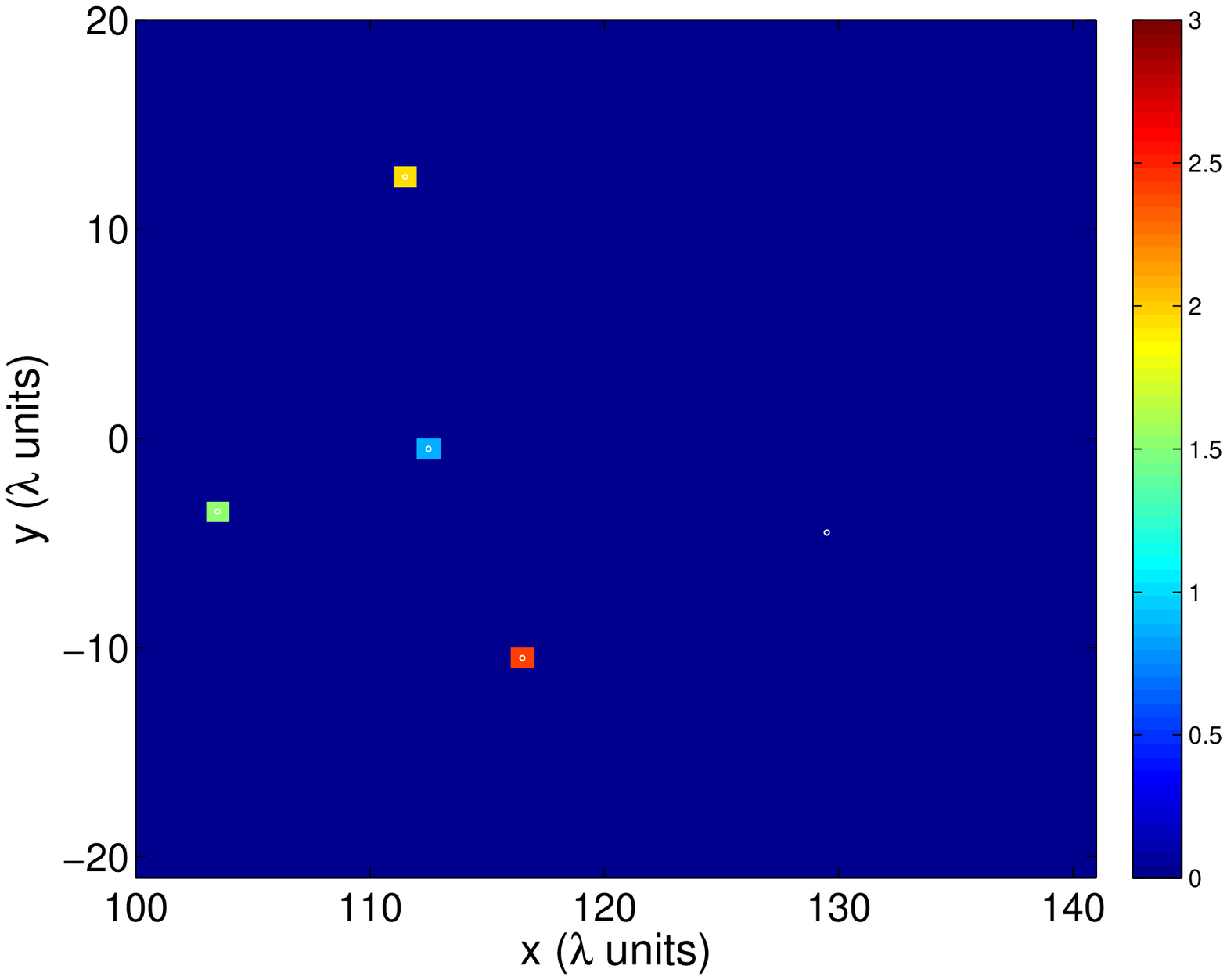}
\end{tabular}
\caption{Images reconstructed by solving \eqref{eq:MMV21noise} when $5$ random
illuminations are used. From left to right, there is  $10\%$, $20\%$, and $50\%$ noise in the data.
}
\label{MMV}
\end{figure}

Figure~\ref{MMV} shows the necessity of using ``good" illuminations that maximize the information content of the
data,  especially when the signal-to-noise ratio (SNR) is low. Figure~\ref{MMVoptimal} displays the images 
obtained with  $1$ (left), $2$ (middle), and $3$ (right), optimal
illuminations associated to the singular vectors $\vect\wf^{j}=\wV_{\cdot j}$, with $j=1,2,3$.
It is remarkable that only a few of them ($2$ or $3$) are enough to obtain very good images, even with $50\%$ of additive noise in the
data.

\begin{figure}
\centering
\begin{tabular}{ccc}
\includegraphics[scale=0.25]{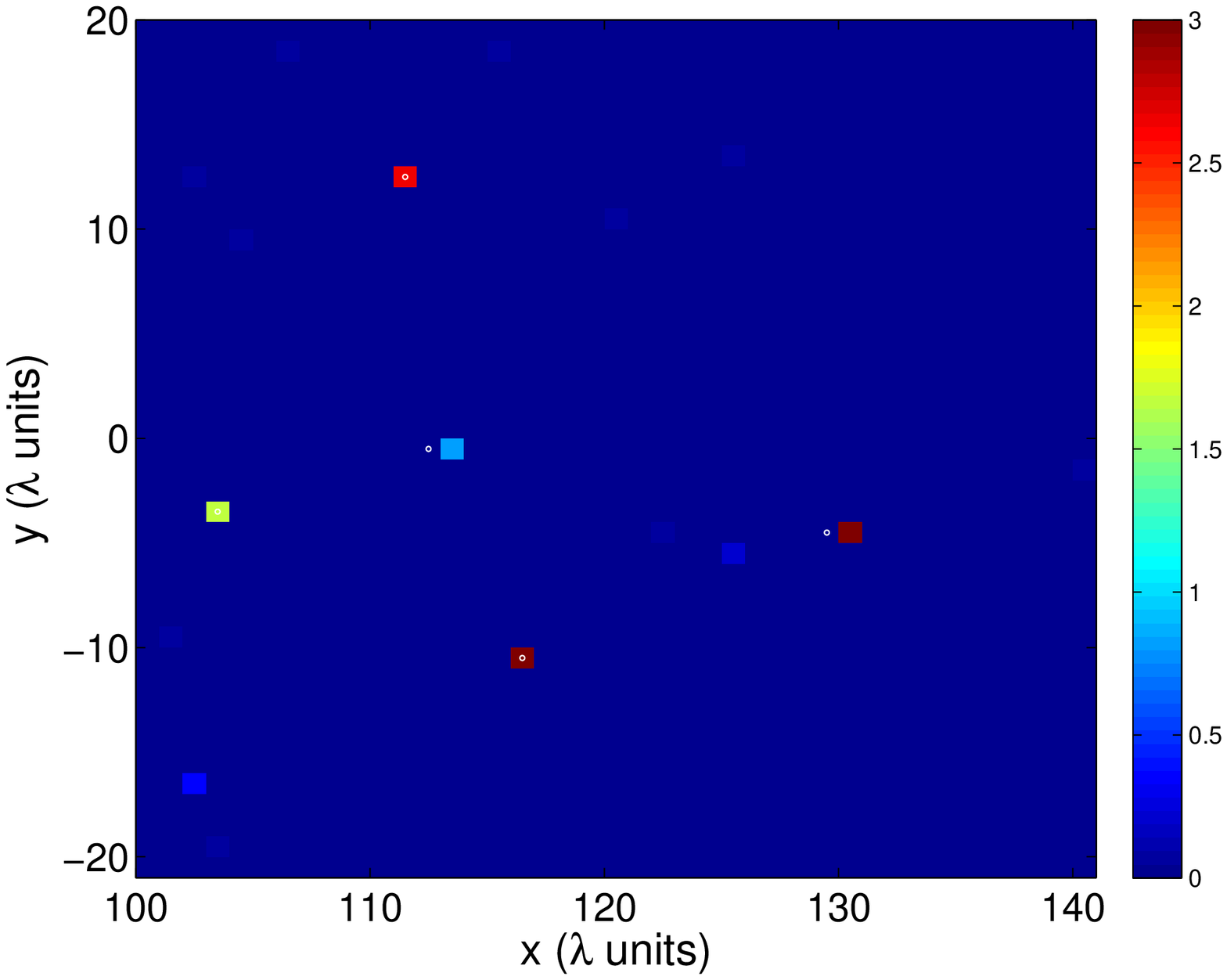} & 
\includegraphics[scale=0.25]{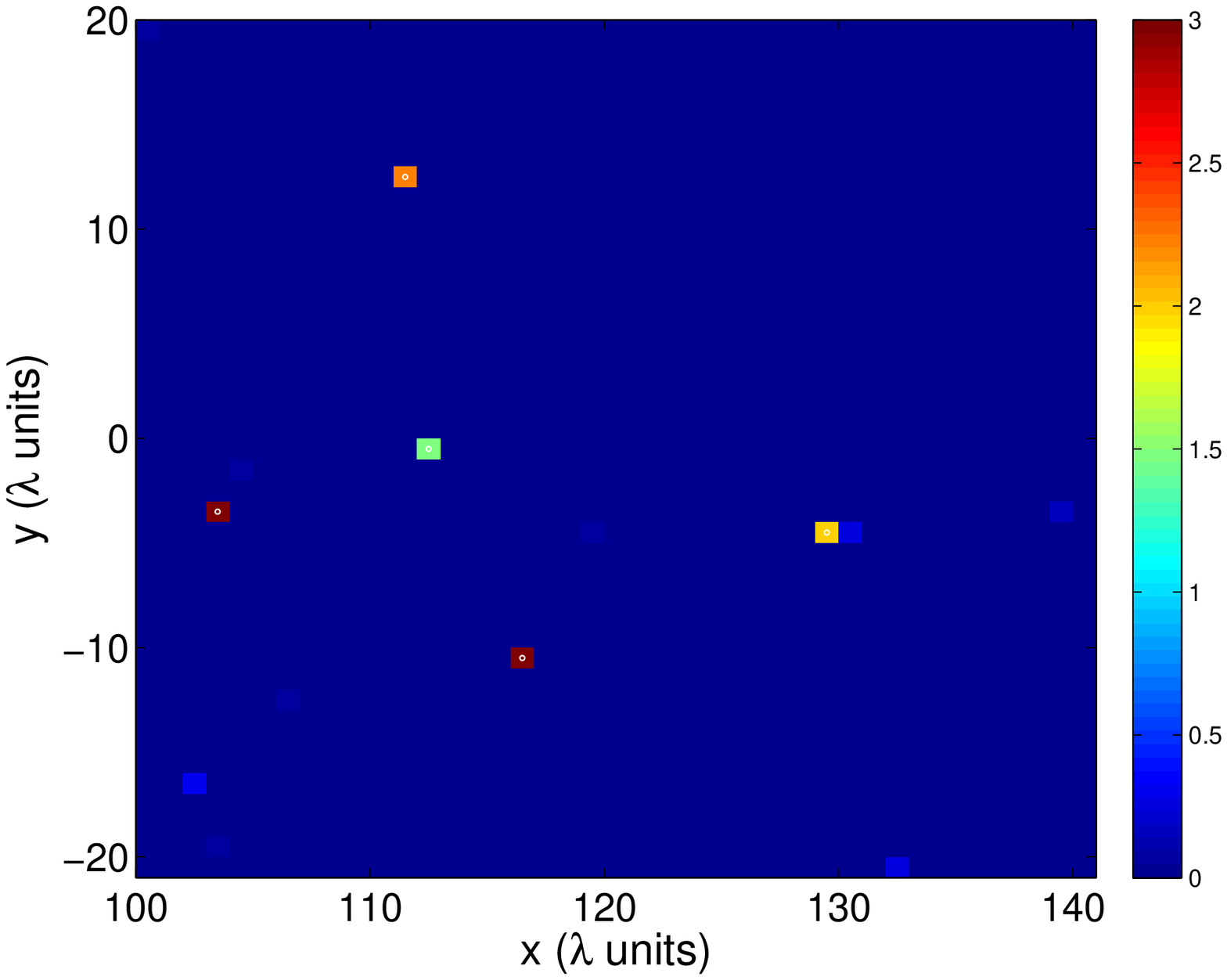} & 
\includegraphics[scale=0.25]{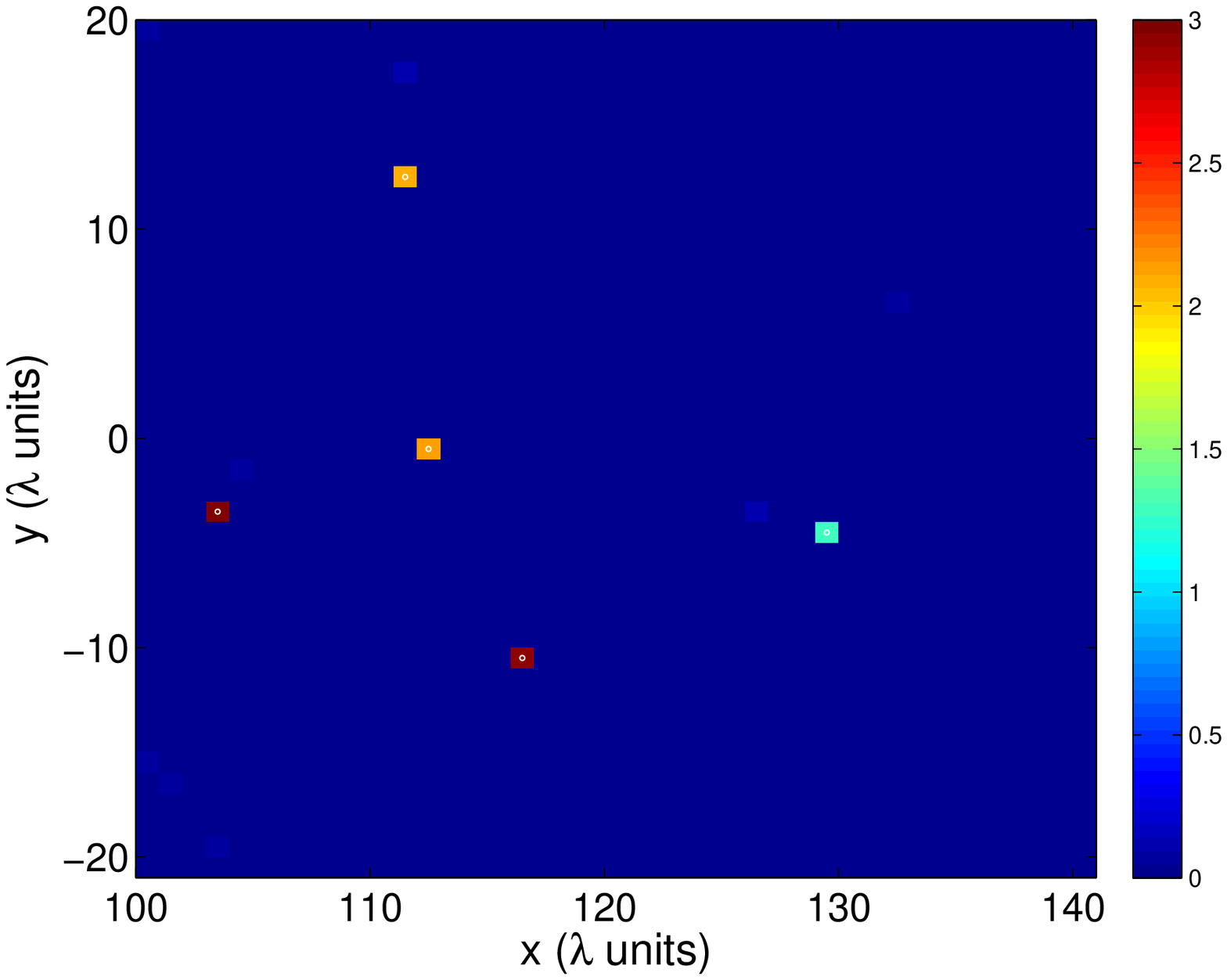}
\end{tabular}
\caption{Images reconstructed by solving \eqref{eq:MMV21noise} when optimal illuminations are used.
From left to right, the images are reconstructed by using $1$, $2$, and $3$ top singular vectors as illumination vectors.
There is $50\%$ noise in the data.
}
\label{MMVoptimal}
\end{figure}

\section{Single scattering as a special case: The hybrid-$\ell_1$ method}\label{sec:single scattering}
In situations where the scatterers are weak and/or are very far apart, 
the incident wave undergoes only one scattering event before coming back to the array.
In these cases, the inverse of the Foldy-Lax matrix $\mathbf{Z}^{-1}(\bfrho)$
in \eqref{eq:responsematrix3} becomes the identity matrix, and the array response
matrix under the so called Born approximation becomes
$$
\widehat{\bfPi}=\vect\Gc_0\diag(\bfrho_0)\vect\Gc_0^T\, .
$$
Then, the data received on the array when an illumination vector $\vect\wf$ probes the medium
is given by $\vect b=\widehat{\bfPi}\vect\wf$. Similar to \eqref{eq:single}, we can model the
data $\vect b$ as applying the operator $\vect\Ac_{\wf}$ to the unknown reflectivity vector
$\bfrho_0$. Because multiple scattering is negligible, $\wg_{\wf}(\vect y_j)=\vect\wg^T_0(\vect y_j)\vect\wf$ in \eqref{eq:Af},
and thus, \eqref{eq:single} is now linear in $\bfrho_0$, which we denote as $\vect\Ac_{0\wf}$.
Therefore, array imaging of localized scatterers in homogeneous media when multiple scattering is negligible amounts to solving 
the linear optimization problem 
\begin{equation}\label{eq:smv born}
    \min\|\bfrho\|_{\ell_1}\qquad\text{s.t.}\qquad\vect\Ac_{0\wf}\bfrho=\vect b.
\end{equation}
In \eqref{eq:smv born} we assume that the data is noiseless. If the data contains noise, we solve
\begin{equation}\label{eq:smv born noise}
        \min\|\bfrho\|_{\ell_1}\qquad\text{s.t.}\qquad\|\vect\Ac_{0\wf}\bfrho-\vect b\|_{\ell_2}<\delta\, .
\end{equation}

Because of the linearity of the equations, a single-step process is sufficient to recover both
the locations and reflectivities of the scatterers from \eqref{eq:smv born} or \eqref{eq:smv born noise}. Moreover, when
multiple scattering is negligible, there is no {\em screening effect}, and hence, 
solving the set of equations generated from a single illumination alone can produce reliable images.
Note that Theorems~\ref{thm0} and \ref{theorem:smv} can be applied to problems \eqref{eq:smv born} and 
\eqref{eq:smv born noise} as well.

Similarly to the case where multiple scattering was important, we can use multiple illuminations and the MMV approach to find 
the reflectivity vector $\bfrho_0$. Furthermore, imaging using optimal illuminations
has a simpler form and weaker requirements on the imaging configuration when multiple scattering is negligible. 
Indeed, assume that the data is noiseless, and
let the SVD of array response matrix be $\widehat{\bfPi}=\vect\wU\vect\Sigma\vect\wV^\ast$,
with $M$ singular values $\{\sigma_j\}_{j=1}^M$ greater than $0$. Following the MMV formulation in
\S\ref{subsec:multiple arbitrary illuminations},
we have $\vect b^j=\widehat{\bfPi}\wV_{\cdot j}=\sigma_j\wU_{\cdot j}$ for $j=1,\ldots,M$. On the other hand,
$\vect b^j=\Ac_{0\wV_{\cdot j}}\bfrho_0=\vect\Gc_0\bfgamma_0^j$, where $\bfgamma_0^j=\diag(\wg_{\wV_{\cdot j}}(\vect y_k))\bfrho_0$
is linear in $\bfrho_0$, which is the major difference from \eqref{eq:source vector}. 
Now consider \eqref{eq:linearsystemmmv optimal} with $\vect{\widetilde{\Ec}}=\vect{0}$ for symplicity, and 
project $\vect B_{opt}$ onto the space spanned by the top left singular vectors corresponding to the significant 
singular values. Denoting $\vect{\mathcal{U}}=[\wU_{\cdot 1},\ldots,\wU_{\cdot M}]$ and $\vect\Rho_V=[\bfgamma^1,\ldots,\bfgamma^M]$, we get
\begin{eqnarray}\label{eq:hybrid1}
    \vect{\mathcal{U}}^\ast\vect\Gc_0\vect\Rho_V&=&\left[
    \begin{array}{cccc}
        \wU_{\cdot 1}^\ast\vect\Gc_0\bfgamma^1 & \wU_{\cdot 1}^\ast\vect\Gc_0\bfgamma^2 & \cdots & \wU_{\cdot 1}^\ast\vect\Gc_0\bfgamma^M\\
        \wU_{\cdot 2}^\ast\vect\Gc_0\bfgamma^1 & \wU_{\cdot 2}^\ast\vect\Gc_0\bfgamma^2 & \cdots & \wU_{\cdot 2}^\ast\vect\Gc_0\bfgamma^M\\
        \vdots & \vdots & \vdots & \vdots \\
        \wU_{\cdot M}^\ast\vect\Gc_0\bfgamma^1 & \wU_{\cdot M}^\ast\vect\Gc_0\bfgamma^2 & \cdots & \wU_{\cdot M}^\ast\vect\Gc_0\bfgamma^M
\end{array}\right]\nonumber\\
&=&\vect{\mathcal{U}}^\ast\vect B_{opt}=\diag(\sigma_1,\ldots,\sigma_M).
\end{eqnarray}
In \eqref{eq:hybrid1}, there are $M$ unknowns $\bfgamma^1,\ldots,\bfgamma^M$, and $M^2$ equations.
However, when multiple scattering is negligible we can associate
a nonzero singular value to each scatterer and the corresponding singular vector is given by
\begin{equation}\label{eq:svd}
    \wU_{\cdot j}=\overline{\wV_{\cdot j}}\propto\frac{\vect\wg_0(\vect y_j)}{\|\vect\wg_0(\vect y_j)\|_{\ell_2}},\qquad j=1,\ldots,M,
\end{equation}
up to an arbitrary complex constant of modulus $1$. Therefore, the rank-$1$ matrices $\wV_{\cdot j}\wU_{\cdot i}^\ast$ 
have a (non-trivial) eigenvalue different from zero only when $i=j$.
That is, only the diagonal blocks of the matrix $\vect{\mathcal{U}}^\ast\vect\Gc_0\vect\Rho_V$ contribute in \eqref{eq:hybrid1},
and hence, we can simplify \eqref{eq:hybrid1} further as a set of $M$ equations.
Using \eqref{eq:svd} and
\begin{equation*}
    \wU_{\cdot i}^\ast\vect\Gc_0\bfgamma^j=\sum_{k=1}^N\big(\wU_{\cdot i}^\ast\vect\wg_0(\vect y_k)\big)\big(\vect\wg_0(\vect y_k)^\ast\wV_{\cdot j}\big)\rho_{0k}
    =\sum_{k=1}^N\vect\wg_0^\ast(\vect y_k)\big(\wV_{\cdot j}\wU_{\cdot i}^\ast\big)\vect\wg_0(\vect y_k)\rho_{0k} \, ,
\end{equation*}
we can write \eqref{eq:hybrid1} as a set of $M$ equations linear in $\bfrho_0$ such that
\begin{equation}\label{eq:hybrid2}
    \vect\Bc\bfrho_0=[\sigma_1,\ldots,\sigma_M]^T\, ,
\end{equation}
where
\begin{equation*}
    \vect\Bc=\left[\begin{array}{ccc}
            \overline{\vect\wg_0^\ast(\vect y_1)\wU_{\cdot 1}}\vect\wg_0^T(\vect y_1)\wV_{\cdot 1} & \cdots & \overline{\vect\wg_0^\ast(\vect y_K)\wU_{\cdot 1}}\vect\wg_0^T(\vect y_K)\wV_{\cdot 1}\\
            \cdots & \ddots & \cdots \\
            \overline{\vect\wg_0^\ast(\vect y_1)\wU_{\cdot M}}\vect\wg_0^T(\vect y_1)\wV_{\cdot M} & \cdots & \overline{\vect\wg_0^\ast(\vect y_K)\wU_{\cdot M}}\vect\wg_0^T(\vect y_K)\wV_{\cdot M}\\
    \end{array}\right].
\end{equation*}
Therefore, the hybrid-$\ell_1$ method is to solve for $\bfrho_0$ from the SMV problem
\begin{equation}\label{eq:hybrid l1}
    \min\|\bfrho\|_{\ell_1}\qquad\text{s.t.}\qquad\vect\Bc\bfrho=[\sigma_1,\ldots,\sigma_M]^T \, .
\end{equation}
The advantage of hybrid-$\ell_1$ method compared to the MMV formulation is its simpler form
and lower dimensionality.
Moreover, the hybrid $\ell_1$ method guarantees exact recovery, the sufficient condition for which also implies that
it can form images with higher resolution (see \cite{CMP13} for details).
\begin{theorem}\label{theorem:hybrid l1}
    Assume that the scatterers are far part such that \eqref{eq:svd} is satisfied.
    Let $\vect E$ be the submatrix of $\vect\Bc$ formed by normalized columns corresponding to the scatterers' location and
    $\vect S$ be the submatrix formed by the remaining normalized columns of $\vect\Bc$.
    If $\|\vect S\|_{1\rightarrow1}<1-\|\vect E-\vect I\|_{1\rightarrow1}$,
    where $\|\cdot\|_{1\rightarrow1}$ is the matrix $1$-norm induced by the vector $\ell_1$ norm,
    and $\vect I$ is the identity matrix, then $\bfrho_0$ is the unique solution to \eqref{eq:hybrid l1}.
\end{theorem}

\subsection{Numerical experiments} \label{sec:numerical simulation single}
Now we present the results of a numerical experiment that shows the performance of the hybrid-$\ell_1$ method compared to those 
obtained with KM and MUSIC. We consider the same imaging set-up as in subsection \ref{sec:numerical simulation}. KM is a simple and robust imaging method (with respect to additive noise)
which propagates the imaging data received on the array back to the medium. Under the setup of imaging in \S\ref{sec:formulation},
the KM imaging functional can be written as
\begin{equation}\label{eq:KM}
    \bfrho_{KM}=\vect\Ac_{0\wf}^\ast\vect b.
\end{equation}
MUSIC, on the other hand, is a subspace projection algorithm that uses the SVD of the array response matrix.
The idea of MUSIC is to project the reference illumination vectors $\vect\wg_0(\vect y_j)$, $j=1,\dots,M$,
onto the noise space using the projection operator
\begin{equation}\label{eq:music projection}
    \mathcal{P}\vect\wg_0(\vect y)=\sum_{j=1}^M(\wU^\ast_{\cdot j}\vect\wg_0(\vect y))\wU_{\cdot j}-\vect\wg_0(\vect y).
\end{equation}
Then, according to \eqref{eq:svd}, the normalized functional
\begin{equation}\label{eq:music tr}
    \mathcal{I}_{MUSIC}(\vect y_k)=\frac{\min_{1\le j\le K}\|\mathcal{P}\vect\wg_0(\vect y_j)\|_{\ell_2}}{\|\mathcal{P}\vect\wg_0(\vect y_k)\|_{\ell_2}},\quad k=1,\ldots,K,
\end{equation}
will peak at the search points $\vect y_k$ only at locations where there is a scatterer.  Imaging with MUSIC
\eqref{eq:music tr} only  locates of the scatterer's positions. Their reflectivities are usually obtained using a 
separate procedure after the locations are known.

Figure~\ref{multiIllum2D} shows the simulation results with $100\%$ of additive noise in the data. From left to right, we display the images
obatined with KM, MUSIC and the hybrid-$\ell_1$ method. In the simulations for MUSIC and hybrid-$\ell_1$ we assume that
the top $5$ singular vectors have been obtained. In the simulations for KM we use $5$ random illuminations. 
It is apparent that the image created by the hybrid-$\ell_1$ method is better than the other two. 
The hybrid-$\ell_1$ method forms a clear image with  perfect recovery of the location of the scatterers,  and 
accurate estimates of their reflectivities. The robustness of the hybrid-$\ell_1$ method relies on the
selection of the top  singular values so that the noise in the
subspace formed by singular vectors associated with small singular values is filtered out.

\begin{figure}[htbp]
\begin{center}
\begin{tabular}{ccc}
\includegraphics[scale=0.2]{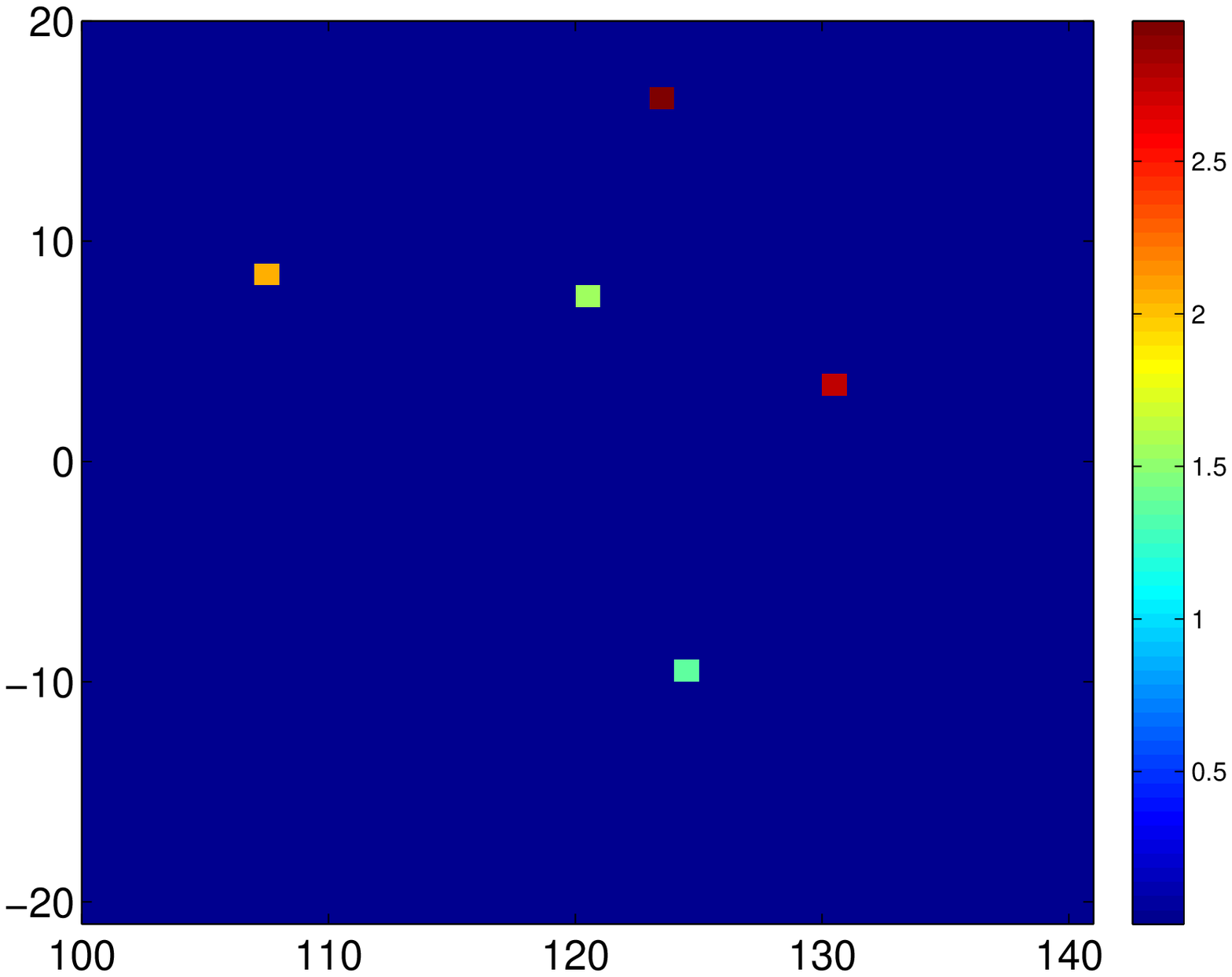}&
\includegraphics[scale=0.2]{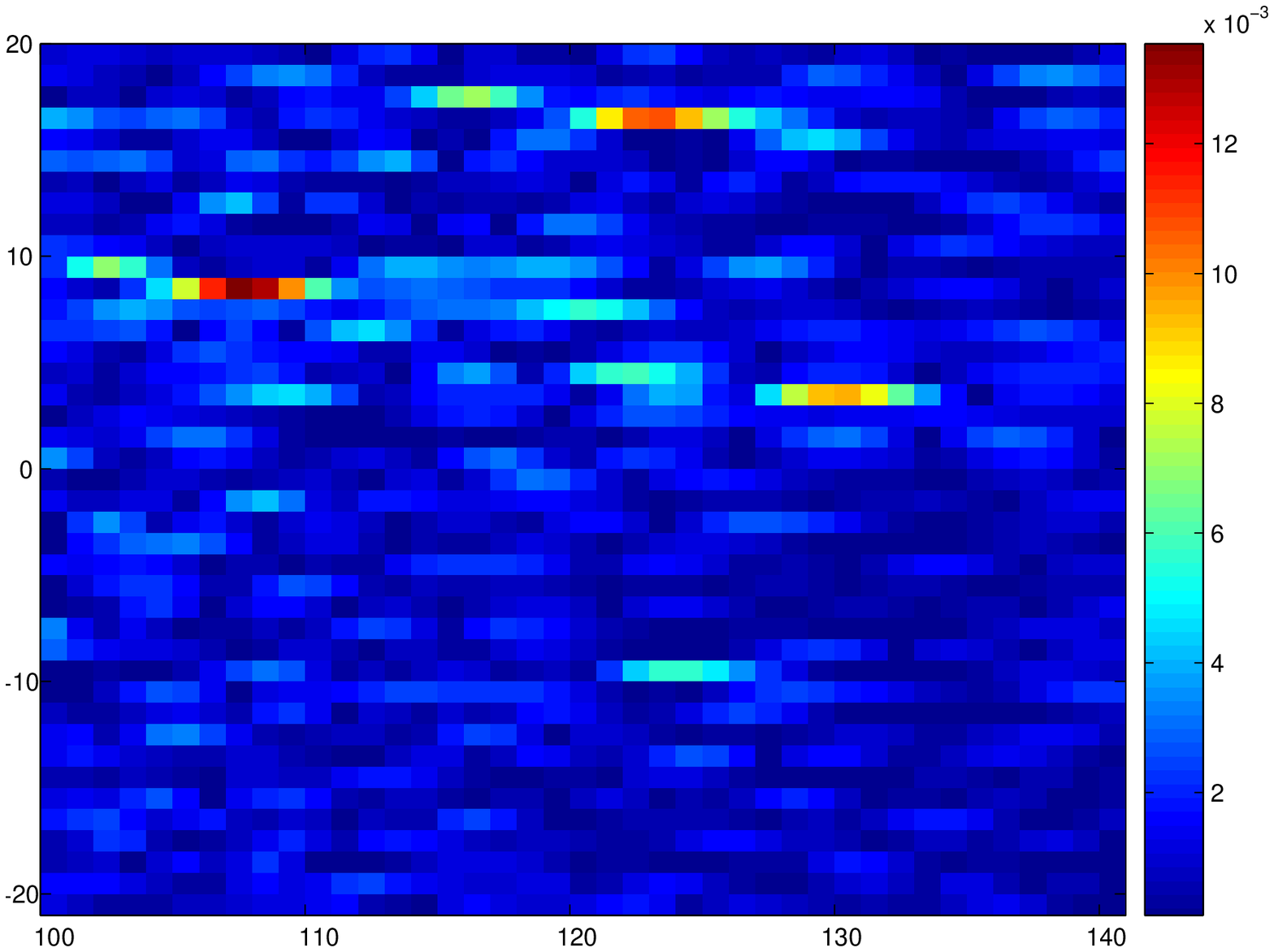}&
\includegraphics[scale=0.2]{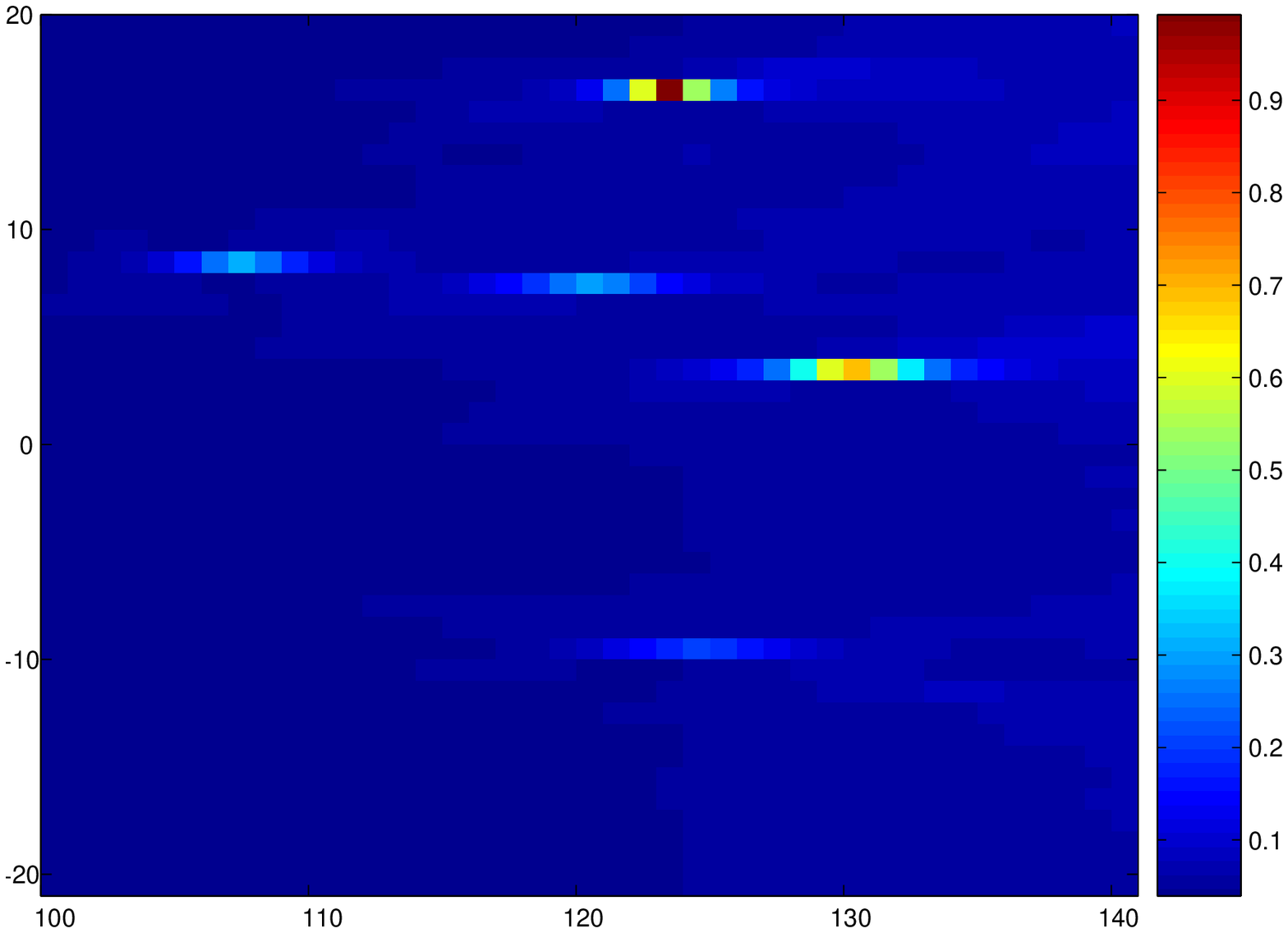}
\includegraphics[scale=0.2]{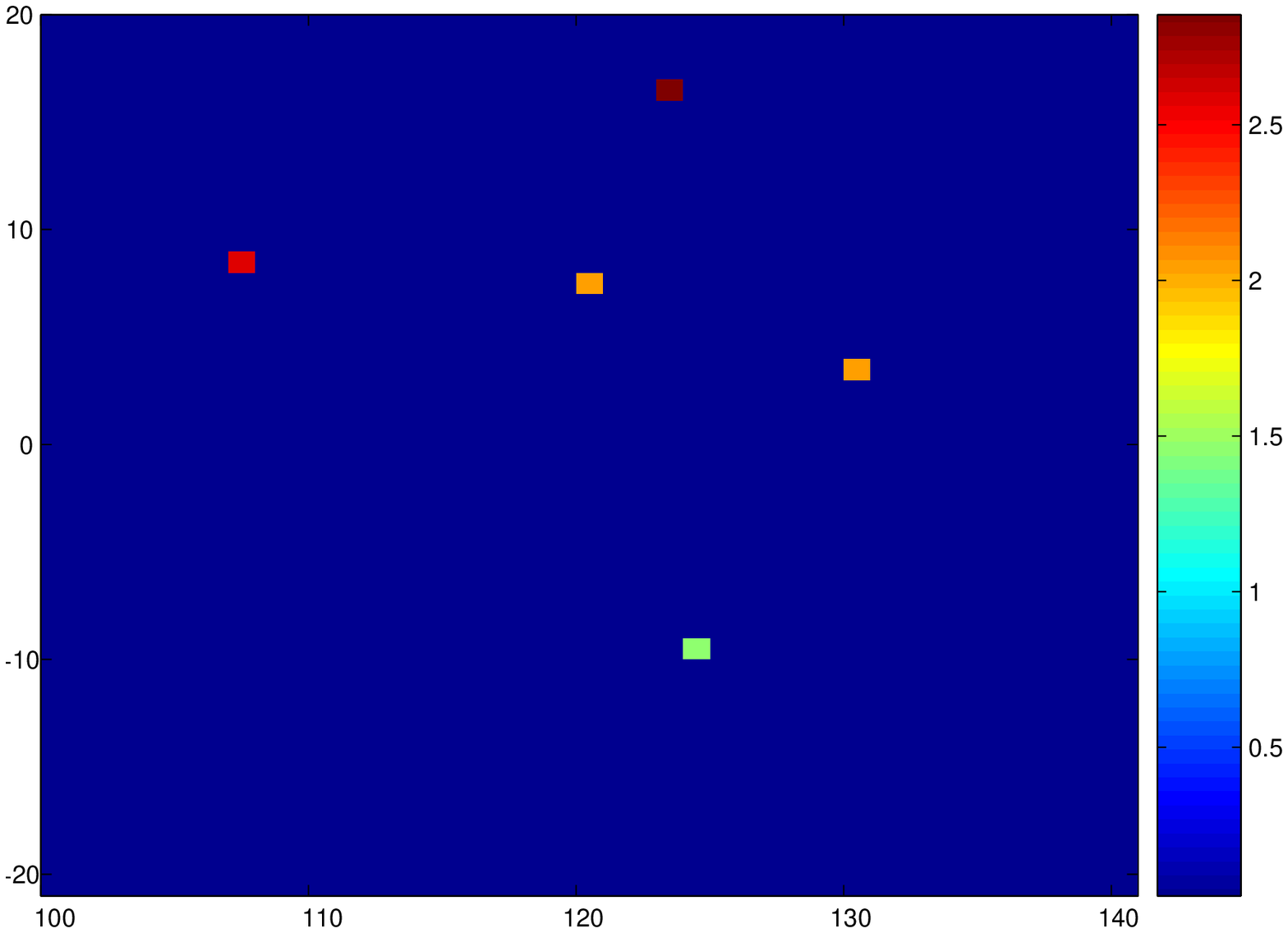}
\end{tabular}
\caption{Images obtained using KM (left), MUSIC (middle), and hybrid-$\ell_1$ (right) with $100\%$ noise in data.
We use $5$ illuminations in KM.}
    \label{multiIllum2D}
\end{center}
\end{figure}

\section{Active array imaging in random media}\label{sec:random medium}
Many natural media vary randomly in space, and therefore, waves propagating in these media also fluctuate in space.
Thus the data collected on the array inherit the uncertainty of the fluctuations of the medium resulting in wave distortion.
Because wave distortion effects induced by the inhomogeneities are very different from additive uncorrelated noise,
imaging through random media is very different from that in homogeneous media.
For example, the images obtained by  Kirchhoff migration become noisy and
change unpredictably with the detailed features of the medium, and thus, are useless.
In this section, we use a simple random phase model
to analyze the performance of the hybrid-$\ell_1$ and MUSIC methods in random media. Both methods use the SVD of the response matrix.
The random phase model distorts the wavefronts but keeps the amplitude of the waves unchanged and is valid in the regime of geometrical optics.

\subsection{Random phase model}\label{subsec:random phase model}
In random media, the Green's function that characterizes the wave propagation from
$\vect x$ to $\vect y$ is given by the wave equation
\begin{equation}\label{eq:wave equation}
\Delta\wG(\vect x, \vect y)+\kappa^2 n^2(\vect x)\wG(\vect x, \vect y)=\delta(\vect x-\vect y),
\end{equation}
where $n(\vect x)=\frac{c_0}{c(\vect x)}$ is the random index of refraction of the
medium with local wave speed $c(\vect x)$. In a homogeneous medium, we have $c(\vect x)\equiv c_0$
for any location $\vect x$, and hence, $\wG_0(\vect x,\vect y)$ defined in
\eqref{greenfunc} is the solution to \eqref{eq:wave equation}.
In random media, however, the wave speed $c(\vect x)$ depends on the position $\vect x$.
In this paper, we consider random fluctuations of the wave speed satisfying the model
\begin{equation}\label{eq:random wave speed}
\frac{1}{c^2(\vect x)}=\frac{1}{c_0^2}\bigg(1+\sigma\mu(\frac{\vect x}{l})\bigg)\, ,
\end{equation}
where $l$ is the correlation length of the medium, $\sigma$ determines the strength of
the flucutation around the constant speed $c_0$, and $\mu(\cdot)$ is a stationary random process with
zero mean and normalized autocorrelation function $R(|\vect x-\vect x'|)=\mE(\mu(\vect x)\mu(\vect x'))$, so $R(0)=1$.
Here, we only consider weak fluctuations such that $\sigma\ll1$. 

When the index of refraction $n(\vect x)$ is random, there is no analytic solution to \eqref{eq:wave equation},
and hence, numerical schemes are often used to obtained an approximated numerical solution. This can be time
consuming, especially for large distances $L$.
Instead, the random phase model provides
an analytical approximation for the Green's function
between two
points at a distance of order $L\gg l\gg \lambda$ from each other, given by
\begin{equation}\label{eq:random green func}
\wG(\vect x,\vect y)=\wG_0(\vect x,\vect y)
\exp{\left[\rmi\sigma\kappa|\vect x-\vect y|\int_0^1\mu(
\frac{\vect x}{l}+\frac{s}{l}(\vect y-\vect x))\,\rmd s\right]}.
\end{equation}
This approximation is valid when (i) the wavelength $\lambda$ is much smaller than the
correlation length $l$ so the geometric optics approximation holds, (ii) the correlation
length $l$ is much smaller than the propagation distance $L$ so the statistics of the phase
is Gaussian, and (iii) the strength of the fluctuation $\sigma$ is small so the amplitude of
the wave is kept unchanged, but large enough to ensure that the perturbations of the phases
are not negligible. The last condition holds when $\frac{\sigma^2L^3}{l^3}\ll\frac{\lambda^2}{\sigma^2lL}\le1$
(for details, see \cite{Tatarski61,Rytov89,Borcea11}). Note that although we take weak fluctuations, 
the distortion of the wavefronts by the inhomogeneities of the medium is observable
because the wave travels long distances. 

Comparing \eqref{eq:random green func} to the homogeneous Green's function \eqref{greenfunc} we see that, in this
regime, only the phase is perturbed by the random medium and the magnitude remains unchanged.
The following result shows that the second order moment of \eqref{eq:random green func}
is close to the expected value.
\begin{proposition}\label{prop:2nd moment}
Assume that the autocorrelation function $R\in\mL_2(\mR_+)$ is differentiable and its derivative
$\dot{R}$ satisfies $\frac{\dot{R}(t)}{t}\in\mL(\mR_+)$ with exponential decay. Let $\vect y_1$ and
$\vect y_2$ be two points on the same plane at a distance $L$ from point $\vect x$, such that
$\lambda\ll|\vect y_1-\vect y_2|\ll L$. Then, for the Green's function \eqref{eq:random green func} we have
\begin{equation}\label{eq:2nd moment}
\mE\bigg(\wG(\vect x, \vect y_1)\overline{\wG(\vect x,\vect y_2)}\bigg)\approx
\wG_0(\vect x, \vect y_1)\overline{\wG_0(\vect x, \vect y_2)}\rme^{-\frac{\kappa^2 a_e^2}{2L^2}|\vect y_1-\vect y_2|^2},
\end{equation}
with
\begin{equation}\label{eq:effective aperture}
a_e=\sigma L\left(-1-\frac{2L}{3l}\int_0^\infty\frac{\dot{R}(t)}{t}\,\rmd t\right)^{\frac{1}{2}}.
\end{equation}
\end{proposition}
According to Proposition~\ref{prop:2nd moment}, the back-propagated signal in random media is equal
to that in a homogeneous medium times a Gaussian factor with variance $\frac{L^2}{\kappa^2a_e^2}$.
The length $a_e$ is an intrinsic property of the random media. It depends on
the propagation distance $L$ and the statistics of the random fluctuations of the medium.
According to \eqref{eq:2nd moment}, as $a_e$ increases, the refocused spot size for time reversal in random media is tighter.
Using the moment estimate \eqref{eq:2nd moment}, we immediately have the following variance estimate.
\begin{corollary}\label{cor:variance estimate}
Under the same conditions given in Proposition~\ref{prop:2nd moment}, we have for  \eqref{eq:random green func} 
\begin{eqnarray}\label{eq:variance estimate}
\mE\left|\wG(\vect x,\vect y_1)\overline{\wG(\vect x,\vect y_2)}-
\mE\left(\wG(\vect x,\vect y_1)\overline{\wG(\vect x,\vect y_2)}\right)\right|^2\approx\frac{1-\rme^{-\frac{\kappa^2a_e^2}{L^2}|\vect y_1-\vect y_2|^2}}{16\pi^2|\vect x-\vect y_1|^2|\vect x-\vect y_2|^2}.
\end{eqnarray}
\end{corollary}
This property of the Green's function \eqref{eq:random green func} is generally referred to as
{\it self-averaging}, which means that a time-reversed, back-propagated signal in a random medium refocuses near the source independently of the particular realization of the random medium. Self-averaging properties of back-propagated signals in random media have been shown under broad-band settings, see for example \cite{fink93, Blomgren02,BPR02, Papanicolaou04,pmsf}, but are not true, in general, for narrow-band signals. Corollary \ref{cor:variance estimate} states that
the Green's function corresponding to the phase random model \eqref{eq:random green func} exhibits
the same properties, even for narrow-band pulses.

The next result states that a single realization of a back-propagated signal via the Green's function vector $\vect\wg(\vect y)$ \eqref{eq:GreenFuncVec} is 
stable when the aperture of the imaging array becomes infinity, in the sense that its value is close to the
back-propagated signal averaged over multiple realizations. 
\begin{proposition}\label{prop:statistical stability}
For a large aperture size $a$, the signal sent from $\vect y_1$, recorded at the array, and back-propagated at $\vect y_2$ is statistically stable 
for  \eqref{eq:random green func} in the sense that
\begin{equation}\label{eq:statistical stability}
\frac{\mE|\vect\wg^\ast(\vect y_1)\vect\wg(\vect y_2)-\mE(\vect\wg^\ast(\vect y_1)\vect\wg(\vect y_2))|^2}
{\mE\|\vect\wg(\vect y_1)\|^2\mE\|\vect\wg(\vect y_2)\|^2}\rightarrow0,\qquad a\rightarrow\infty.
\end{equation}
\end{proposition}
Furthermore, as shown in Appendix~\ref{appendix:proofs}, the decay rate of \eqref{eq:statistical stability}
is bounded by
\begin{equation}\label{eq:log decay}
\left(1-\rme^{-\frac{\kappa^2}{L^2}|\vect y_1-\vect y_2|^2a_e^2}\right)\left(\frac{l^2}{L^2\log\left(1+\frac{a^2}{4L^2}\right)}\right)
\end{equation}
for any  aperture size $a$. Thus, for short distances so $L$ is of order $a$, the decay rate is only logarithmic in $a$. However, 
when the distance is large such
that $L\gg a$ (e.g. the remote sensing regime), we can use a linear approximation on the logarithm function and obtain, up to a constant,
the approximated decay rate 
\begin{equation}\label{eq:quadratic decay}
\left(1-\rme^{-\frac{\kappa^2}{L^2}|\vect y_1-\vect y_2|^2a_e^2}\right)\left(\frac{l}{a}\right)^2,
\end{equation}
which implies a quadratic decay rate in $a$. Such quadratic decay rate in $a$ can also be justified by using the paraxial approximation when
$L \gg a$. As it is noted in Remark~\ref{rem:paraxial}, the decay rate in the paraxial approximation is given by \eqref{eq:paraxial approx}, so
the bound in \eqref{eq:quadratic decay} coincides with the paraxial approximation up to oscillations caused by the $\sinc$ function.

We emphasize that, although \eqref{eq:log decay} indicates little improvement in the refocused spot size when the physical aperture of
the array $a$ is large, Proposition~\ref{prop:statistical stability} implies that large arrays stabilize refocusing
of narrow-band pulses in random media. The fact that large arrays improve the quality of imaging in cluttered media has been recognized
for long time in radar and seismic imaging. The theoretical results obtained here using the random phase model support this empirical
observation.

Next, we give a result that is essential for the performance of hybrid-$\ell_1$ and MUSIC methods in random media.
When imaging in random media, the Green's function vector $\vect\wg(\vect y)$ 
is random and not known. The best we can do is to use the deterministic Green's function vector $\vect\wg_0(\vect y)$ which
is, in general, quite different from the random one. Furthermore, replacement of $\vect\wg(\vect y)$ by $\vect\wg_0(\vect y)$
might give good results for some realizations of the random medium but not for others. The next result tells us that a single realization of 
a narrow-band backpropagated signal $\vect\wg_0^\ast(\vect y_1)\vect\wg(\vect y_2)$, with the Green’s function corresponding to the
phase random model, is stable as the aperture
of the array becomes infinity, which means that its value is close to the back-propagated signal averaged over multiple
realizations.
\begin{proposition}\label{prop:asymp orthogonality}
Let $\vect y_1$ and $\vect y_2$ be two points satisfying $\lambda\ll|\vect y_1-\vect y_2|\ll L$.
Then, for the Green's function \eqref{eq:random green func}, we asymptotically have
\begin{equation}\label{eq:o1}
\vect\wg^\ast(\vect y_1)\vect\wg(\vect y_2)\rightarrow0\,,\qquad a\rightarrow\infty,
\end{equation}
and for the mixed inner products
\begin{equation}\label{eq:o2}
\vect\wg_0^\ast(\vect y_1)\vect\wg(\vect y_2)\rightarrow0\,,\qquad a\rightarrow\infty\, ,
\end{equation}
where the limit is under probability measure $\mathbb{P}$ induced by $\mu(\cdot)$. Moreover, the mixed
inner product is also statistically stable in the sense that
\begin{equation}\label{eq:mixed stability}
\frac{\mE\left|\vect\wg_0^\ast(\vect y_1)\vect\wg(\vect y_2)-\mE(\vect\wg_0^\ast(\vect y_1)\vect\wg(\vect y_2))\right|^2}
{\|\vect\wg_0(\vect y_1)\|^2\mE\|\vect\wg(\vect y_2)\|^2}\rightarrow0,\quad a\rightarrow\infty.
\end{equation}
\end{proposition}
Proposition~\ref{prop:asymp orthogonality} implies that imaging in random media,
i.e., a narrow-band pulse $\vect\wg(\vect y_2)$ back-propagated via $\vect\wg_0(\vect y_1)$ in the medium, is 
statistically stable for large arrays, and the pulse heading to the point $\vect y_1$ will 
focus around the correct location $\vect y_2$.

\subsection{The hybrid-$\ell_1$ method}\label{subsec:hybrid l1}
Assume that multiple scattering between the scatterers is negligible, and wave distortion is well described by the 
random phase model \eqref{eq:random green func}. Then, the response matrix has the form
\begin{equation}
\label{eq:randomP}
\vect\wP =\sum_{j=1}^M\alpha_j\vect\wg(\vect y_{n_j})\vect\wg^T(\vect y_{n_j}) \, ,
\end{equation}
with random Green's function vectors $\vect \wg(\vect y)=[\wG(\vect x_{1},\vect y),\ldots,\wG(\vect x_{N},\vect y)]^T$, 
where $\wG(\vect x,\vect y)$ is given by \eqref{eq:random green func}. 
Because multiple scattering is not important, the transformation $\vect\Ac_{\wf}\bfrho=\vect b$
that relates the reflectivity vector $\bfrho_0$ with the data measured at the array $\vect b$ is linear. However, 
when the medium contains random inhomogeneities, 
this linear transformation is random because 
the column vectors of the operator $\vect\Ac_{\wf}$  are random functions.
Therefore, to solve for the reflectivity vector $\bfrho_0$,
we have to approximate the unknown Green's function vectors $\vect\wg(\vect y)$ by the ones for a homogeneous medium, 
after which, $\vect\Ac_{\wf}$ becomes the linear operator $\vect\Ac_{0\wf}$ given in \S\ref{sec:single scattering}.
The substitution of $\vect\wg(\vect y)$ by $\vect\wg_0(\vect y)$
introduces a discrepancy between model and data, and thus, we must solve the $\ell_1$ minimization 
problem \eqref{eq:smv born noise} with an inequality constraint when data from  a single illumination is available.
We note that to estimate the error bound of such discrepancy is nontrivial  due to the existing correlation in 
the measurement noise in random media.

According to Proposition~\ref{prop:asymp orthogonality}, when the scatterers are far apart, the random Green's
function vectors $\vect\wg(\vect y_{n_j})$, $j=1,\dots,M$, are approximately orthogonal in probability as the size
of the imaging array becomes large. In this case, we can associate to each of the $M$ scatterers a nonzero singular value 
$\sigma_j=\alpha_j\|\vect\wg(\vect y_{n_j})\|_{\ell_2}^2$ with singular vectors 
\begin{equation}\label{eq:random svd}
    \wU_{\cdot j}=\frac{\vect\wg(\vect y_{n_j})}{\|\vect\wg(\vect y_{n_j})\|_{\ell_2}},\quad
    \wV_{\cdot j}=\overline{\frac{\vect\wg(\vect y_{n_j})}{\|\vect\wg(\vect y_{n_j})\|_{\ell_2}}}\, ,\qquad j=1,\dots,M.
\end{equation}
If the full array response matrix $\vect\wP$ (or its SVD) is available, we can use the hybrid-$\ell_1$ as follows.
We use the right singular vectors of $\vect\wP$ 
as the illumination vectors, and we project the data to the space spanned by the left singular vectors in the same way as
shown in \S\ref{sec:single scattering}. Thus, we form the same hybrid-$\ell_1$ optimization given in \eqref{eq:hybrid l1} and
there is no need to estimate the error bound of the discrepancy between model and data.
Recall that $(ij)^\mathrm{th}$ entry of the hybrid-$\ell_1$ matrix $\vect\Bc$ is related to the singular vectors in the form of
$\overline{\vect\wg_0^\ast(\vect y_j)\wU_{\cdot i}}\vect\wg_0^T(\vect y_j)\wV_{\cdot i}$, and the singular vectors of the response
matrix $\vect\wP$ are related to the random Green's function vectors as in \eqref{eq:random svd}. Hence, each entry
is the mixed inner product \eqref{eq:o2} between $\vect\wg_0(\vect y_j)$ and $\vect\wg(\vect y_i)$.

Without loss of generality, assume that the $M$ scatterers correspond to the first $M$ columns in $\vect\Bc$, i.e. $n_j=j$,
$j=1,\ldots,M$. We can write $\vect\Bc$ in the form of block matrices as $\vect\Bc=[\vect\Bc_{M\times M}\ \ \vect S]$, where the
columns in $\vect S$ correspond to grid points where there is no scatterers. Due to Proposition~\ref{prop:asymp orthogonality},
the submatrix $\vect\Bc_{M\times M}$ is a diagonal matrix $\vect D$ perturbed by a matrix $\vect E$ with $\|\vect E\|_{1\rightarrow1}\ll1$.
The value of $\|\vect E\|_{1\rightarrow1}$ depends on the minimal distance between any two scatterers, and is close to zero if they are
well separated. Due to the incomplete phase cancellation in the mixed inner products, $\vect D$ is composed of complex values
but, as long as the scatterers are far apart, $\vect D +\vect E$ is diagonal dominated. 
Furthermore, because of the statistical stability of the mixed inner product \eqref{eq:mixed stability}, if
$\|\vect D^{-1}\vect S\|_{1\rightarrow1}<1-\|\vect D^{-1}\vect E\|_{1\rightarrow1}$,  \eqref{eq:hybrid l1}  gives the right
solution even though there exists (correlated) noise in the data caused by the random media. 
This saves computational time by eliminating the need for difficult error bounds testing.

\subsection{MUSIC}\label{subsec:music}
In random media, the exact knowledge of the reference illumination $\vect\wg(\vect y)$ is unknown,
and therefore imaging using MUSIC amounts to approximate the unknow illumination with that in homogeneous context, i.e.
\begin{equation}\label{eq:music imager}
    \mathcal{I}^{RM}_{MUSIC}(\vect y_k)=\frac{\min_{1\le j\le K}\|\mathcal{P}\vect\wg_0(\vect y_j)\|_{\ell_2}}{\|\mathcal{P}\vect\wg_0(\vect y_k)\|_{\ell_2}},\quad k=1,\ldots,K.
\end{equation}
According to \eqref{eq:random svd} and \eqref{eq:o2}, the proxy $\vect\wg_0(\vect y)$ will focus around the scatterers when the imaging
array is large enough. However, the image obtained by \eqref{eq:music imager} will have lower resolution compared to
the ideal case of \eqref{eq:music tr}, since the random phase in $\vect\wg(\vect y)$ cannot be fully cancelled
out using the homogenous Green's function vector $\vect\wg_0(\vect y)$.

\subsection{Numerical experiments}
\label{sec:numerical simulation random}

\begin{figure}[t]
\begin{center}
\includegraphics[scale=0.25]{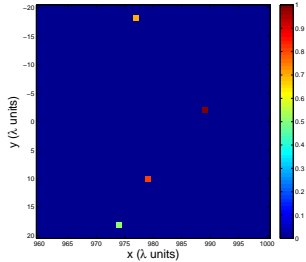}
\caption{Original configuration of the scatterers in a $41\times41$ image window with grid points separated by $1$.}
\label{original}
\end{center}
\end{figure}

We now present numerical simulations in two dimensions to illustrate the performance  of the MUSIC and hybrid-$\ell_1$ imaging methods in random media. 
We consider a random medium with correlation length $l=20\lambda$, and standard deviation of the fluctuations $\sigma=0.1\%$. 
For comparison purposes we also show images obtained by Kirchhoff migration with a single illumination sent from
the central transducer of the array. 

We consider a small and a large array, both consisting of $501$ transducers uniformily distributed over the aperture.
The small array has an aperture of $25l$, and the large array an aperture of $100l$. 
Four scatterers are placed within an IW  of size $41\lambda\times41\lambda$,
which is at a distance $L= 50 l$ from the linear array, see Figure~\ref{original}. 
We discretize the IW using a uniform grid with points separated by one wavelength $\lambda$ (the spatial unit in all the figures is $\lambda$).
The amplitudes of the reflectivities of the scatterers, $|\alpha_j|$, are $0.8$, $1.0$, $0.5$, and $0.7$.  
The phases are set randomly in each realization. 
We compute the array data \eqref{eq:randomP} using the Green's function given by \eqref{eq:random green func}. 
The line integral of the random field in \eqref{eq:random green func} is approximated 
by a quadrature rule. 

As a reference, we show in Figure~\ref{fig:homogeneous} the images obtained in a homogeneous medium using  KM (left column), MUSIC (center column), 
and hybrid-$\ell_1$ methods (right column) with noiseless data. The top and bottom rows show the results for the small array and the large array, respectively. 
As expected, the resolution of the KM images improves greatly for large arrays. On the other hand, in a homogeneous medium, MUSIC and hybrid-$\ell_1$ 
achieve an excellent resolution, even for small arrays. 
The images shown in Figure~\ref{fig:homogeneous} do not change too much when the data is corrupted with up to 
$100 \%$ of additive noise \cite{CMP13}. 

\begin{figure}[t]
\begin{center}
\begin{tabular}{ccc}
\includegraphics[scale=0.20]{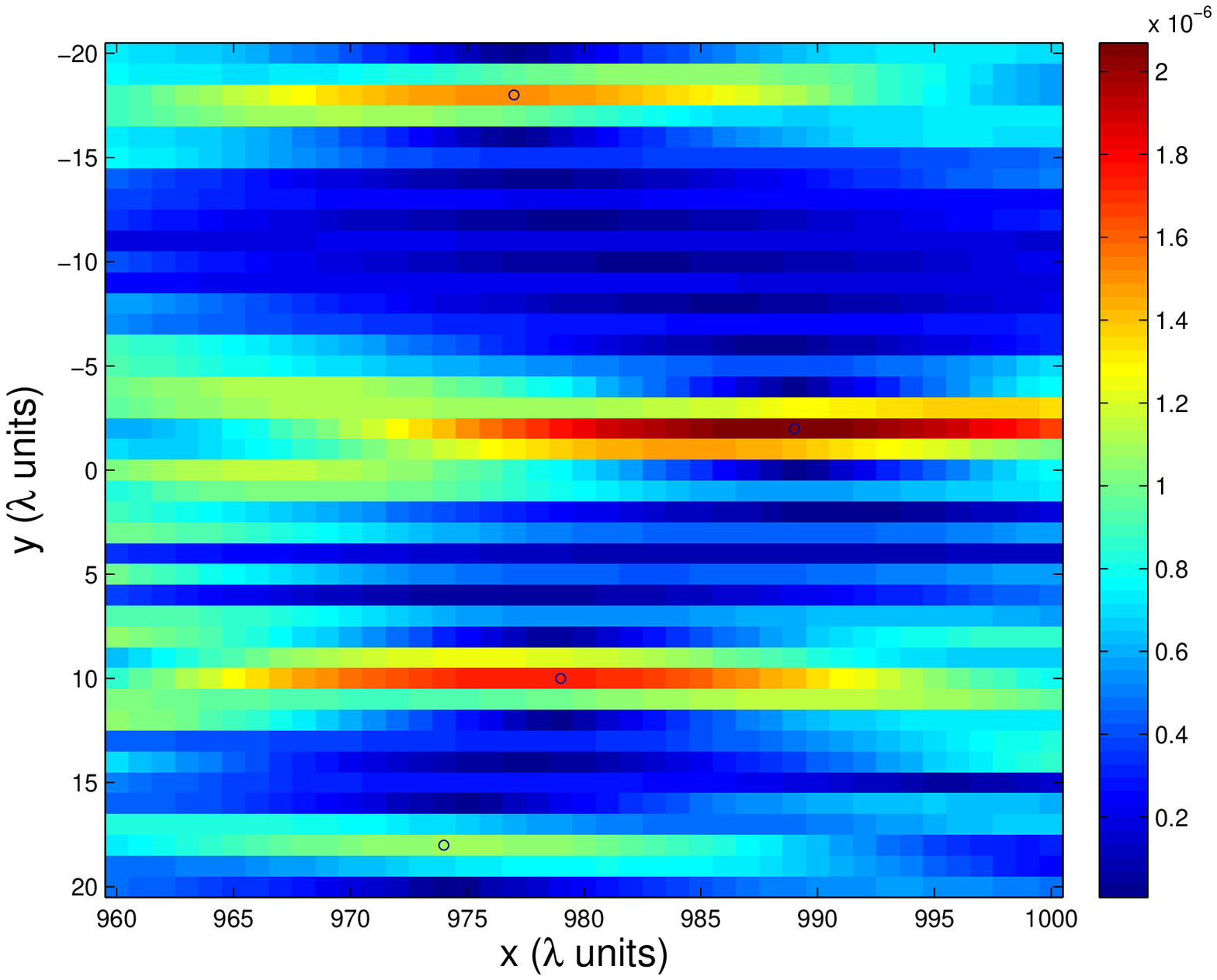}&
\includegraphics[scale=0.20]{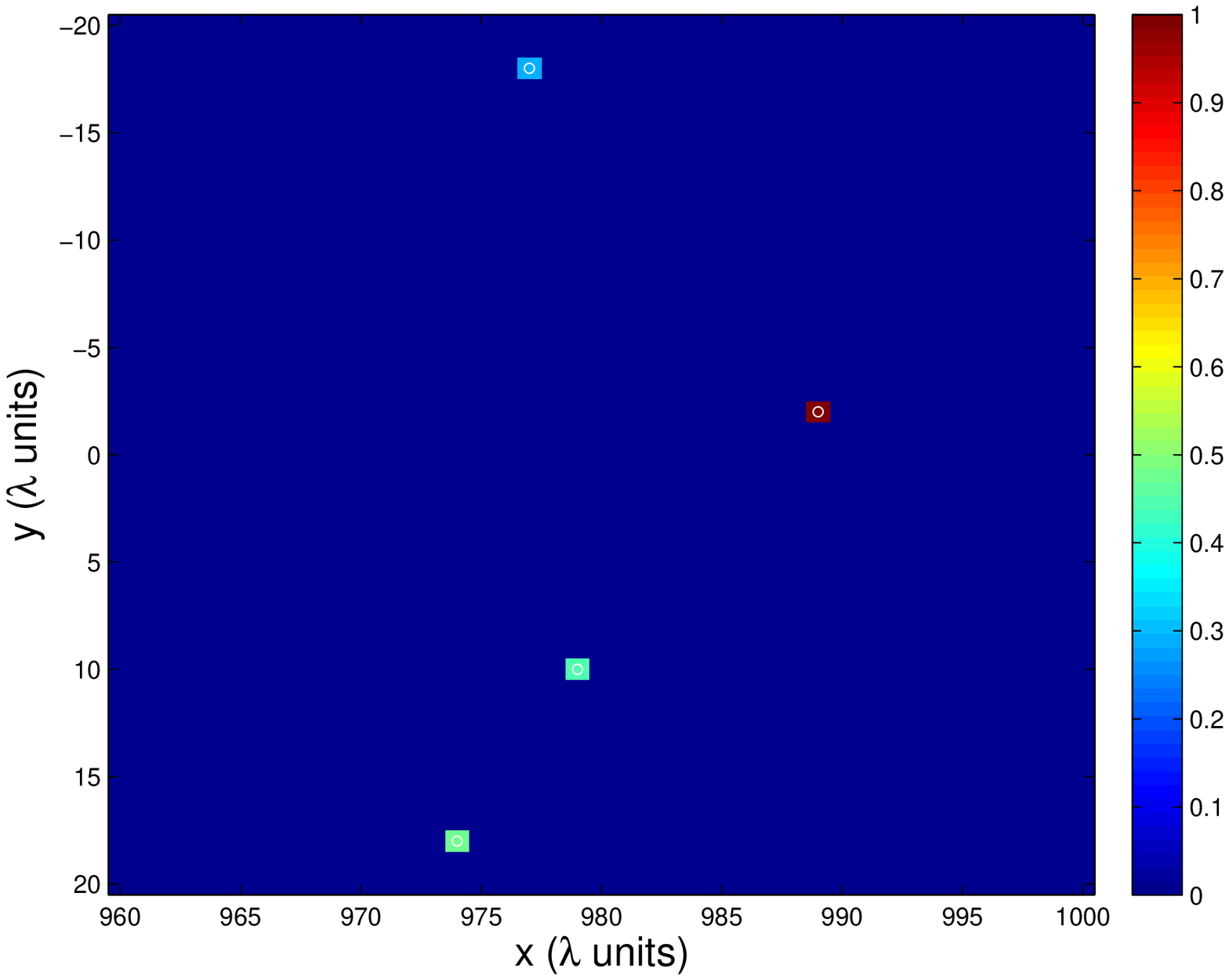}&
\includegraphics[scale=0.20]{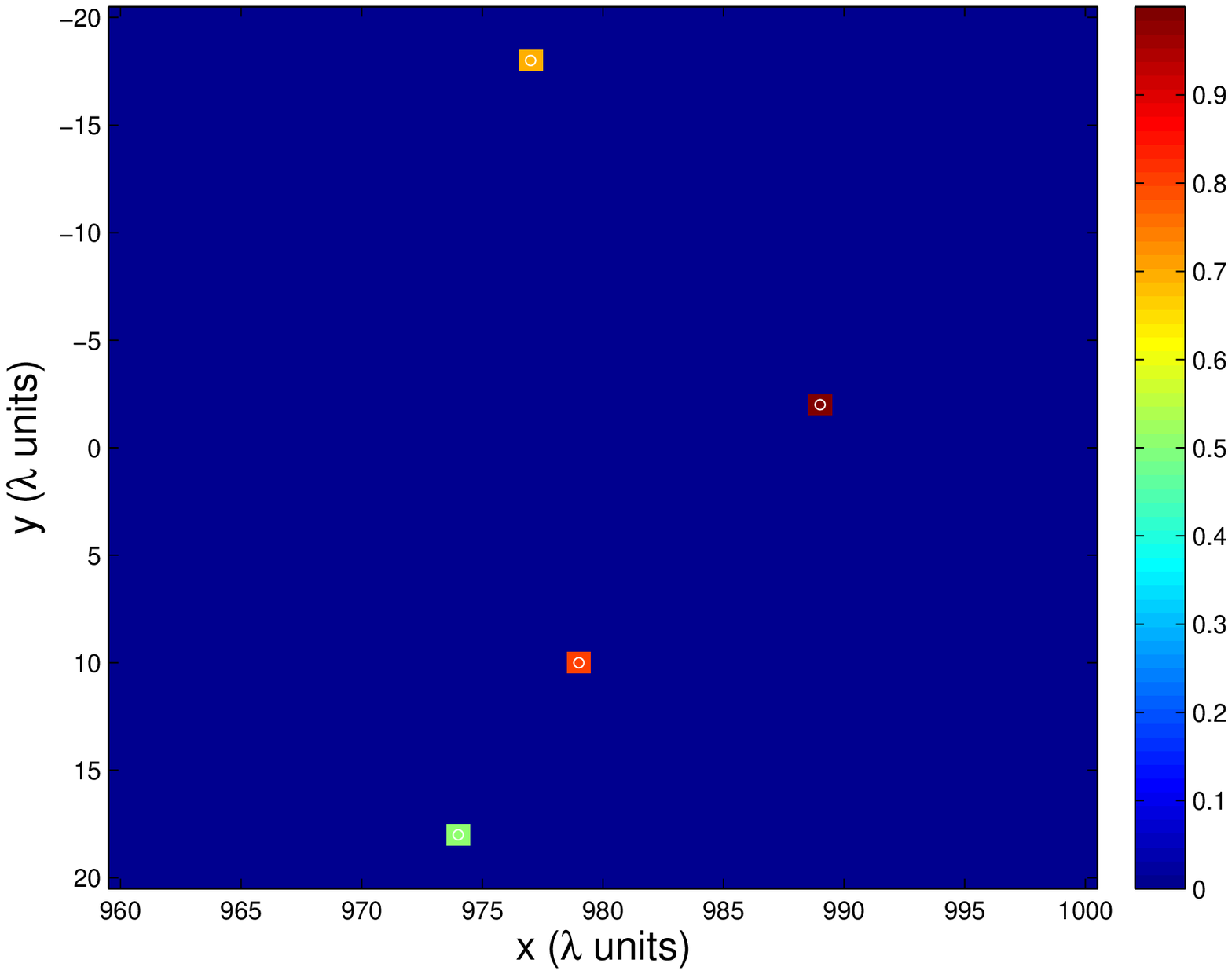}\\
\includegraphics[scale=0.20]{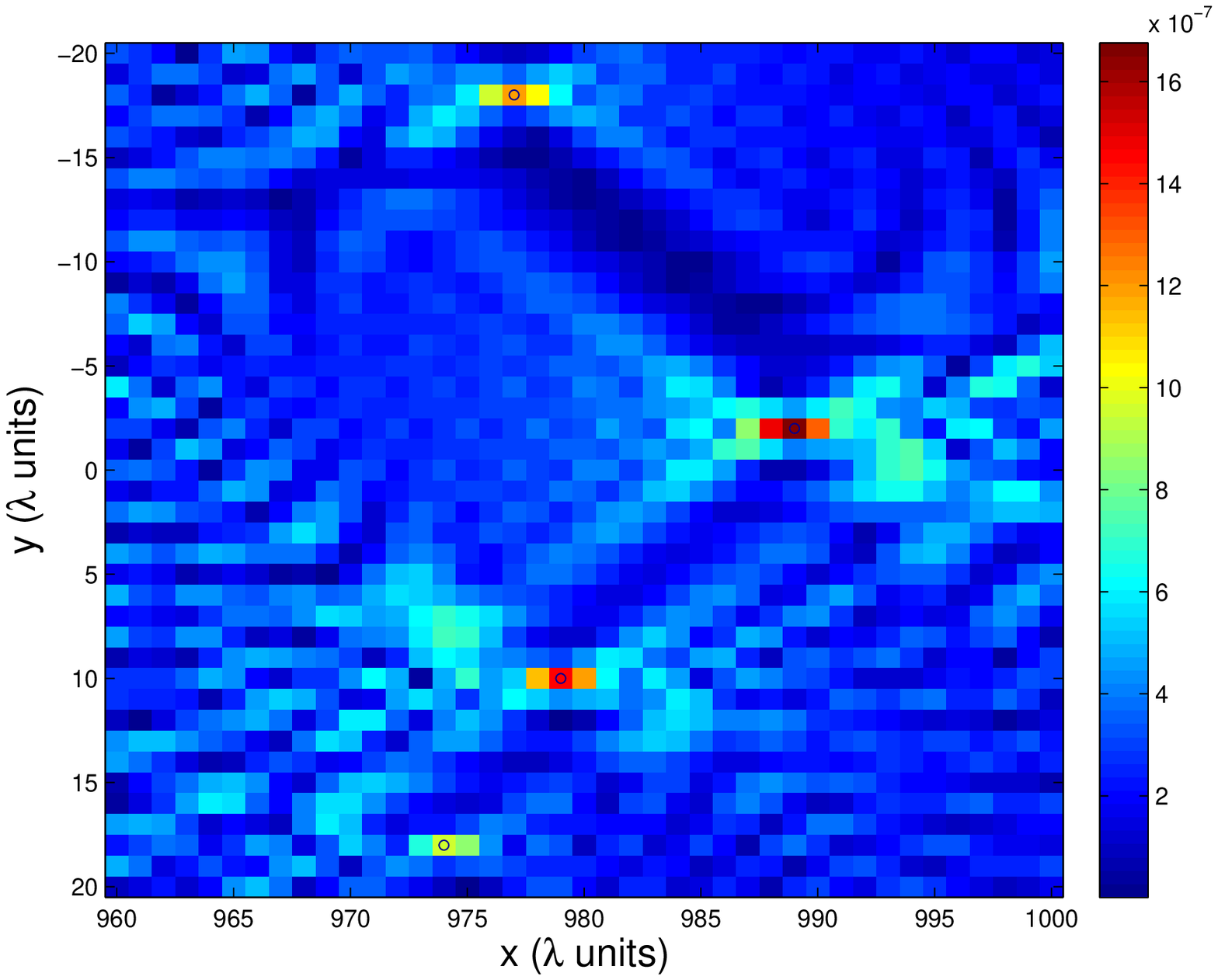}&
\includegraphics[scale=0.20]{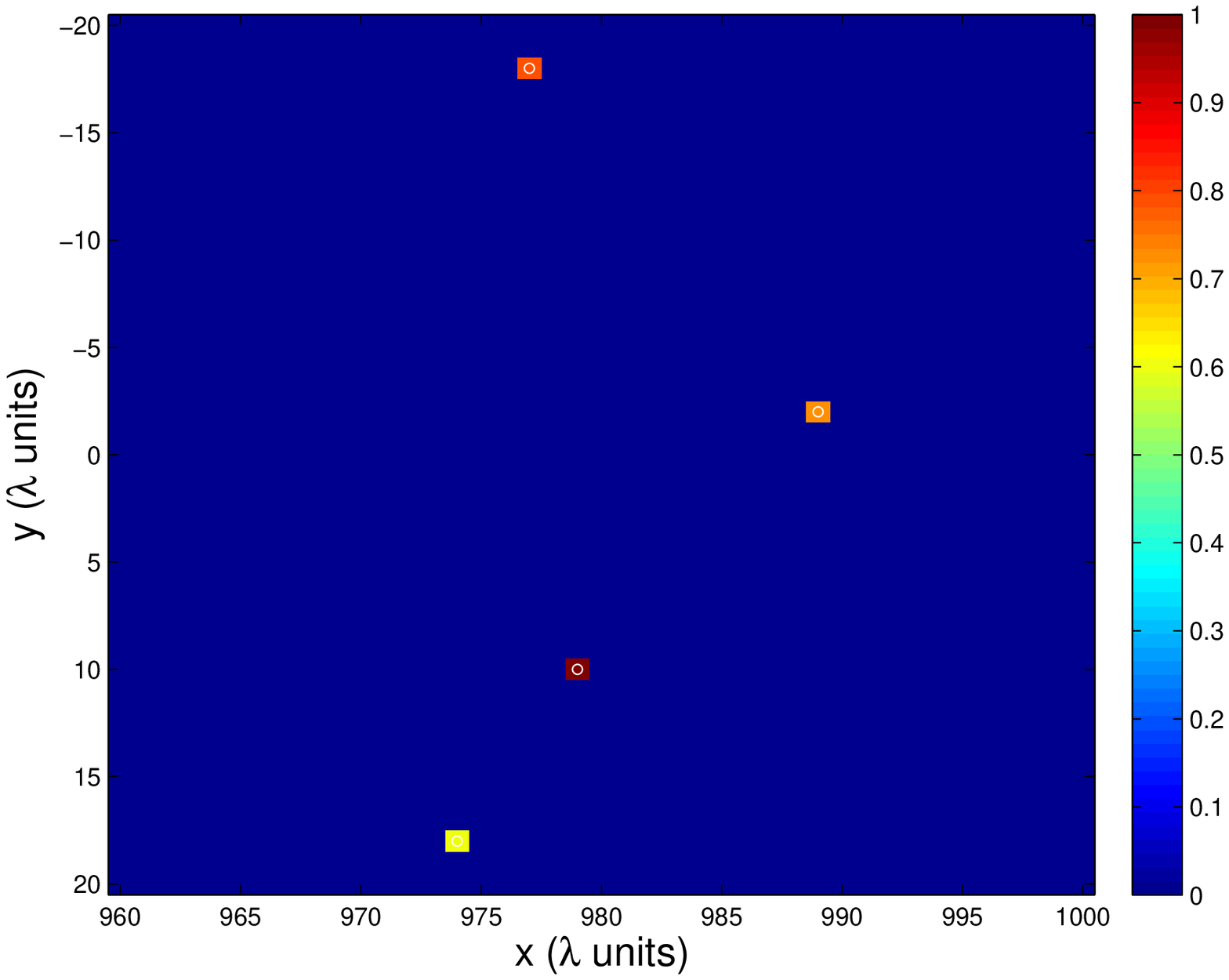}&
\includegraphics[scale=0.20]{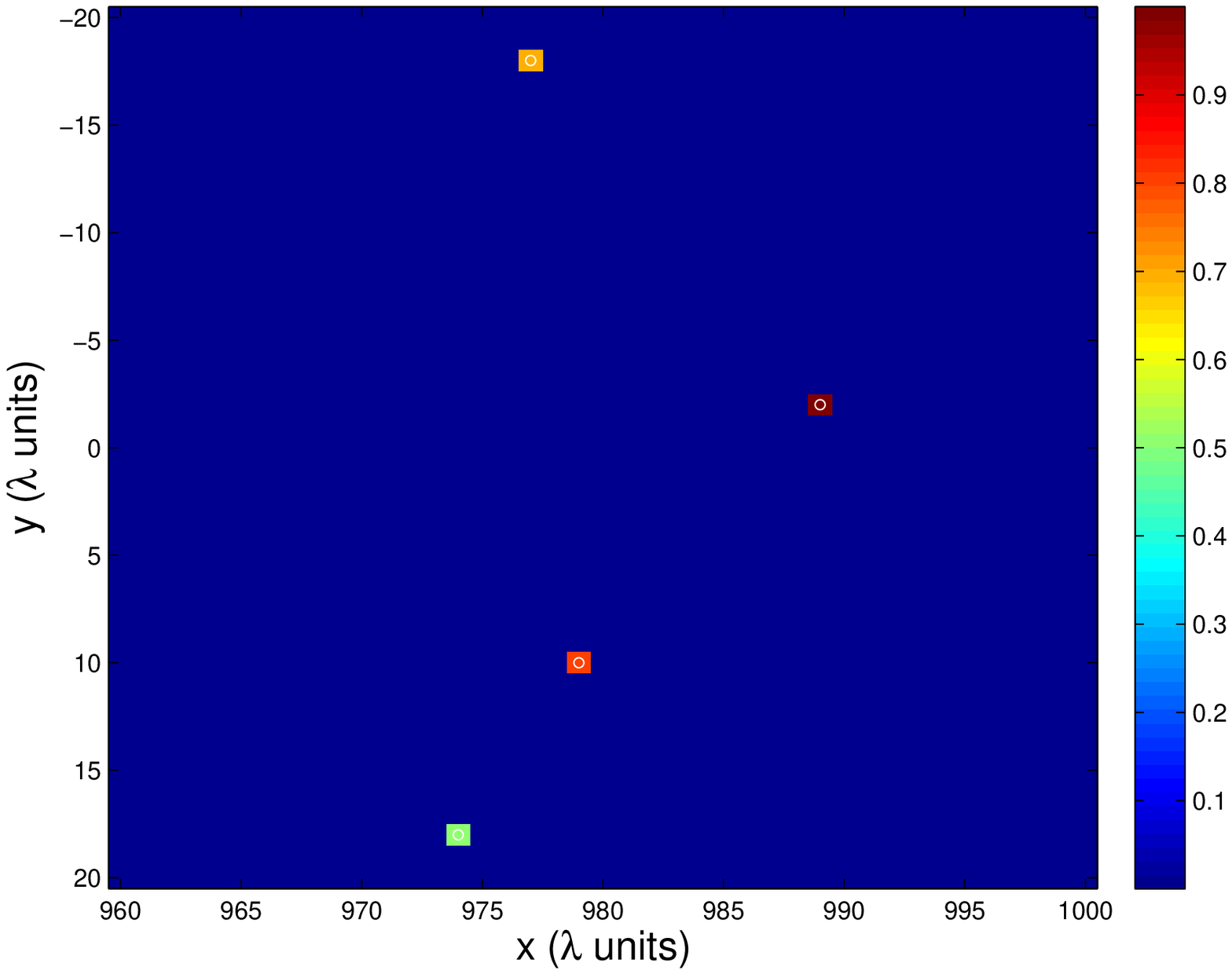}
\end{tabular}
\end{center}
\caption{Top row: small array. Bottom row: large array. KM (left column), MUSIC (center column) and hybrid-$\ell_1$ (right column) images in homogeneous media.
There is no additive noise in the data.}
\label{fig:homogeneous}
\end{figure}

The situation changes when there is correlated noise in the data because the signals propagate through a random medium with a complex
structure. This is illustrated in Figure~\ref{fig:smallarray}, where we show the images produced
by these imaging methods using a small array ($a= 25 l$). The three imaging methods show different behaviors though. Kirchhoff migration 
completely fails to image the scatterers, as can be seen in the top row the figure. There is not only degradation in the resolution,
but also loss of stability. Observe that the images obtained with KM are significantly different from one realization to another. 

The images obtained with MUSIC (middle row of Figure~\ref{fig:smallarray}) are also blurred compared to those obtained in a homogeneous medium. 
Furthermore, the images change from one realization of the random medium to another and, therefore, MUSIC is also unstable if the array size is small. 
The hybrid-$\ell_1$ method also produces images that change from one realization to another (bottom row of Figure~\ref{fig:smallarray}), 
but tries to keep a good resolution to provide a sparse solution. Observe that the detected scatterers dance along the cross-range direction around the true 
locations indicated in the figure with white dots. The images obtained with the hybrid-$\ell_1$ method also show some ghosts.

\begin{figure}[t]
\begin{center}
\begin{tabular}{cccc}
\includegraphics[scale=0.17]{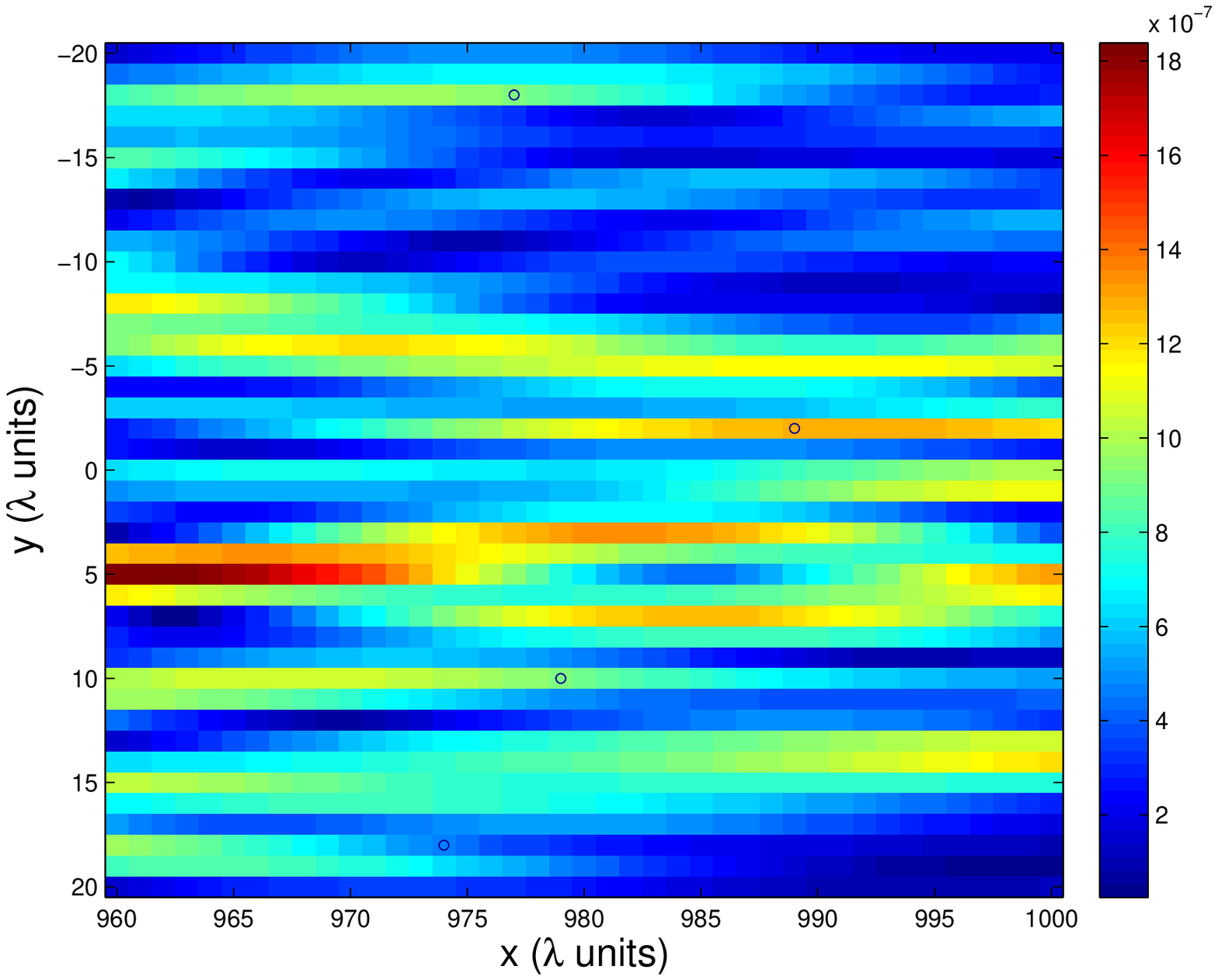}&
\includegraphics[scale=0.17]{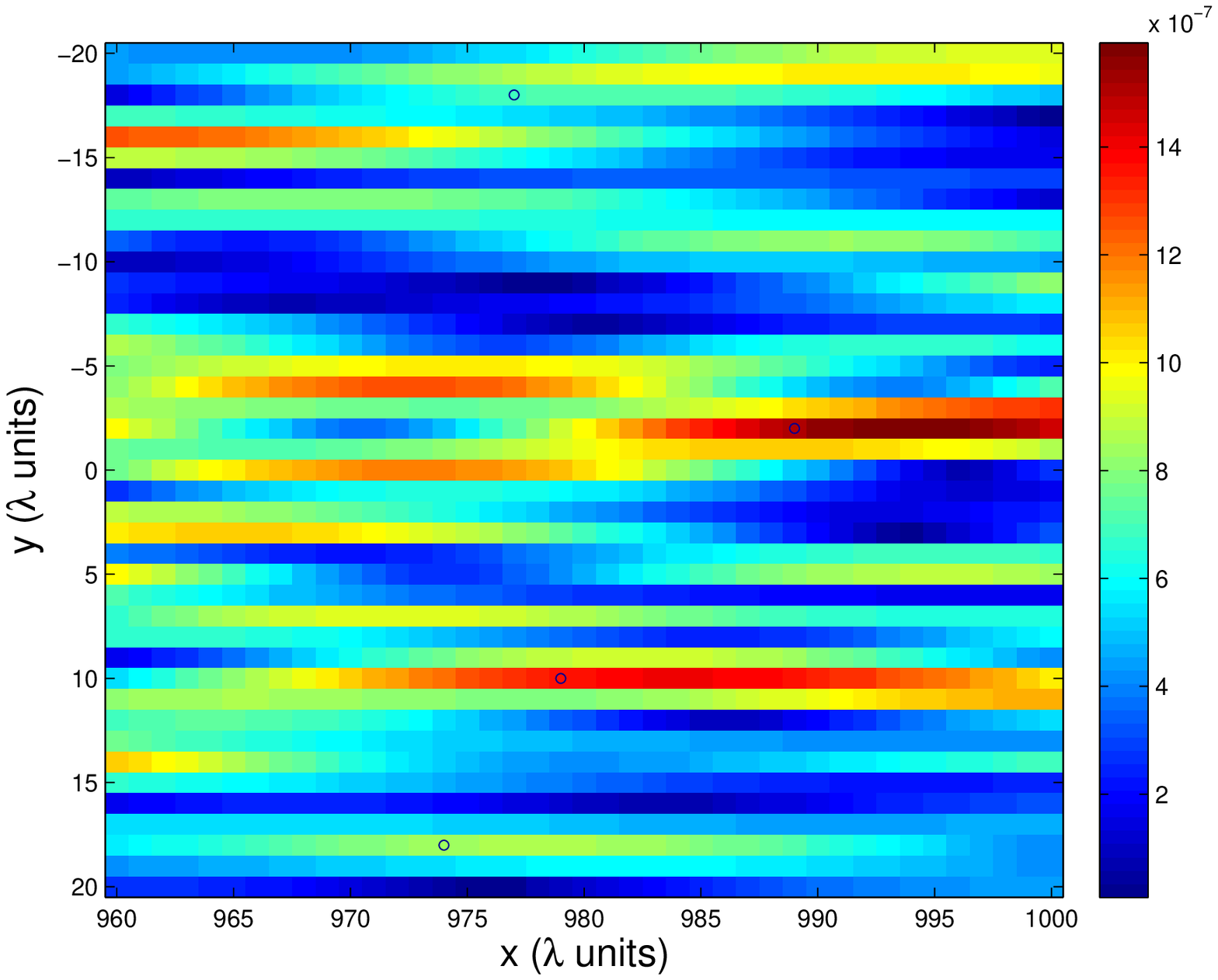}&
\includegraphics[scale=0.17]{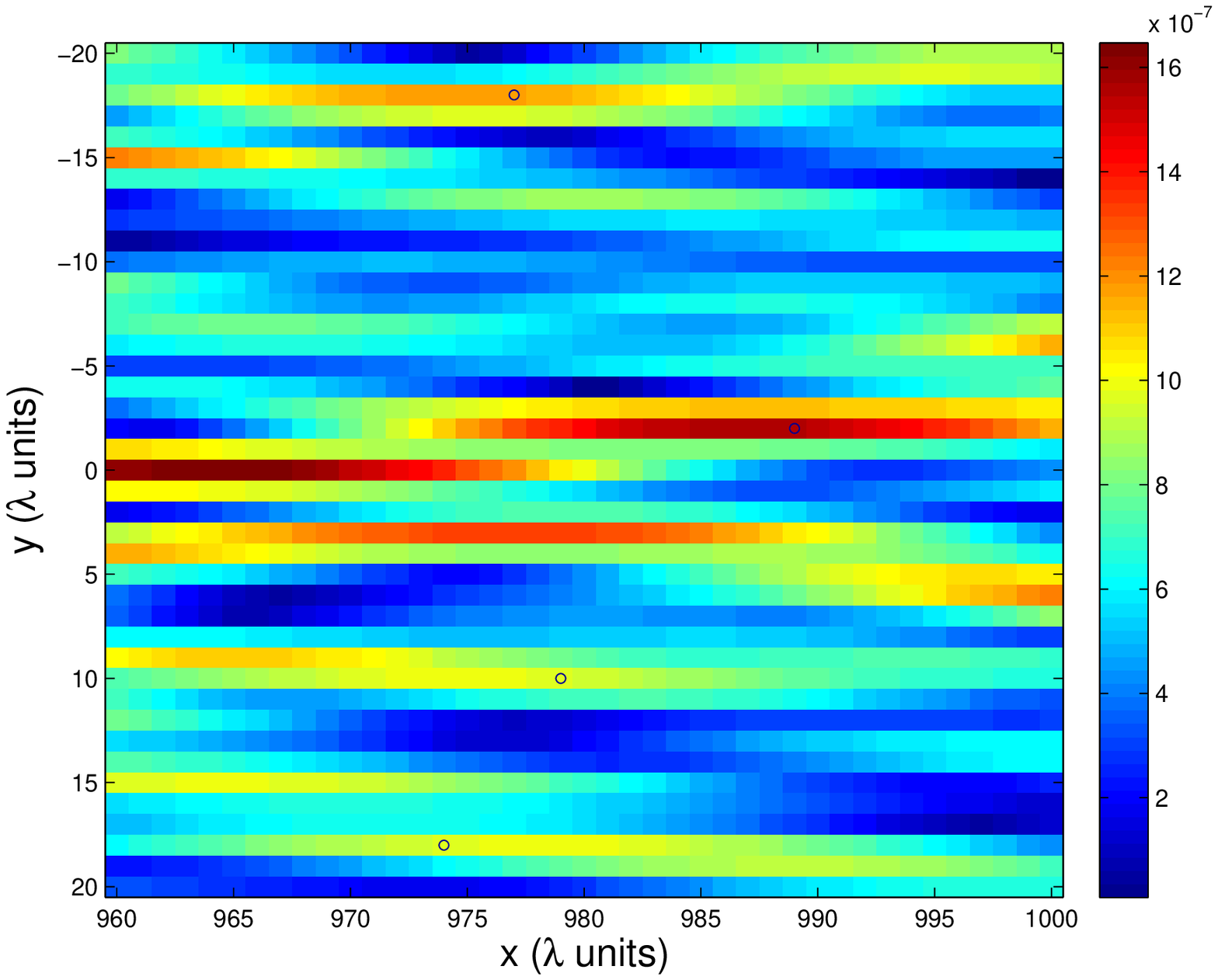}&
\includegraphics[scale=0.17]{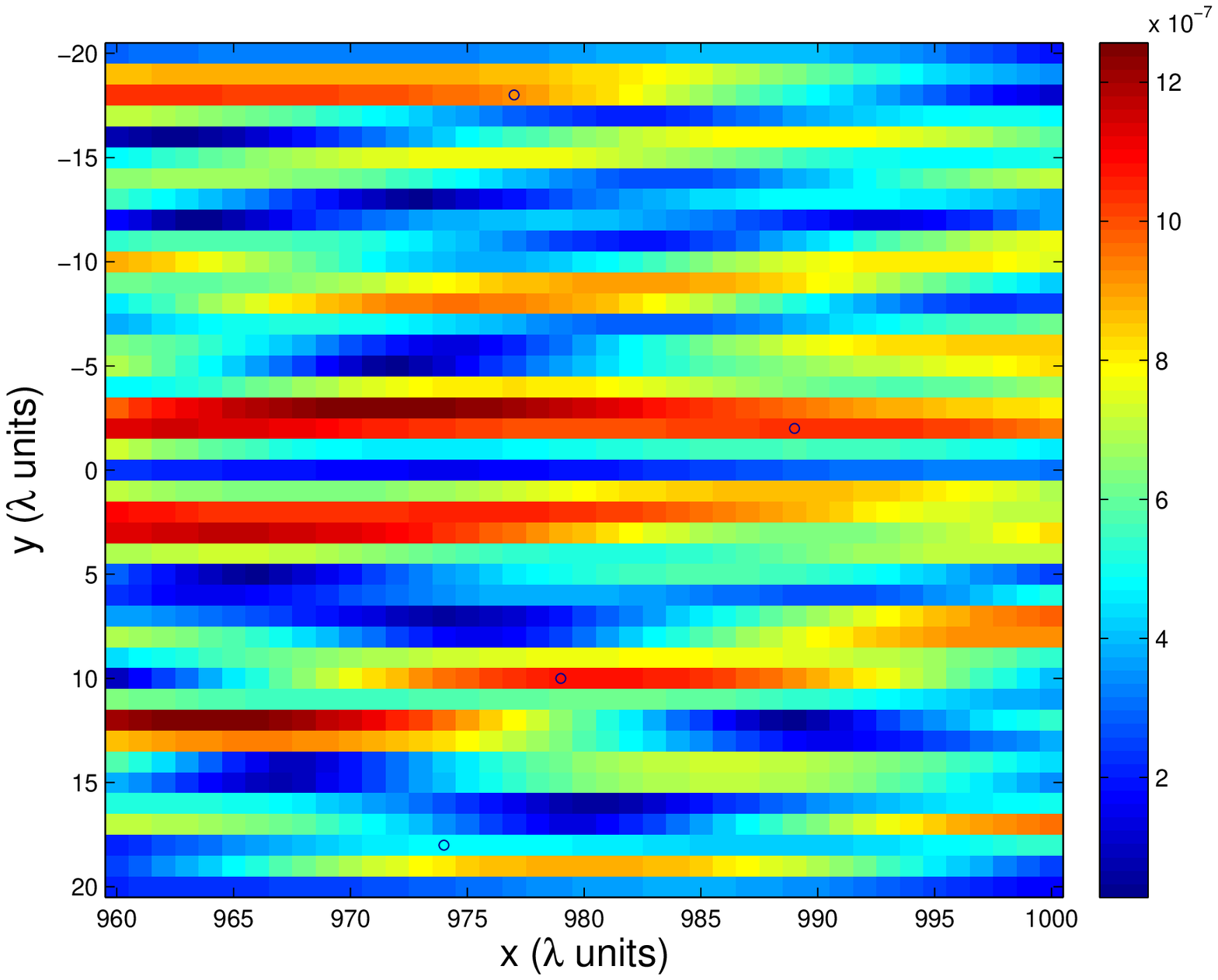}\\
\includegraphics[scale=0.17]{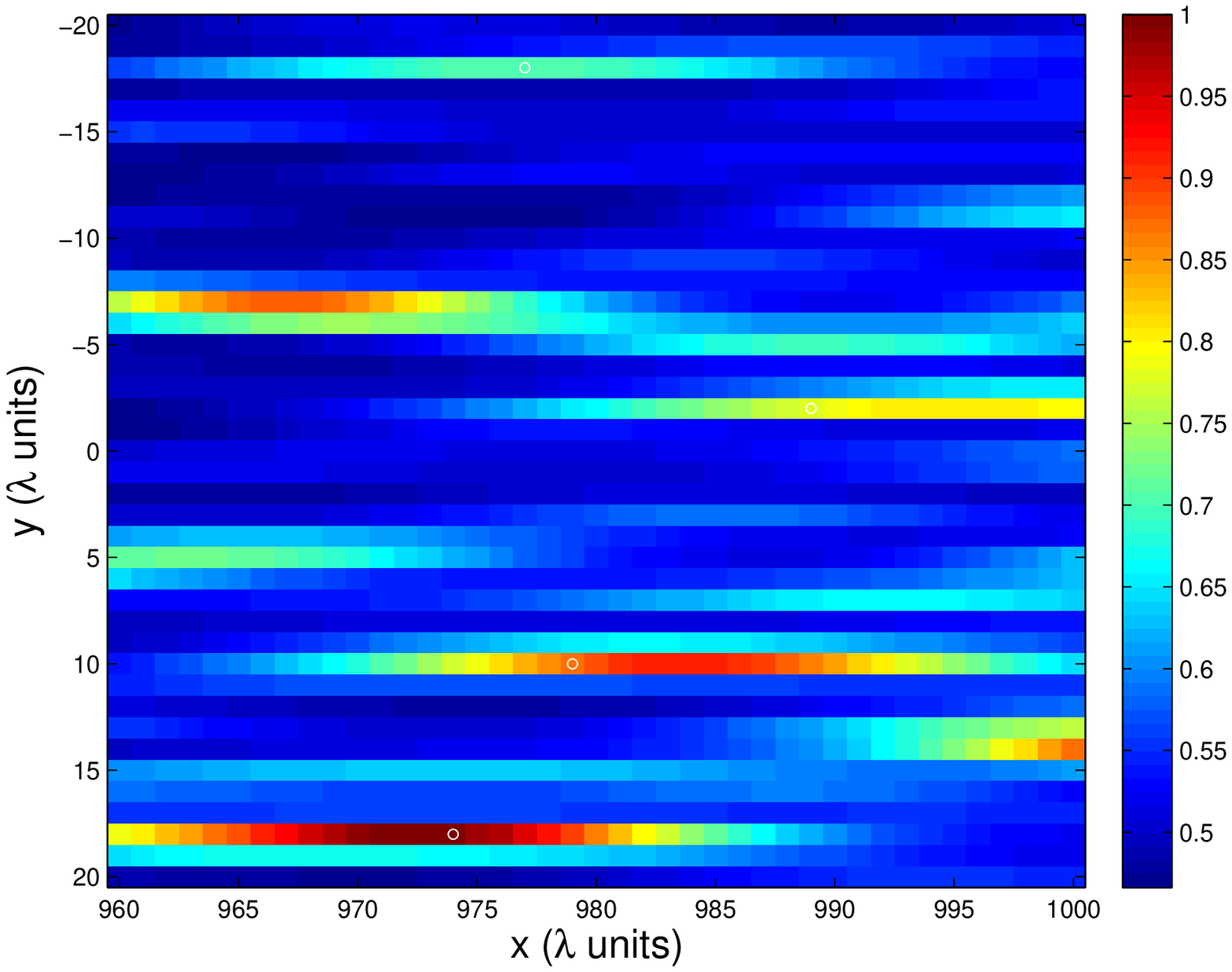}&
\includegraphics[scale=0.17]{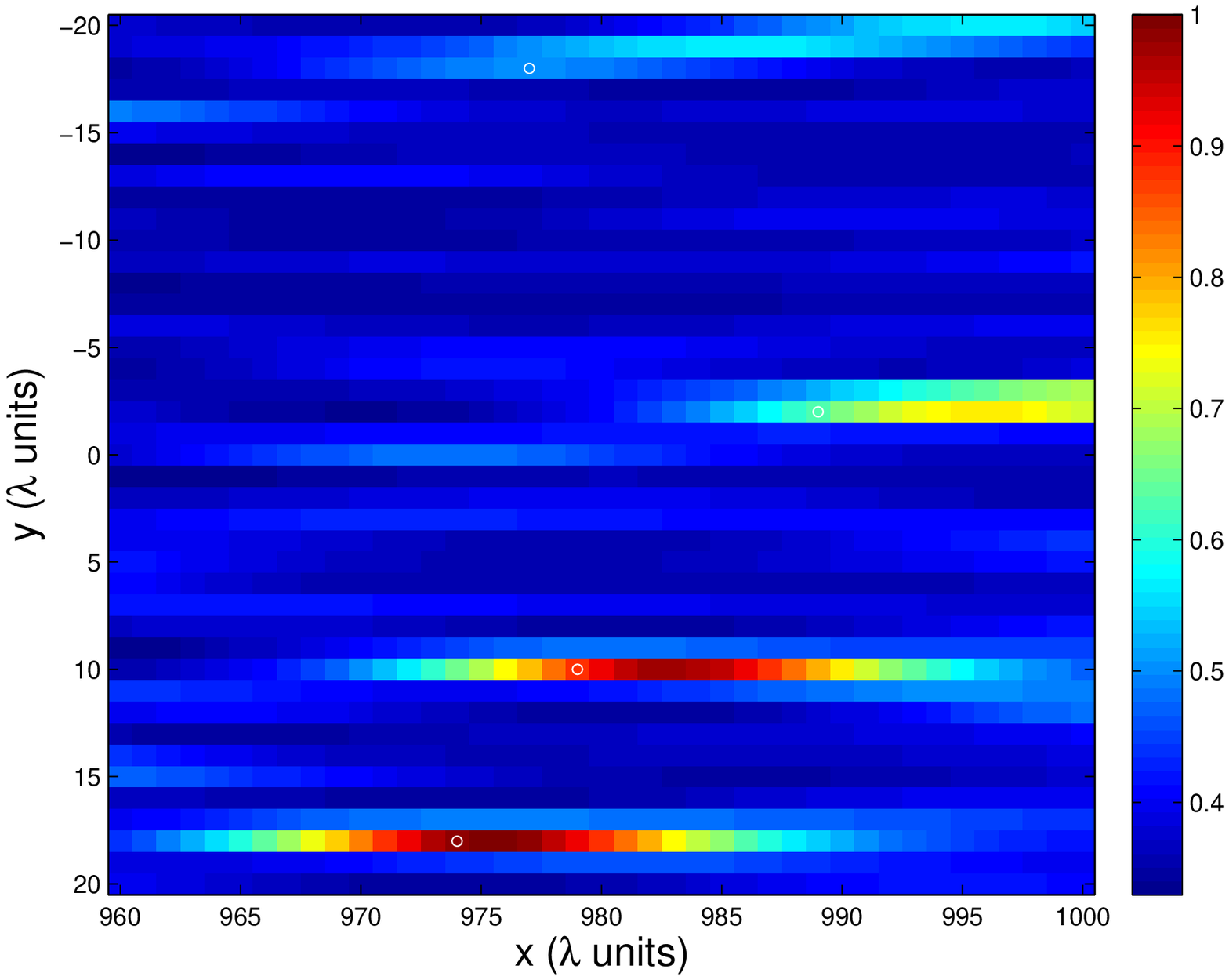}&
\includegraphics[scale=0.17]{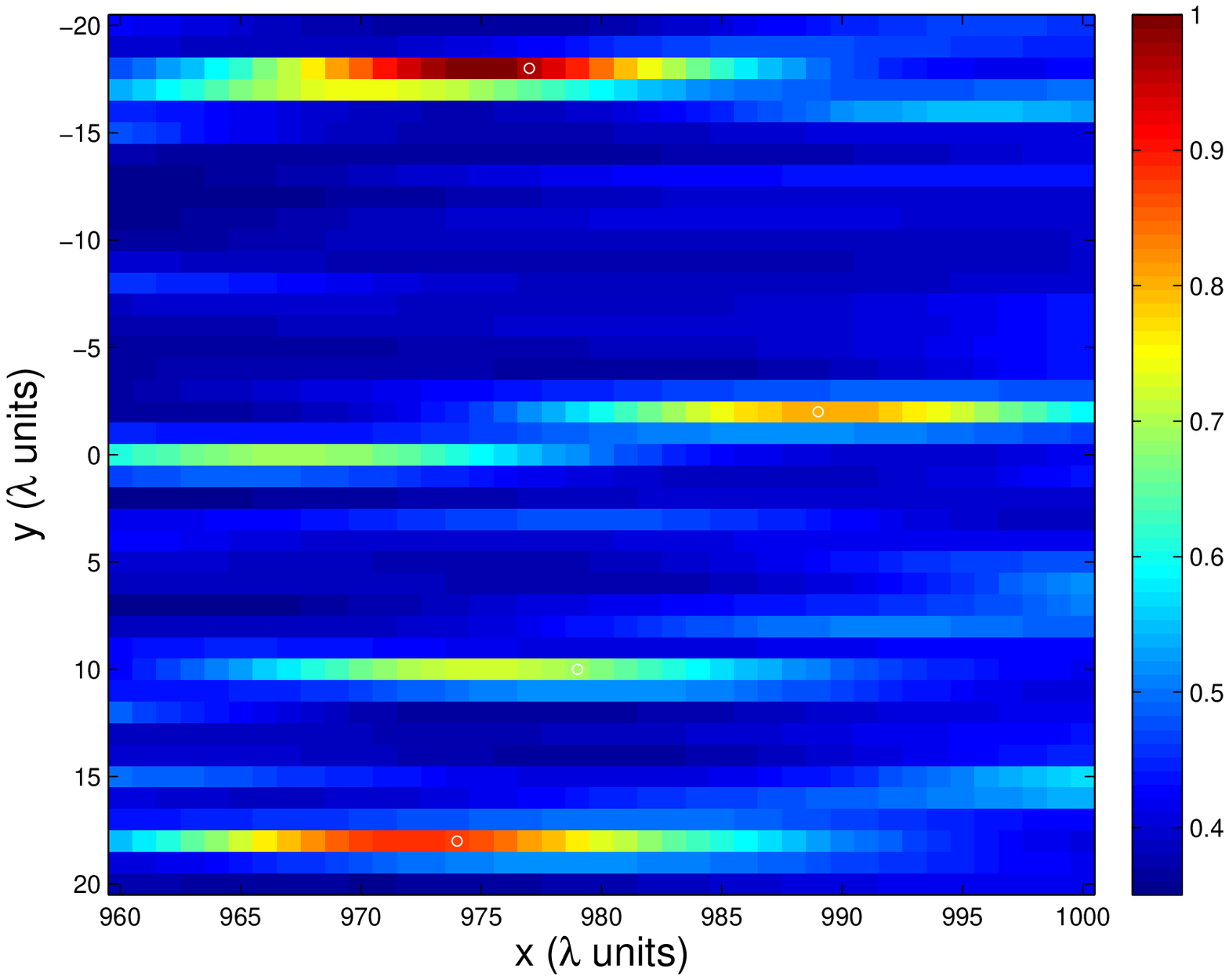}&
\includegraphics[scale=0.17]{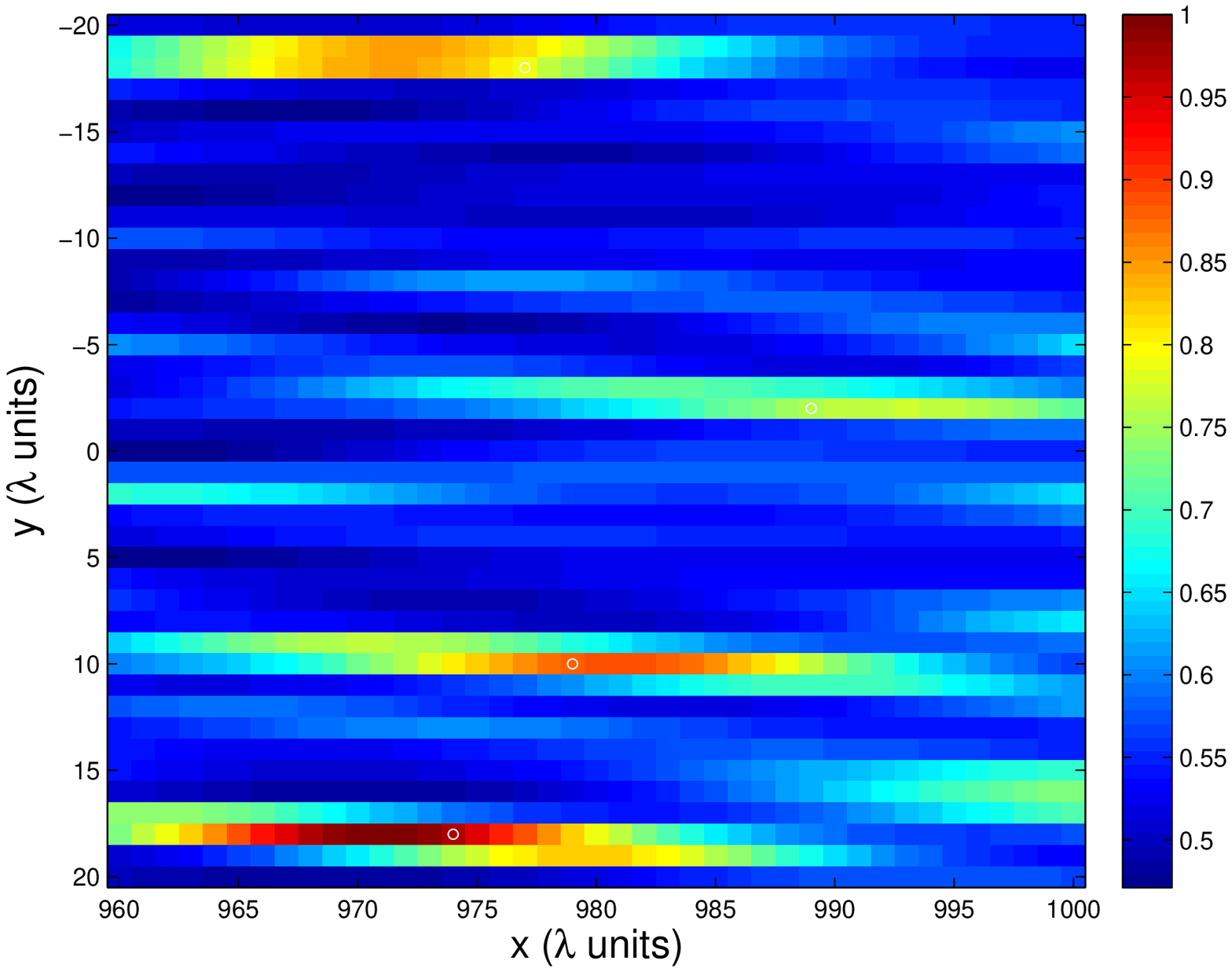}\\
\includegraphics[scale=0.17]{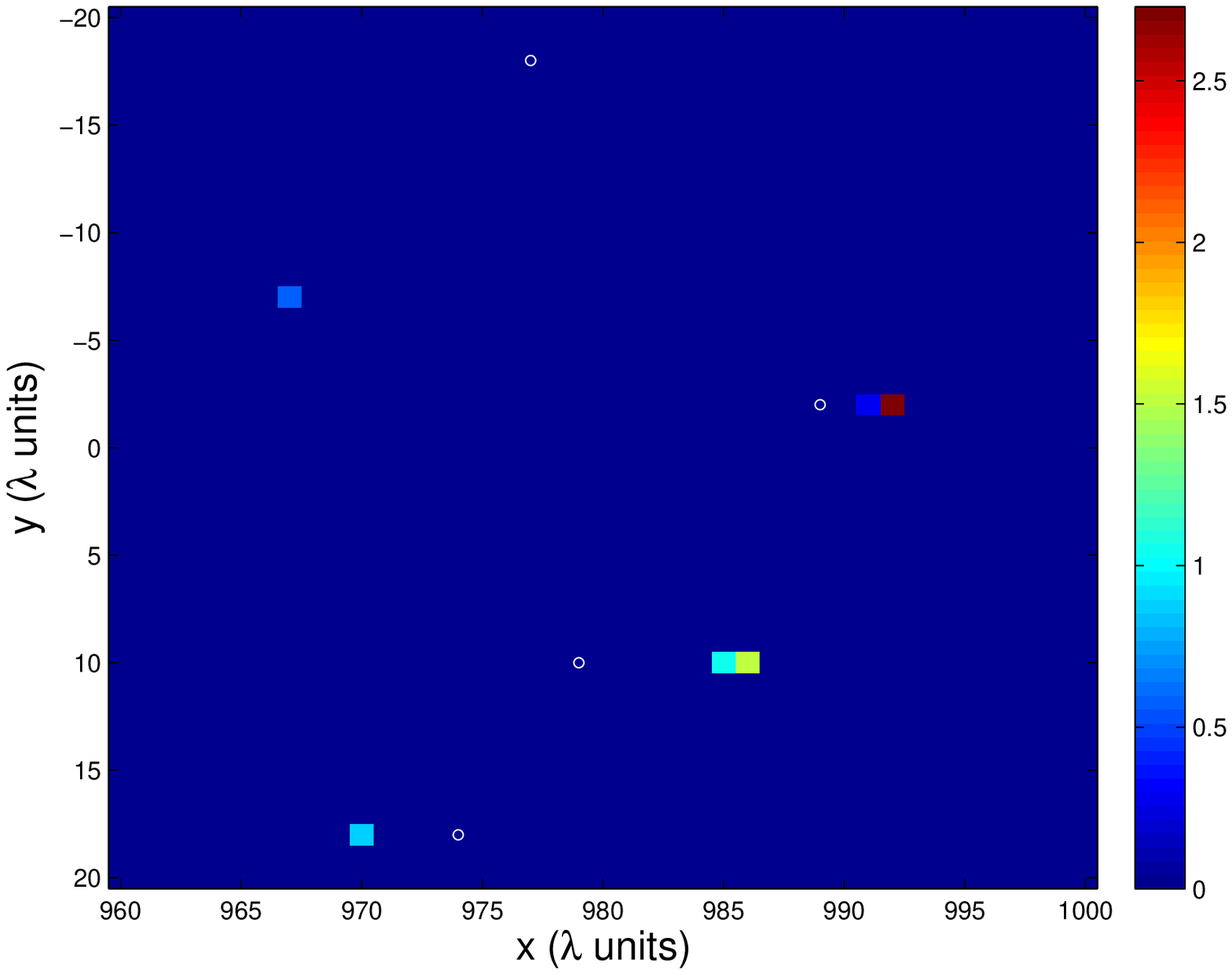}&
\includegraphics[scale=0.17]{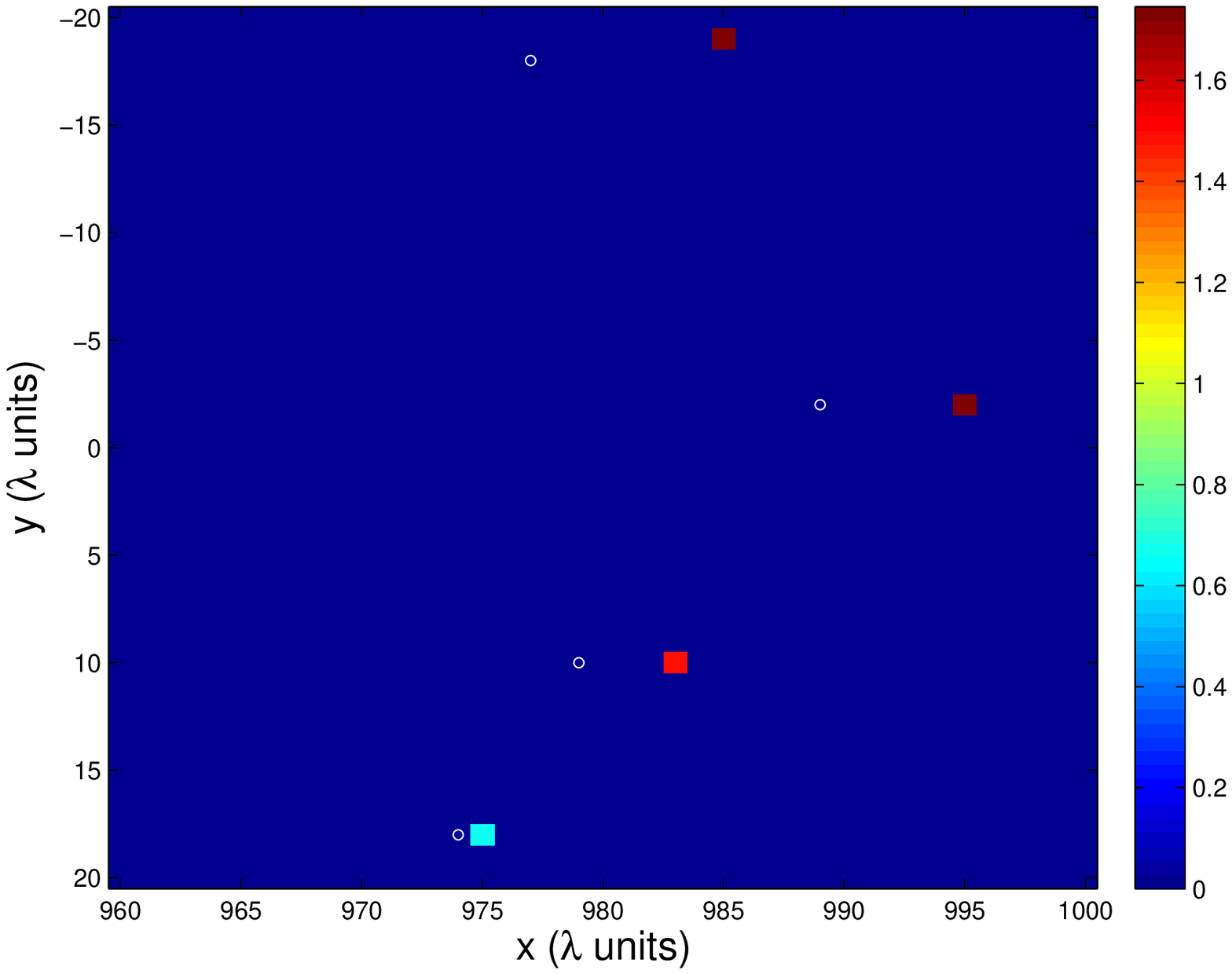}&
\includegraphics[scale=0.17]{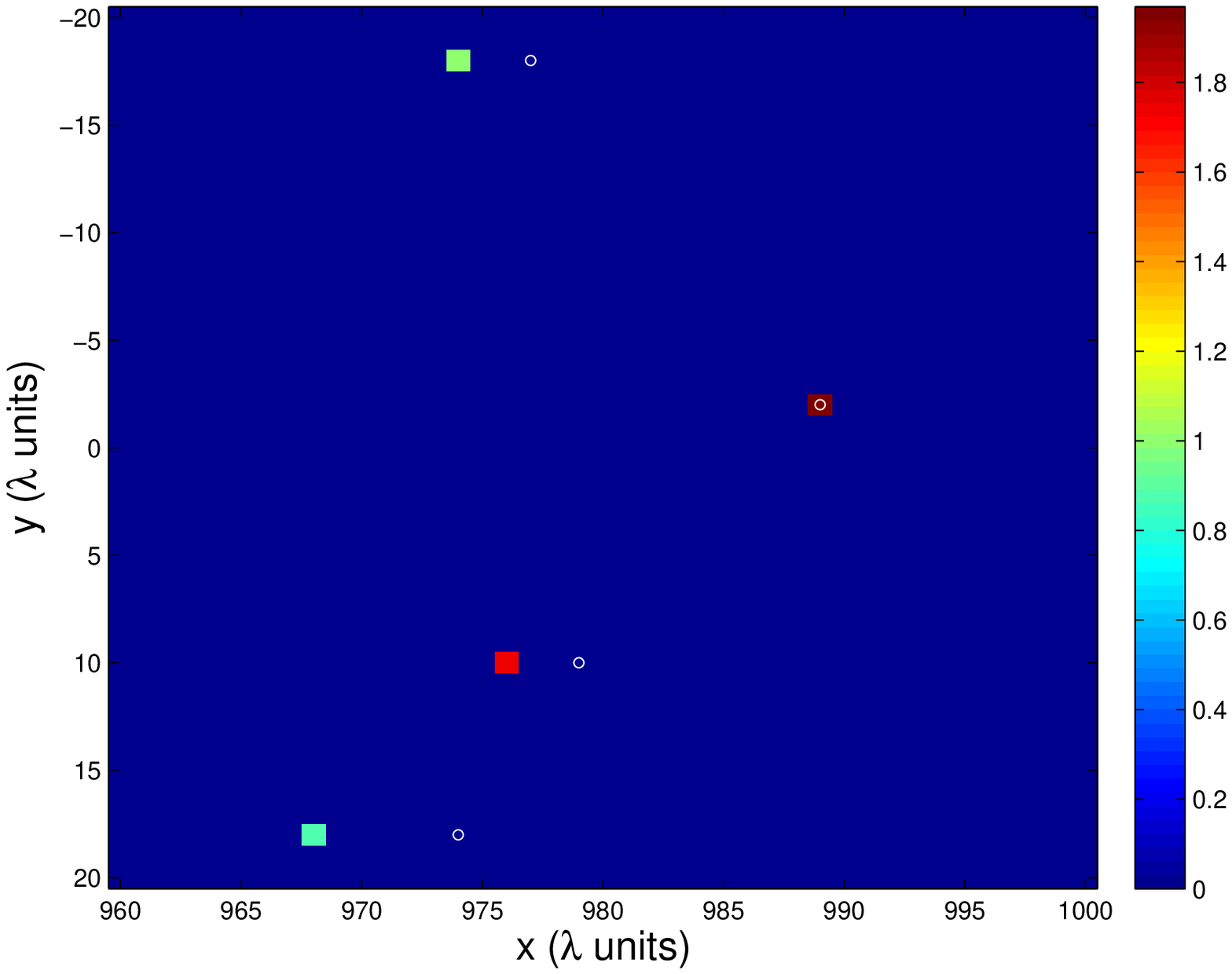}&
\includegraphics[scale=0.17]{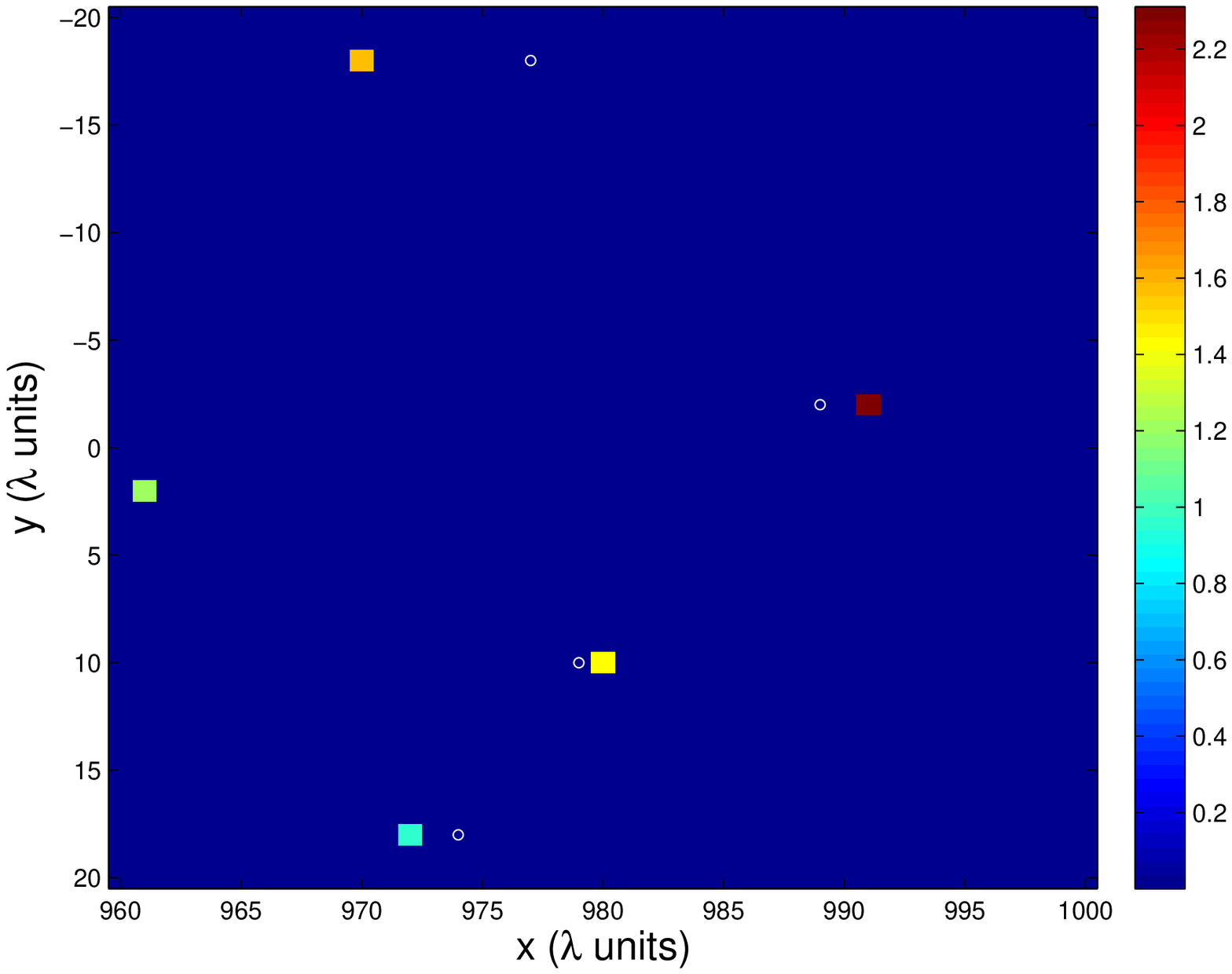}\\
\end{tabular}
\end{center}
\caption{{\bf Small array}. KM (top row) , MUSIC (center row) and hybrid-$\ell_1$ (bottom row) images in the same four realizations of a random medium.}
\label{fig:smallarray}
\end{figure}

These problems are overcome when the array is large ($a= 100 l$), see Figure~\ref{fig:largearray}.
As expected, the resolution of the images obtained by KM and MUSIC improve a lot.
However, KM still fails to image in random media as it produces clutter noise in the images 
from which it is hard or impossible to identify the location of the four scatterers.
On the other hand, MUSIC and the hybrid-$\ell_1$ method are able to recover the sparse solution.
We observe, though, that the performance of the hybrid-$\ell_1$ method is better.
The locations of the four scatterers are found exactly by this method. 

\begin{figure}[t]
\begin{center}
\begin{tabular}{cccc}
\includegraphics[scale=0.17]{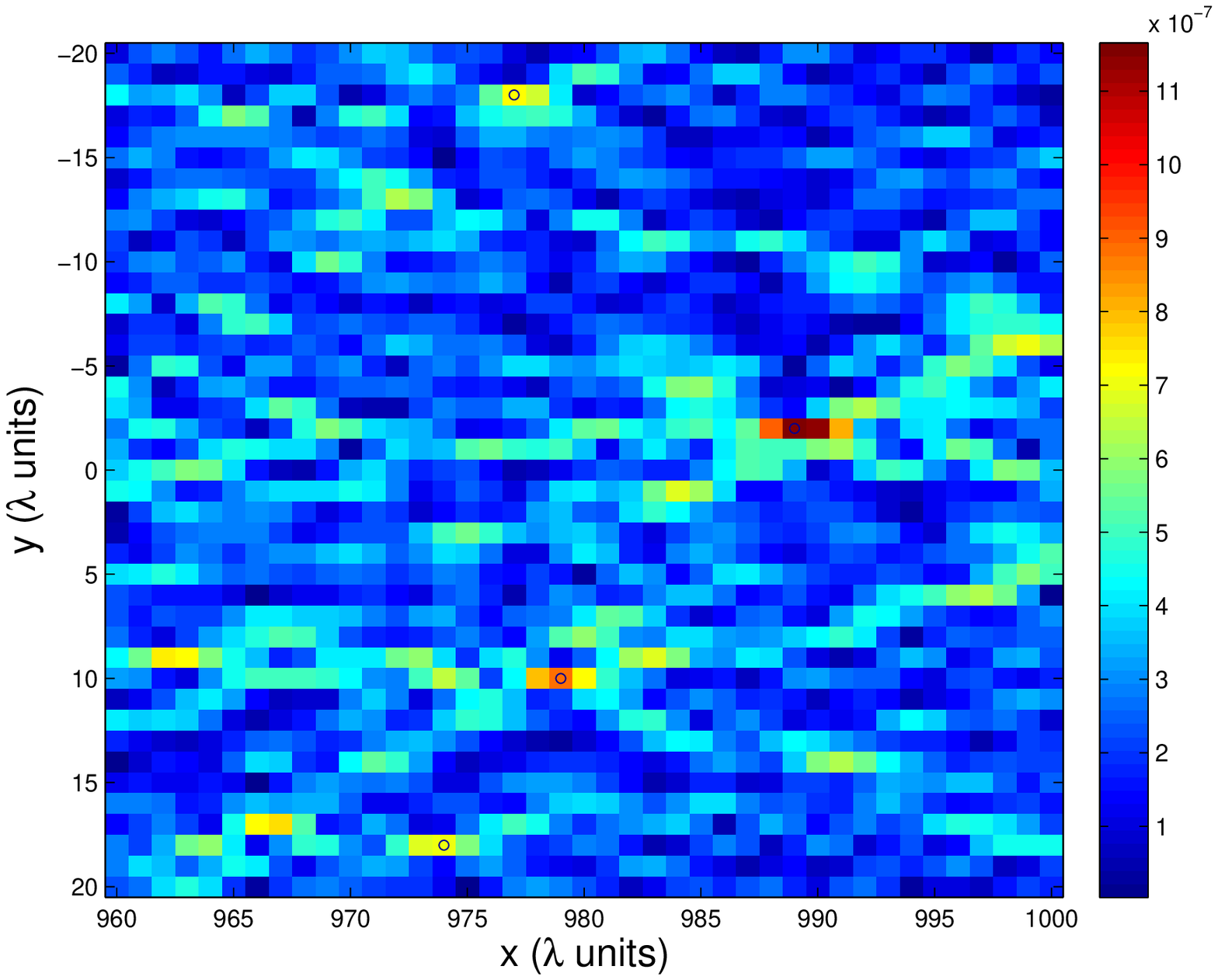}&
\includegraphics[scale=0.17]{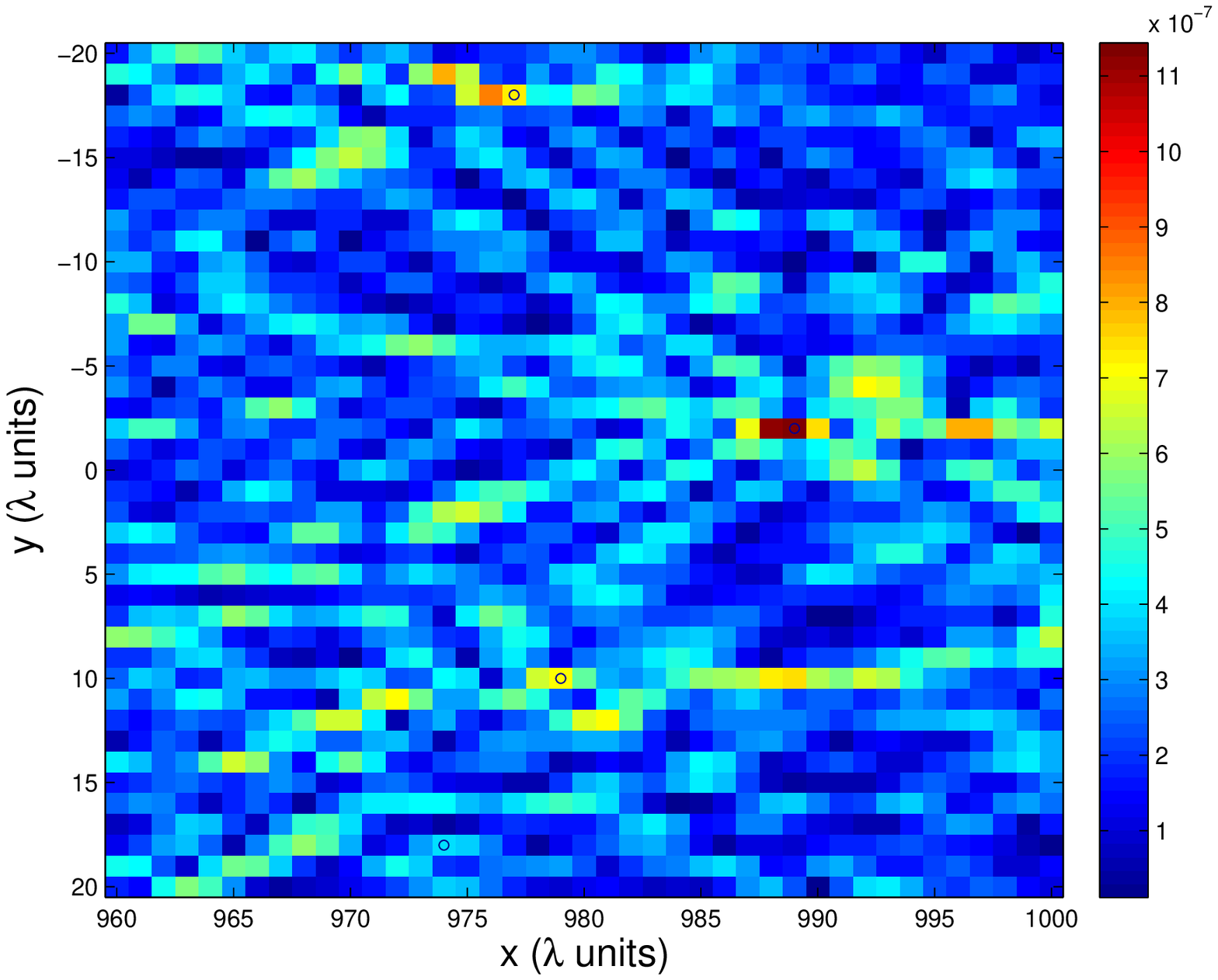}&
\includegraphics[scale=0.17]{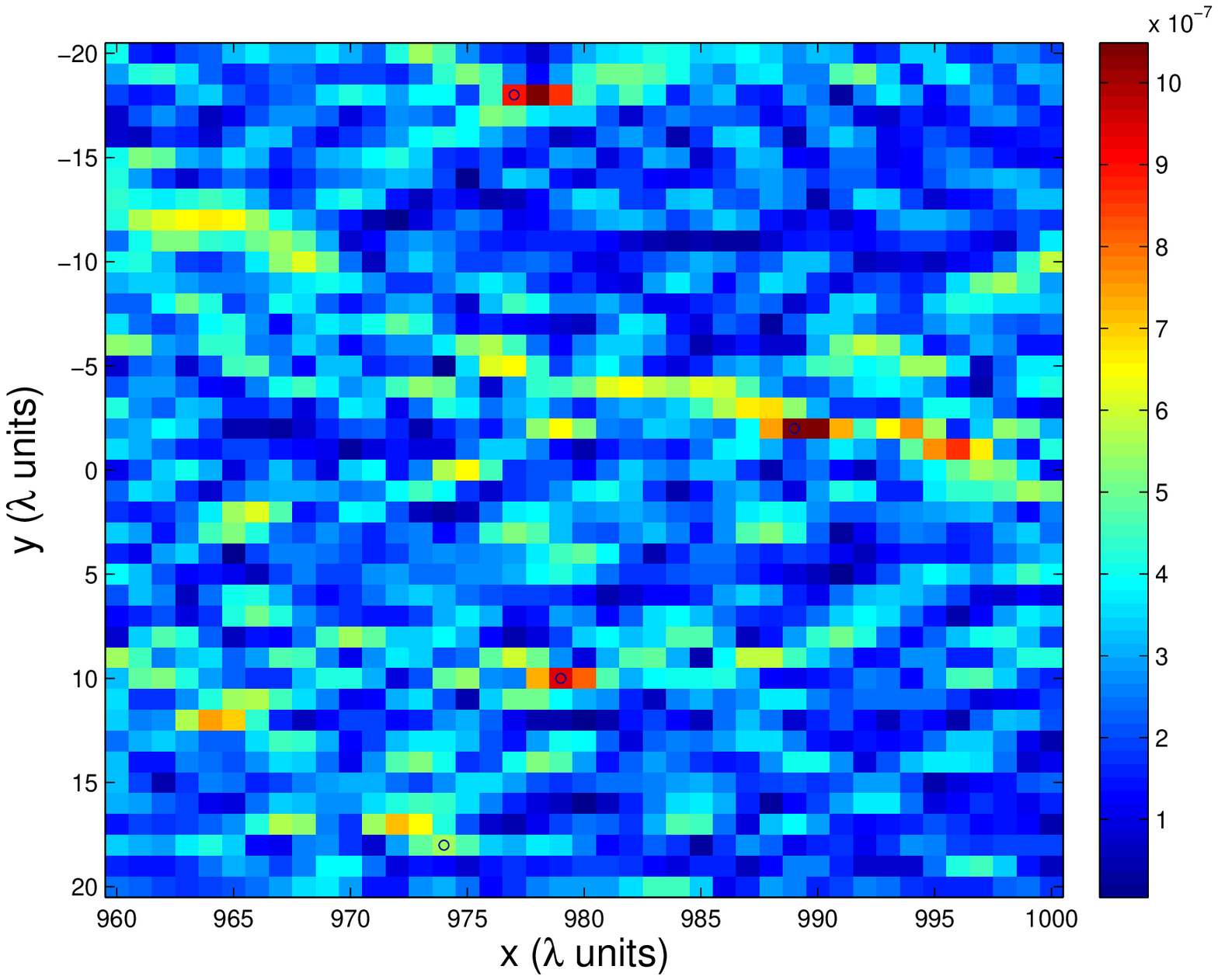}&
\includegraphics[scale=0.17]{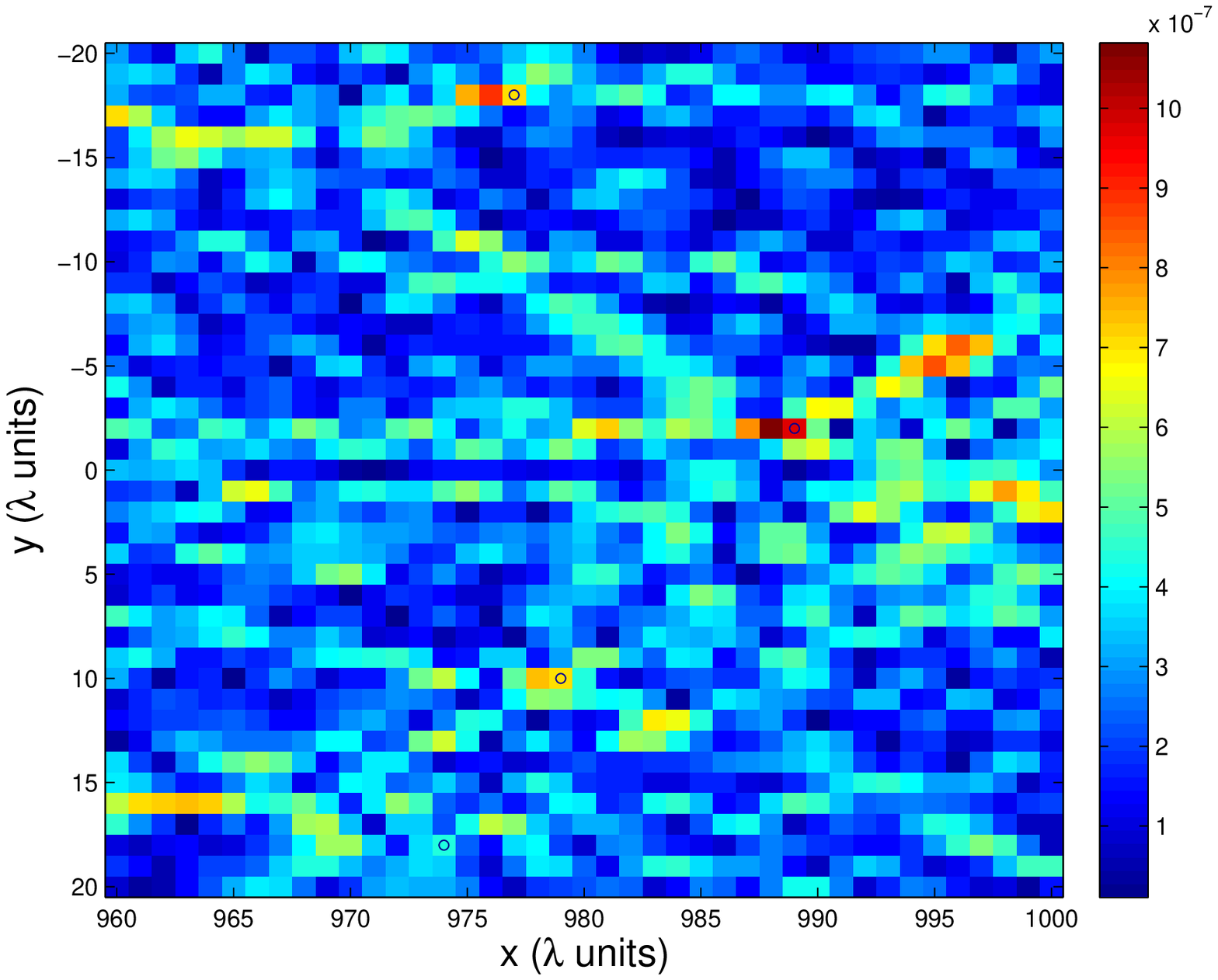}\\
\includegraphics[scale=0.17]{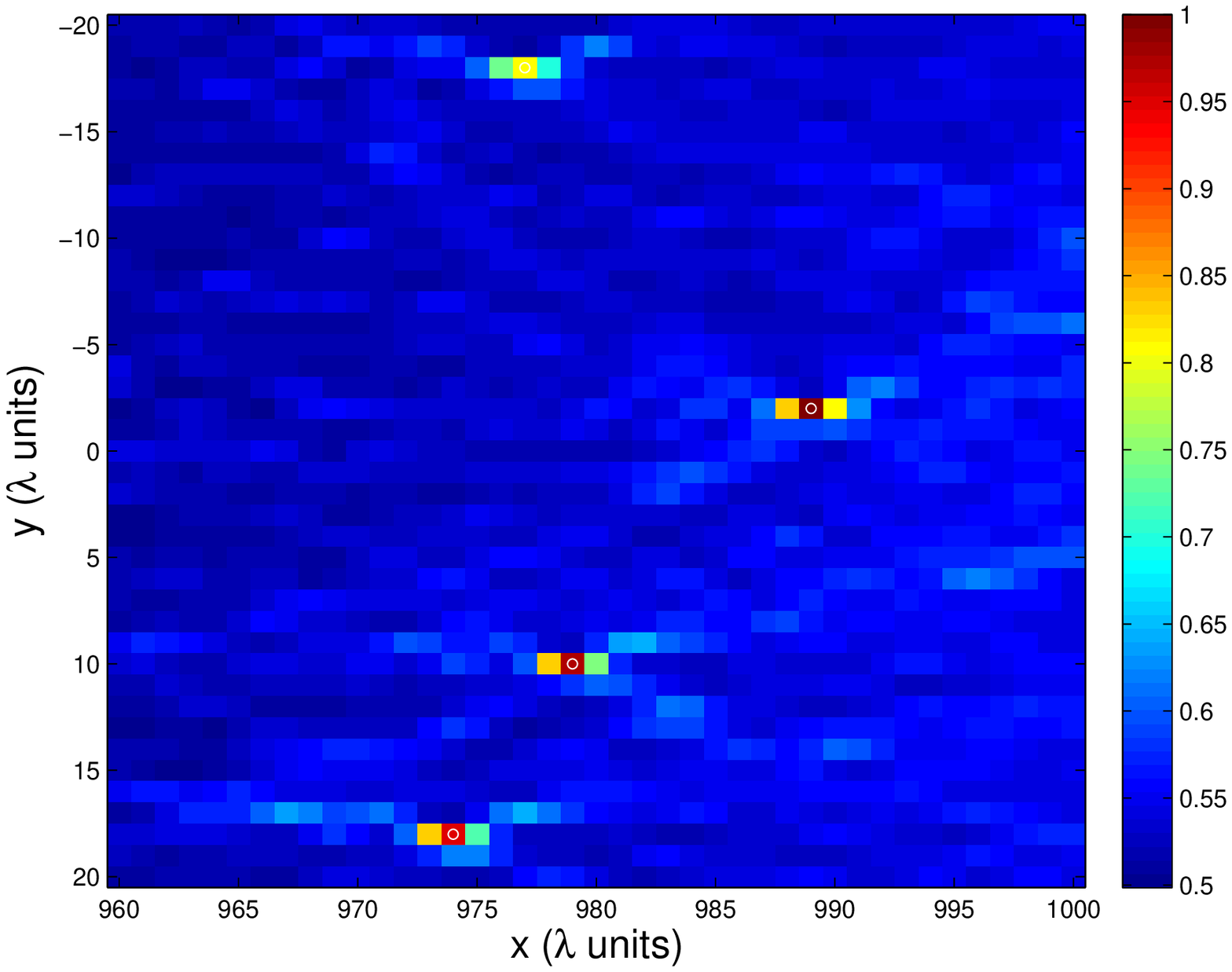}&
\includegraphics[scale=0.17]{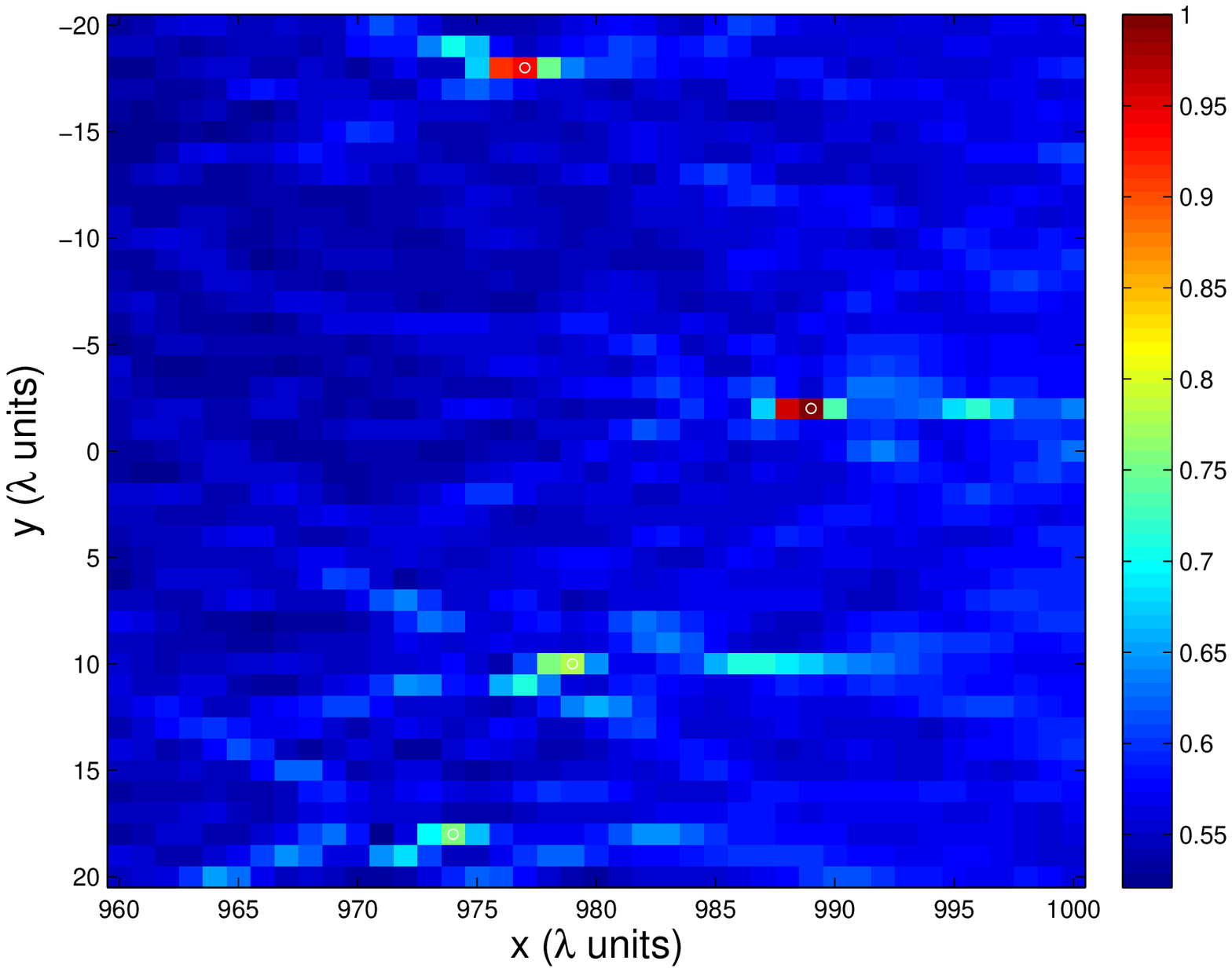}&
\includegraphics[scale=0.17]{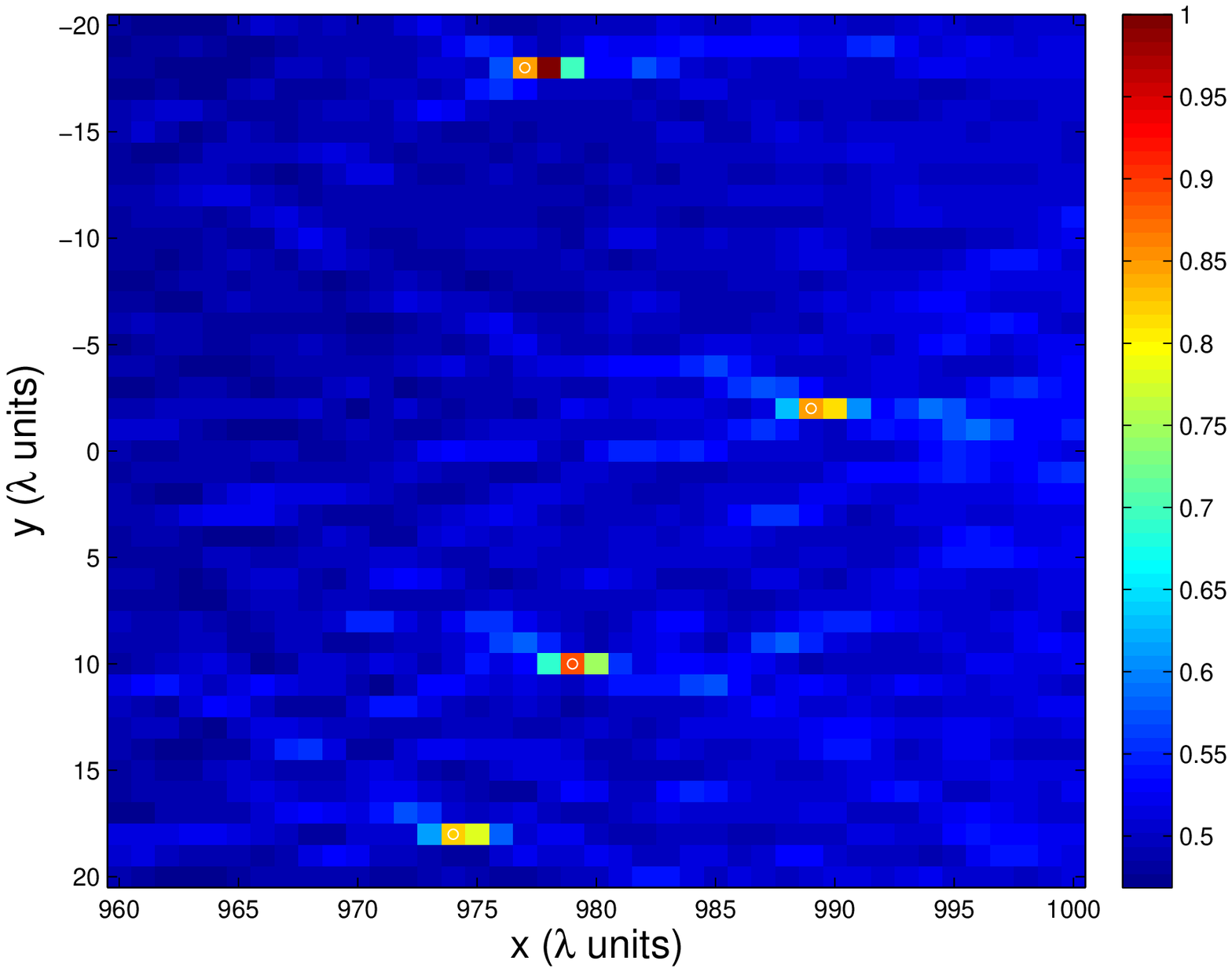}&
\includegraphics[scale=0.17]{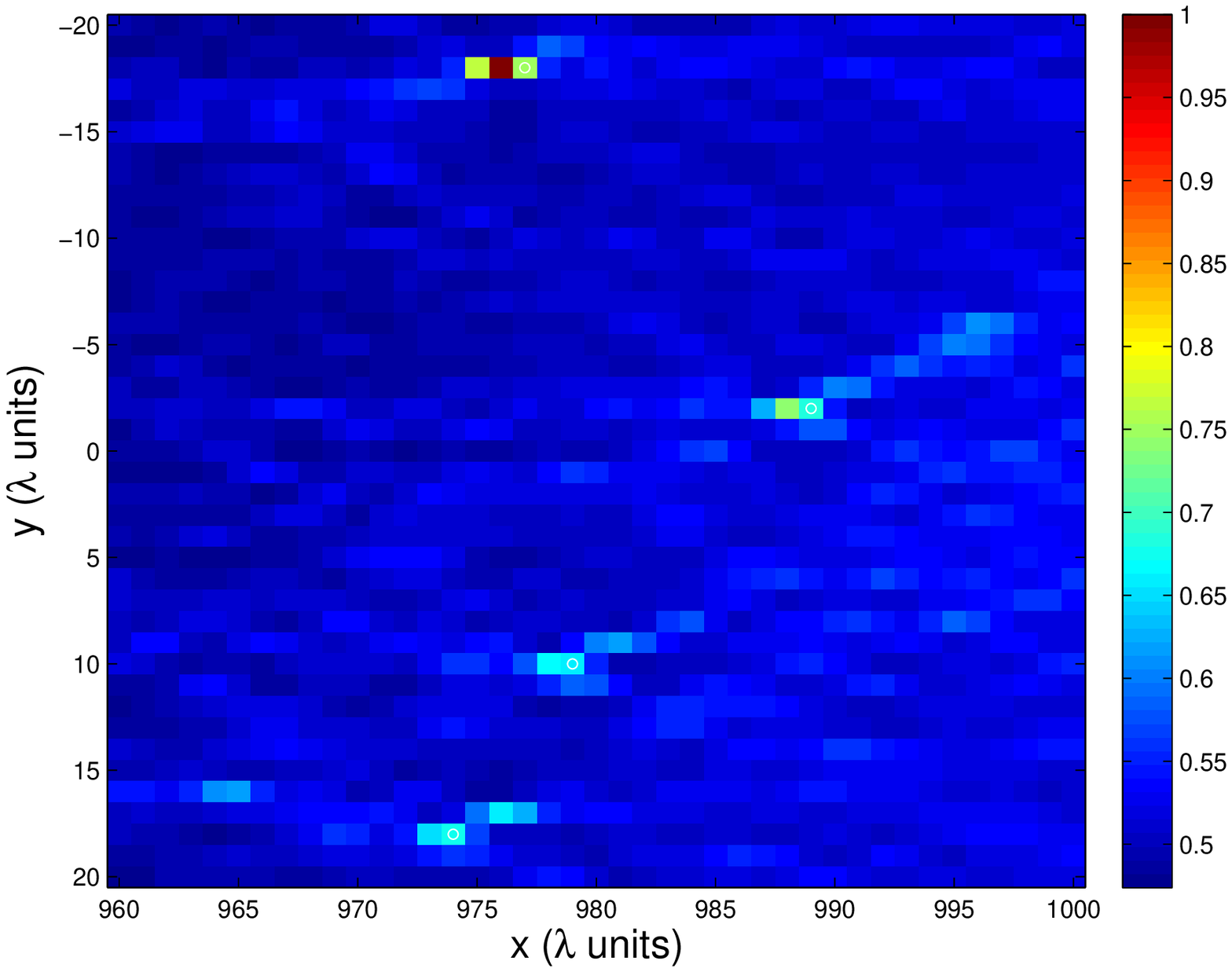}\\
\includegraphics[scale=0.17]{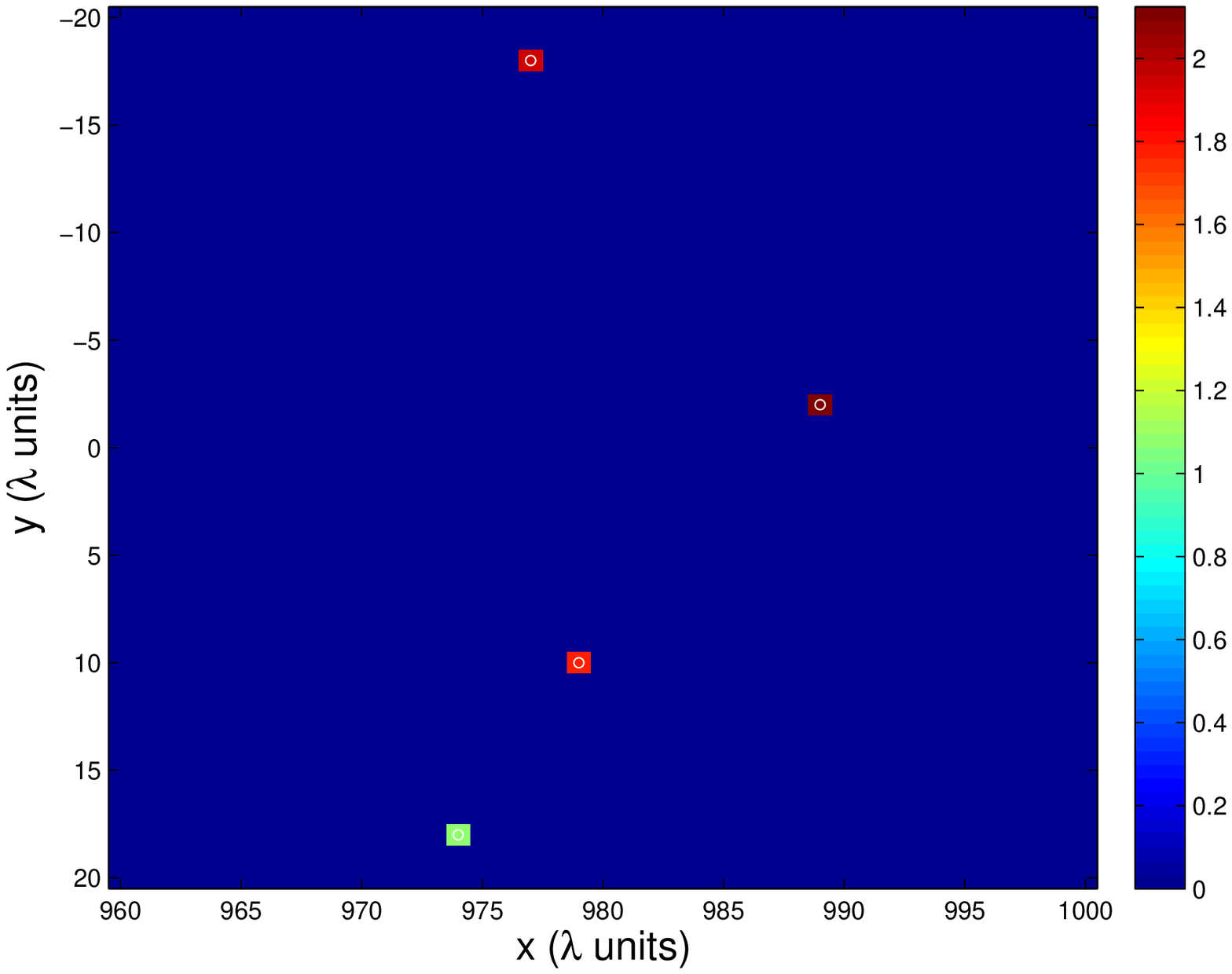}&
\includegraphics[scale=0.17]{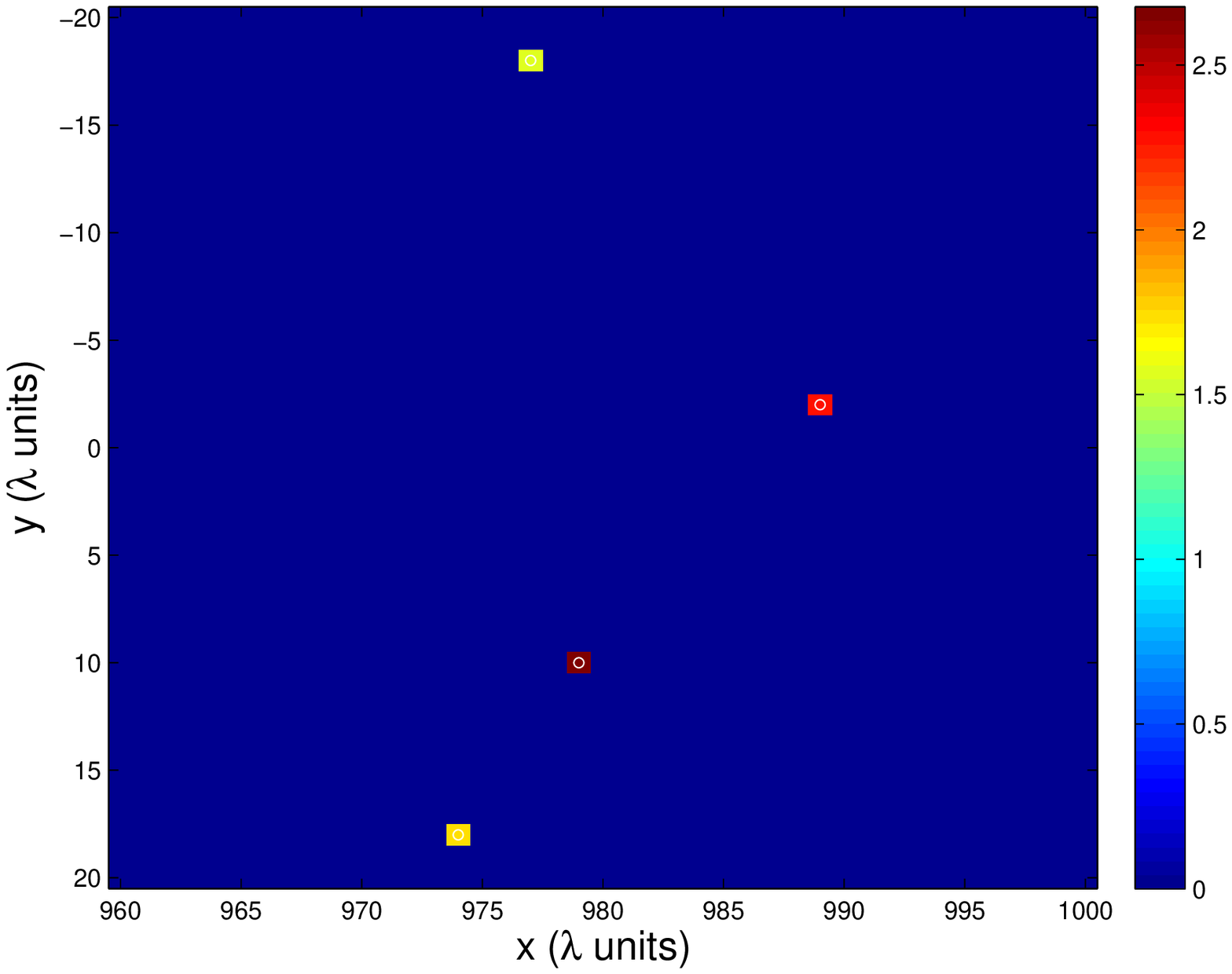}&
\includegraphics[scale=0.17]{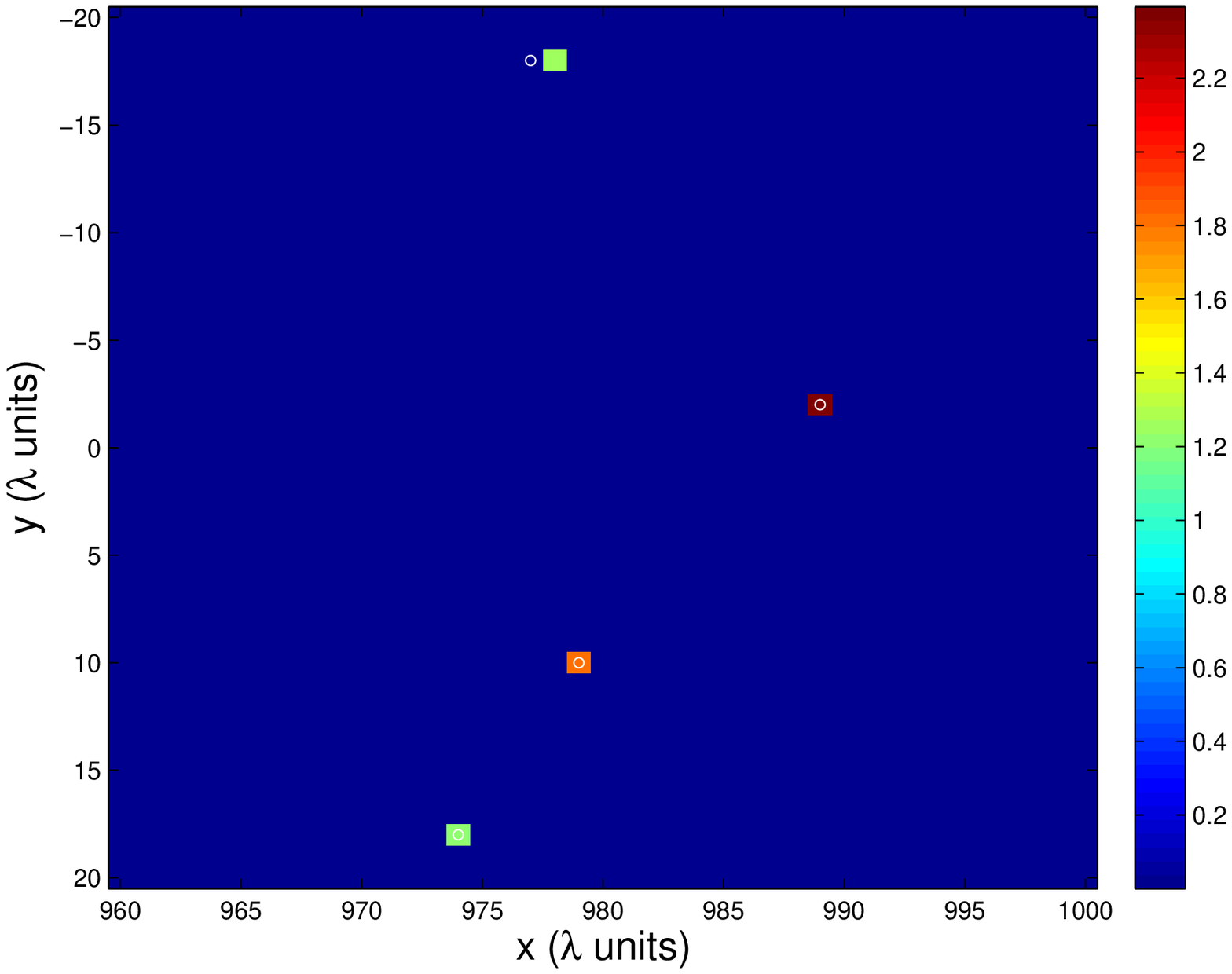}&
\includegraphics[scale=0.17]{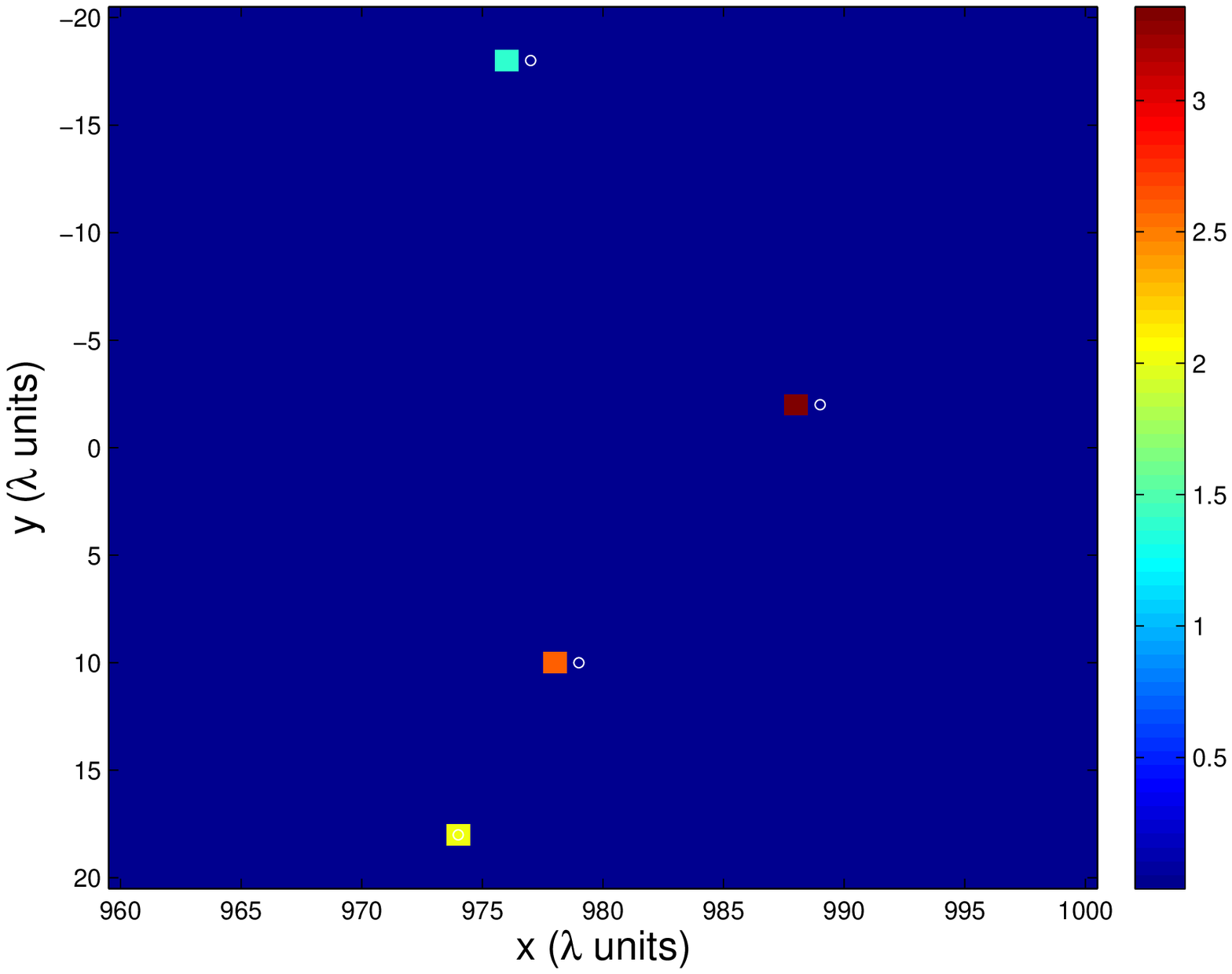}\\
\end{tabular}
\end{center}
\caption{{\bf Large array}. KM (top row) , MUSIC (center row) and hybrid-$\ell_1$ (bottom row) images in the same four realizations of a random medium.}
\label{fig:largearray}
\end{figure}

\section{Conclusions} \label{sec:conclusions}
We present a comprehensive study of optimization based methods applied to narrow band array imaging of localized scatterers. 
We have considered homogeneous and heterogeneous media. When the media is homogeneous but multiple scattering between the scatterers
is important, we give a non-iterative formulation of the nonlinear inverse problem that allows us to determine the locations and
reflectivities of the scatterers non-iteratively using sparsity promoting optimization. We also propose to apply optimal illuminations
to improve the robustness of the imaging methods and the resolution of the images. When multiple scattering is negligible,
the optimization problem becomes linear. In this case, our formulation can be reduced to a hybrid-$\ell_1$ method that uses
the optimal illuminations and $\ell_1$ minimization. This method reduces the dimensionality of the problem, 
filters out the noise, and keeps all the essential properties of $\ell_1$ minimization.

When the media is random, we study the important concept of statistical stability which relates to the robustness
of the imaging methods with respect to different realizations of the random media. Provided the imaging array is
large enough, we show that the hybrid-$\ell_1$ method gives very accurate results and is statistically stable. 

We illustrate the theoretical results with various numerical examples and compared the performance of the proposed
optimization based methods to the widely used Kirchhoff migration and the MUSIC methods.

\section*{Acknowledgments}
AC's work was partially supported by a Hewlett Packard Stanford Graduate Fellowship.
MM's work was partially supported by the Spanish MICINN grant FIS2013-41802-R.
GP's work was partially supported by AFOSR grant FA9550-14-1-0275.

\appendix
\section*{Appendix}
\renewcommand\thesection{\Alph{section}}
\section{Foldy-Lax model} \label{appendix:FLmodel}
Multiple scattering between point-like scatterers is modeled by means of the
Foldy-Lax equations \cite{F45,L51,L52}. Under this  framework,
the scattered wave received at transducer $\vect x_r$ due to a narrow band
signal of angular frequency $\omega$ sent from $\vect x_s$ can be written as the
superposition of all scattered waves from the $M$ scatterers at $\{\vect y_{n_1},\ldots,\vect y_{n_M}\}$ (see Fig. \ref{fig:schematic} (c)), so
\begin{equation}\label{eq:Foldy-Lax}
    \wP(\vect x_r,\vect x_s)=\sum_{j=1}^M\hxi_j^s(\vect x_r; \vect y_{n_1},\ldots,\vect y_{n_M}).
\end{equation}
Here, $\hxi_j^s(\vect x_r;\vect y_{n_1},\ldots,\vect y_{n_M})$
represents the scattered wave observed at $\vect x_r$ due to the emaniting wave from the scatterer at position $\vect y_{n_j}$.
It depends on the positions of all the scatterers and is given by
\begin{equation}\label{eq:scattered-field}
    \hxi_j^s(\vect x_r;\vect y_{n_1},\ldots,\vect y_{n_M})=\alpha_j\wG(\vect x_r,\vect y_{n_j})\hxi_j^e(\vect y_{n_1},\ldots,\vect y_{n_M}),
\end{equation}
where $\hxi_j^e(\vect y_{n_1},\ldots,\vect y_{n_M})$ represents the exciting
field at the scatterer located at $\vect y_{n_j}$. Ignoring the self-interacting
fields, the exciting field at $\vect y_{n_j}$ is equal to the sum of the incident
field $\hxi^{inc}_j:=\hxi^{inc}(\vect y_{n_j},\vect x_s)$ at $\vect y_{n_j}$ and
the scattered fields at $\vect y_{n_j}$ due to all scatterers except for the one
at $\vect y_{n_j}$. Hence, it is given by
\begin{equation}\label{eq:exciting-field}
    \fl\hxi_j^e(\vect y_{n_1},\ldots,\vect y_{n_M})=\hxi^{inc}(\vect y_{n_j},\vect x_s)+
    \sum_{k\neq j}\alpha_k\wG(\vect y_{n_j},\vect y_{n_k})\hxi_k^e(\vect y_{n_1},\ldots,\vect y_{n_M}),
\end{equation}
for $j=1,2,\ldots,M$. This is a self-consistent system of $M$ equations for the
$M$ unknown exciting fields
\[
    \hxi_1^e:=\hxi_{1}^e(\vect y_{n_1},\ldots,\vect y_{n_M}),\ldots, \hxi_M^e:=\hxi_{M}^e(\vect y_{n_1},\ldots,\vect y_{n_M}).
\]
We write \eqref{eq:exciting-field}  in matrix form as
\begin{equation}\label{eq:Foldy-Lax matrix}
\mathbf{Z}_M (\vect\alpha) \, \mathbf{\Phi}^e=\mathbf{\Phi}^{inc} \,\, ,
\end{equation}
where $\mathbf{\Phi}^e=[\hxi^e_1,\ldots,\hxi^e_M]^T$
and $\mathbf{\Phi}^{inc}=[\hxi^{inc}_1,\ldots,\hxi^{inc}_M]^T$ are
vectors whose components are the exciting and incident fields on the $M$ scatterers,
respectively, and
\begin{equation}\label{eq:Zm}
    \big(Z_M(\vect\alpha)\big)_{ij}=
    \begin{cases}
        1 & i=j\\
        -\alpha_j\wG(\vect y_{n_i},\vect y_{n_j}) & i\neq j
    \end{cases}
\end{equation}
is the $M\times M$ Foldy-Lax matrix which depends on the reflectivities $\vect\alpha=[\alpha_1,\ldots,\alpha_M]^T$,
and on the positions of the scatterers $\vect y_{n_j}$.
With the solution of \eqref{eq:Foldy-Lax matrix}, we use \eqref{eq:scattered-field}
and \eqref{eq:Foldy-Lax} to compute the scattered data received at the array.

We extend the $M\times M$ matrix $\mathbf{Z}_M(\vect\alpha)$, defined only for pairwise combinations of
scatterers at $\vect y_{n_j}$, to a larger $K\times K$ matrix 
\begin{equation} \label{eq:Z}
    \big(Z(\vect\rho_0)\big)_{ij}=\begin{cases}
        1,& i=j\\
        -\rho_{0j}\wG(\vect y_{i},\vect y_{j}),& i\neq j
    \end{cases}
\end{equation}
which includes all pairwise combinations of the $K$ grid points $\vect y_j$
in the IW. Using \eqref{eq:Z} and \eqref{eq:sensingmatrix}, we can write the response
matix as \eqref{eq:responsematrix0}.

\section{Proof of results in \S\ref{sec:random medium}}\label{appendix:proofs}
In order to prove Proposition~\ref{prop:2nd moment}, we need to show the following lemma first.
\begin{lemma}\label{lem:autocorrelation}
    Assume the autocorrelation function $R$ of the random field in $\mL_2(\mR_+)$ with derivative $\dot{R}$ such that
    $\frac{\dot{R}(t)}{t}\in\mL(\mR_+)$ and both $R$ and $\dot{R}$ decay exponentially. Then
    \[\int_0^\infty\frac{\dot{R}(t)}{t}\,\rmd t<0.\]
\end{lemma}
\begin{proof}
    Since $R\in\mL_2(\mR_+)$, the Fourier transform exists, i.e.
\[\widehat{R}(|\vect k|)=\int_{\mR^3}R(|\vect x|)\rme^{\rmi\vect k\cdot\vect x}\,\rmd\vect x\, .\]
Moreover, $\widehat R(\cdot)\ge0$. Using spherical coordinate $(r,\phi,\theta)$ we obtain
\begin{eqnarray*}
\widehat R(|\vect k|)&=&\int_0^{2\pi}\int_0^\pi\int_0^\infty R(r)r^2\sin\phi\,\rme^{\rmi|\vect k|r\cos\phi}\,\rmd r\,\rmd\phi\,\rmd\theta\\
&=&2\pi\int_0^\infty R(r)r^2\int_0^\pi\sin\phi\,\rme^{\rmi|\vect k|r\cos\phi}\,\rmd r\,\rmd\phi\\
&=&4\pi\int_0^\infty R(r)r\frac{\sin|\vect k|r}{|\vect k|}\,\rmd r
\end{eqnarray*}
Since the correlation function $R(\cdot)$ is defined only on $[0,+\infty)$, we extend it
symmetrically to the whole real line by $R(-\cdot)=R(\cdot)$ for any value in
$(-\infty,0)$. Then, the above integral becomes
\begin{eqnarray*}
\widehat R(|\vect k|)=\frac{2\pi}{|\vect k|}\int_{-\infty}^\infty R(r)r\sin|\vect k|r\,\rmd r
=-\frac{2\pi}{|\vect k|^3}\int_{-\infty}^\infty \dot{R}(r)(\sin|\vect k|r-|\vect k|r\cos|\vect k|r)\,\rmd r,
\end{eqnarray*} where the last equality holds due to integration by parts and
the exponential decay of $R(\cdot)$. Then, with Riemann Lebesgue lemma, the following
limit is
\begin{eqnarray*}
\lim_{K\rightarrow\infty}\int_0^K|\vect k|^3\widehat R(|\vect k|)\,\rmd|\vect k|
=-2\pi\int_{-\infty}^\infty\frac{\dot{R}(r)}{r}\,\rmd r.
\end{eqnarray*} 
The symmetric extension of $R(\cdot)$ implies that $\dot{R}(r)=-\dot{R}(-r)$ when $r<0$. Therefore,
\[
\int_0^\infty r^3\widehat{R}(r)\,\rmd r=\lim_{K\rightarrow\infty}\int_0^K|\vect k|^3\widehat R(|\vect k|)\,\rmd|\vect k|=-4\pi\int_0^\infty\frac{\dot{R}(r)}{r}\,\rmd r\, ,
\]
which implies that the integral on the right hand side is not positive.
\end{proof}

Note that the conditions on the autocorrelation function $R$ in Lemma~\ref{lem:autocorrelation}
are not restrictive. Many autocorrelation functions satisfy those conditions, as for example,
$R(|\vect x|)=\rme^{-\frac{1}{2}|\vect x|^2}$ for a Gaussian random field and the power law
function $R(|\vect x|)=(1+|\vect x|)\rme^{-|\vect x|}$. With this lemma, we can now prove Proposition~\ref{prop:2nd moment}.

\begin{proof}[Proof of Proposition~\ref{prop:2nd moment}]
    To simplify the notation, we first define the path integral of
the random process function $\mu(\cdot)$ as
\[\nu(\vect x,\vect y)=\int_0^1\mu\left(\frac{\vect x}{l}+s\frac{\vect y-\vect x}{l}\right)\,\rmd s,\]
and 
\[
\vartheta(\vect x,\vect y)=|\vect x-\vect y|\nu(\vect x,\vect y).\]
We need to compute moment estimation
\[\mE(\wG(\vect x,\vect y_1)\overline{\wG(\vect x,\vect y_2)})
=\wG_0(\vect x,\vect y_1)\overline{\wG_0(\vect x,\vect y_2)}E\, ,\]
where
\begin{equation}\label{eq:exponent}
E=\mE\left(\rme^{-\rmi\kappa\sigma(|\vect x-\vect y_1|\nu(\vect x,\vect y_1)-|\vect x-\vect y_2|\nu(\vect x,\vect y_2))}\right)
=\rme^{-\frac{1}{2}\kappa^2\sigma^2\mE\left(\vartheta(\vect x,\vect y_1)-\vartheta(\vect x,\vect y_2)\right)^2}.
\end{equation}
Because $|\vect y_1-\vect y_2|\ll L$ we can write
\[|\vect x-\vect y_2|=|\vect x-\vect y_1|+\frac{(\vect x-\vect y_1)^T}{|\vect x-\vect y_1|}(\vect y_2-\vect y_1)+o(|\vect y_2-\vect y_1|),\]
and then
\[\nu(\vect x,\vect y_2)=\nu(\vect x,\vect y_1)+\nabla_{\vect y_1}^T\nu(\vect x,\vect y_1)(\vect y_2-\vect y_1)+o(|\vect y_2-\vect y_1|).\]
Dropping terms of order $2$ or higher, we estimate the expectation in the exponent as
\begin{eqnarray*}
\mE\Big(\vartheta(\vect x,\vect y_1)-\vartheta(\vect x,\vect y_2)\Big)^2\\
&\hskip-1cm\approx\mE\left(\nu(\vect x,\vect y_1)\frac{(\vect x-\vect y_1)^T}{|\vect x-\vect y_1|}(\vect y_2-\vect y_1)
    +|\vect x-\vect y_1|\nabla_{\vect y_1}^T\nu(\vect x,\vect y_1)(\vect y_2-\vect y_1)\right)^2.
\end{eqnarray*}
For simplicity, we consider $\vect y_1$ and $\vect y_2$ on the same plane, at a distance $L$ from point $\vect x = (0,0,0)$,
so  $\vect y_1=(0,0,L)$ and$\vect y_2=(\xi,\eta,L)$. Then, $(\vect y_2-\vect y_1)^T(\vect x-\vect y_1)=0$ and
the variance is reduced to the estimation of 
\begin{equation}\label{eq:expectation}
    \begin{split}
        \mE&\left(\nabla^T_{\vect y}\nu(\vect x,\vect y)(\vect y'-\vect y)\right)^2\\
        &=\xi^2\mE\left(\frac{\partial}{\partial y_1}\nu(\vect x,\vect y)\right)^2+\eta^2\mE\left(\frac{\partial}{\partial y_2}\nu(\vect x,\vect y)\right)^2+2\xi\eta\mE\left(\frac{\partial}{\partial y_1}\nu(\vect x,\vect y)\frac{\partial}{\partial y_2}\nu(\vect x,\vect y)\right)\\
        &=\left(\xi^2\frac{\partial}{\partial y_1}\frac{\partial}{\partial y'_1}\mE(\nu(\vect x,\vect y)\nu(\vect x,\vect y'))+\eta^2\frac{\partial}{\partial y_2}\frac{\partial}{\partial y'_2}\mE(\nu(\vect x,\vect y)\nu(\vect x,\vect y'))\right.\\
        &\hskip2cm\left.\left.+2\xi\eta\frac{\partial}{\partial y_1}\frac{\partial}{\partial y'_2}\mE(\nu(\vect x,\vect y)\nu(\vect x,\vect y'))\right)\right|_{\vect y'=\vect y},
    \end{split}
\end{equation}
where $\frac{\partial}{\partial y_i}$ means the derivative with respect to the $i^\mathrm{th}$ component of vector $\vect y$.

Denote $\tilde{y'}=[\xi,\eta]^T$ and $\tilde y=[0,0]^T$, such that $\vect y=[\tilde{y}^T, L]^T$ and
$\vect y'=[{\tilde{y'}}^T, L]^T$, and let $\varsigma=\sqrt{\frac{1}{l^2}|s\tilde y-s'\tilde y'|^2+\frac{L^2}{l^2}(s-s')^2}$.
We compute the derivatives of the expectations 
\begin{equation*}
    \mE(\nu(\vect x,\vect y)\nu(\vect x,\vect y'))=\int_0^1\int_0^1R(\varsigma)\,\rmd s'\,\rmd s
\end{equation*}
in \eqref{eq:expectation} as follows:
\begin{equation*}
    \frac{\partial}{\partial y_1}\mE(\nu(\vect x,\vect y)\nu(\vect x,\vect y'))=\int_0^1\int_0^1\dot{R}(\varsigma)
\frac{s}{\varsigma l^2}(sy_1-s'y'_1)\,\rmd s'\,\rmd s \, ,
\end{equation*}

\begin{equation*}
    \begin{split}
        \frac{\partial^2}{\partial y'_1\partial y_1}&\mE(\nu(\vect x,\vect y)\nu(\vect x,\vect y'))\\
        &=\int_0^1\int_0^1\left(\dot{R}(\varsigma)\frac{ss'(sy_1-s'y'_1)^2}{\varsigma^3l^4}-\dot{R}(\varsigma)\frac{ss'}{\varsigma l^2}-\ddot{R}(\varsigma)\frac{ss'}{\varsigma^2 l^4}(sy_1-s'y'_1)^2 \right)\,\rmd s'\,\rmd s\, ,
    \end{split}
\end{equation*}

\begin{equation*}
\frac{\partial^2}{\partial y_1\partial y'_2}\mE(\nu(\vect x,\vect y)\nu(\vect x,\vect y'))=\int_0^1\int_0^1\frac{ss'}{\varsigma^2l^4}\left(\frac{\dot{R}(\varsigma)}{\varsigma}-\ddot{R}(\varsigma)\right)(sy_1-s'y'_1)(sy_2-s'y'_2)\,\rmd s'\,\rmd s.
\end{equation*}
Taking $\vect y=\vect y'=(0,0,L)$ in the above second order derivatives, we obtain that
\begin{equation*}
    \mE\Big(\frac{\partial}{\partial y_1}\nu(\vect x,\vect y)\frac{\partial}{\partial y_2}\nu(\vect x,\vect y)\Big)=0 \, ,
\end{equation*}
so that
\begin{equation*}
    \mE\left(\nabla^T_{\vect y}\nu(\vect x,\vect y)(\vect y'-\vect y)\right)^2=(\xi^2+\eta^2)\mE\Big(\frac{\partial}
    {\partial y_1}\nu(\vect x,\vect y)\Big)^2\, .
\end{equation*}
Next, we compute
\begin{eqnarray*}
\mE\left(\frac{\partial}{\partial y_1}\nu(\vect x,\vect y)\right)^2&=&-\int_0^1\int_0^1\dot{R}\left(\frac{L|s-s'|}{l}\right)\frac{ss'}{lL|s-s'|}\,\rmd s'\,\rmd s\\
&=&\frac{2l^2}{3L^4}\int_0^{L/l}sR(s)\rmd s-\frac{2}{3lL}\int_0^{L/l}\frac{\dot{R}(t)}{t}\,\rmd t+\frac{2}{3L^2}R\left(\frac{L}{l}\right)-\frac{R(0)}{L^2}\\
&\approx&\frac{1}{L^2}\left(-1-\frac{2L}{3l}\int_0^\infty\frac{\dot{R}(t)}{t}\,\rmd t\right),
\end{eqnarray*}
where the last approximation is based on the condition $l\ll L$ and $R(\cdot)$ has exponential decay with normalization $R(0)=1$.
Now, let $a_e^2=\sigma^2L^4\mE\left(\frac{\partial}{\partial y_1}\nu(\vect x,\vect y)\right)^2$. We then have
\begin{equation}\label{eq:difference square expectation}
\mE(\vartheta(\vect x,\vect y_1)-\vartheta(\vect x,\vect y_2))^2\approx|\vect x-\vect y_1|^2\mE(\nabla^T_{\vect y}\nu(\vect x,\vect y_1)(\vect y_2-\vect y_1))^2\approx\frac{|\vect y_1-\vect y_2|^2}{\sigma^2L^2}a_e^2 \, ,
\end{equation}
and using \eqref{eq:exponent}
\begin{equation*}
\mE(\wG(\vect x,\vect y_1)\overline{\wG(\vect x,\vect y_2)})=\wG_0(\vect x,\vect y_1)\overline{\wG_0(\vect x,\vect y_2)}E
\approx\wG_0(\vect x,\vect y_1)\overline{\wG_0(\vect x,\vect y_2)}\rme^{-\frac{\kappa^2a_e^2}{2L^2}|\vect y_1-\vect y_2|^2},
\end{equation*}
i.e. the moment estimate \eqref{eq:2nd moment}.
\end{proof}

\begin{proof}[Proof of Corollary~\ref{cor:variance estimate}]
        The result is a direct application of Proposition~\ref{prop:2nd moment}.
Let $\mathcal{C}=\frac{1}{(4\pi|\vect x-\vect y_1|)^2(4\pi|\vect x-\vect y_2|)^2}$.
Using \eqref{eq:2nd moment}, we have
\begin{equation*}
    \begin{split}
        \mE&\left|\wG(\vect x,\vect y_1)\overline{\wG(\vect x,\vect y_2)}-\mE\left(\wG(\vect x,\vect y_1)\overline{\wG(\vect x,\vect y_2)}\right)\right|^2\\
           &\approx\mE\left|\wG_0(\vect x,\vect y_1)\overline{\wG_0(\vect x,\vect y_2)}\left(\rme^{\rmi\kappa\sigma\left(\vartheta(\vect x,\vect y_1)-\vartheta(\vect x, \vect y_2)\right)}-\rme^{-\frac{1}{2}\kappa^2\frac{a_e^2}{L^2}|\vect y_1-\vect y_2|^2}\right)\right|^2\\
           &=\mathcal{C}\mE\big(1+\rme^{-\frac{\kappa^2a_e^2}{L^2}|\vect y_1-\vect y_2|^2}-\rme^{-\frac{\kappa^2a_e^2}{2L^2}|\vect y_1-\vect y_2|^2}(\rme^{\rmi\kappa\sigma(\vartheta(\vect x, \vect y_1)-\vartheta(\vect x, \vect y_2))}+\rme^{-\rmi\kappa\sigma(\vartheta(\vect x, \vect y_1)-\vartheta(\vect x, \vect y_2))})\big)\\
           &=\mathcal{C}\big(1+\rme^{-\frac{\kappa^2a_e^2}{L^2}|\vect y_1-\vect y_2|^2}-2\rme^{-\frac{\kappa^2a_e^2}{2L^2}|\vect y_1-\vect y_2|^2}\te^{-\frac{1}{2}\kappa^2\sigma^2\mE(\vartheta(\vect x, \vect y_1)-\vartheta(\vect x,\vect y_2))^2}\big)\\
           &\approx\mathcal{C}\big(1-\rme^{-\kappa^2\frac{a_e^2}{L^2}|\vect y_1-\vect y_2|^2}\big),
    \end{split}
\end{equation*}
where we use the estimate \eqref{eq:difference square expectation} in the last step.
\end{proof}

\begin{proof}[Proof of Proposition~\ref{prop:statistical stability}]
    For simplicity, we use the same configuration as in the proof of Proposition~\ref{prop:2nd moment},
    where $\vect y_1=[0,0,L]^T$ and $\vect y_2=[\xi,\eta,L]^T$. We first look at the denominator of the
    ratio in \eqref{eq:statistical stability}
\begin{equation}\label{norm of random green vector}
\|\vect\wg(\vect y_1)\|^2=\sum_{j=1}^{N}\Big|\wG(\vect x_j,\vect y_1)\Big|^2
=\sum_{j=1}^{N}\Big|\wG_0(\vect x_j,\vect y_1)\Big|^2=\|\vect\wg_0(\vect y_1)\|^2.
\end{equation}
Using the same approach as in the proof in \cite{CMP13}, under continuous limit, we have 
\begin{eqnarray*}
\|\vect\wg(\vect y_1)\|^2\approx\frac{1}{(4\pi h)^2}\int_{\Omega(\vect x)}\frac{\rmd\vect x}{|\vect x-\vect y|^2}
=\frac{1}{16\pi h^2}\log\Big(1+\frac{a^2}{4L^2}\Big).
\end{eqnarray*}
On the other hand, the numerator can be computed as follows
\begin{equation*}
    \begin{split}
        \mE&\left|\vect\wg^\ast(\vect y_1)\vect\wg(\vect y_2)-\mE\left(\vect\wg^\ast(\vect y_1)\vect\wg(\vect y_2)\right)\right|^2\\
           &=\mE\left|\sum_{j=1}^{N}\wG_0(\vect x_j,\vect y_1)\overline{\wG_0(\vect x_j,\vect y_2)}\left(\rme^{\rmi\kappa\sigma(\vartheta(\vect x_j,\vect y_1)-\vartheta(\vect x_j,\vect y_2))}-\rme^{-\frac{\kappa^2a^2_e}{2L^2}|\vect y_1-\vect y_2|^2}\right)\right|^2\\
           &=\sum_{j,j'=1}^{N}\wG_0(\vect x_j,\vect y_1)\overline{\wG_0(\vect x_j,\vect y_2)}\overline{\wG_0(\xjp,\vect y_1)}\wG_0(\xjp,\vect y_2)\\
           &\ \ \ \ \ \times\mE\big(\rme^{\rmi\kappa\sigma(\vartheta(\vect x_j,\vect y_1)-\vartheta(\vect x_j,\vect y_2)-\vartheta(\xjp,\vect y_1)+\vartheta(\xjp,\vect y_2))}\\
           &\ \ \ \ \ \ \ \ \ -\rme^{-\frac{\kappa^2a^2_e}{2L^2}|\vect y_1-\vect y_2|^2}\,\rme^{-\rmi\kappa\sigma(\vartheta(\xjp,\vect y_1)-\vartheta(\xjp,\vect y_2))}\\
           &\ \ \ \ \ \ \ \ \ -\rme^{-\frac{\kappa^2a^2_e}{2L^2}|\vect y_1-\vect y_2|^2}\,\rme^{-\rmi\kappa\sigma(\vartheta(\vect x_j,\vect y_1)-\vartheta(\vect x_j,\vect y_2))}+\rme^{-\frac{\kappa^2a_e^2}{L^2}|\vect y_1-\vect y_2|^2}\big)\\
           &=\sum_{j,j'=1}^{N}\wG_0(\vect x_j,\vect y_1)\overline{\wG_0(\vect x_j,\vect y_2)}\overline{\wG_0(\xjp,\vect y_1)}\wG_0(\xjp,\vect y_2)\\
           &\ \ \ \ \ \times\rme^{-\frac{\kappa^2a_e^2}{L^2}|\vect y_1-\vect y_2|^2}\left(\rme^{\kappa^2\sigma^2\mE(\vartheta(\vect x_j,\vect y_1)-\vartheta(\vect x_j,\vect y_2))(\vartheta(\xjp,\vect y_1)-\vartheta(\xjp,\vect y_2))}-1\right).
    \end{split}
\end{equation*}
From the above expression, we can see the variance of $\vect\wg^\ast(\vect y_1)\vect\wg(\vect y_2)$
is close to the denominator up to the factor
\begin{equation}\label{df}
    \rme^{-\frac{\kappa^2a_e^2}{L^2}|\vect y_1-\vect y_2|^2}\left(\rme^{\kappa^2\sigma^2\mE(\vartheta(\vect x_j,\vect y_1)-\vartheta(\vect x_j,\vect y_2))(\vartheta(\xjp,\vect y_1)-\vartheta(\xjp,\vect y_2))}-1\right).
\end{equation}
The expectation
\[
    \mE(\vartheta(\vect x_j,\vect y_1)-\vartheta(\vect x_j,\vect y_2))(\vartheta(\xjp,\vect y_1)-\vartheta(\xjp,\vect y_2))
\]
is nonzero only when there is strong correlation between the path starting from $\vect x_j$ and $\xjp$. This is controlled by the 
correlation length $l$ of the random medium. When those paths are within $l$, the value of \eqref{df} is reduced to
$
    1-\rme^{-\frac{\kappa^2a_e^2}{L^2}|\vect y_1-\vect y_2|^2},
$
and otherwise \eqref{df} is equal to zero. Thus in the continuous limit, the value of numerator can be approximated by
\begin{equation}\label{variance of random green vector}
    \begin{split}
        \mE|\vect\wg^\ast(\vect y_1)\vect\wg(\vect y_2)&-\mE\left(\vect\wg^\ast(\vect y_1)\vect\wg(\vect y_2)\right)|^2\approx\frac{1}{h^4}\Big(1-\rme^{-\frac{\kappa^2a_e^2}{L^2}|\vect y_1-\vect y_2|^2}\Big)\\
                                                       &\times\int_{\Omega(\vect x)}\wG_0(\vect x,\vect y_1)\overline{\wG_0(\vect x,\vect y_2)}\,\rmd\vect x\overline{\int_{B(\vect0,l)}\wG_0(\vect x',\vect y_1)\overline{\wG_0(\vect x',\vect y_2)}\,\rmd\vect x'},
    \end{split}
\end{equation}
where $B(\vect0,l)$ is a ball centered at $\vect0$ with radius $l$ within which factor \eqref{df}
is not equal to zero. 

Therefore the ratio on the left handside of \eqref{eq:statistical stability} is bounded by
\begin{multline*}
    \left|\frac{\mE|\vect\wg^\ast(\vect y_1)\vect\wg(\vect y_2)-\mE(\vect\wg^\ast(\vect y_1)\vect\wg(\vect y_2))|^2}{\|\vect\wg(\vect y_1)\|^2\|\vect\wg(\vect y_2)\|^2}\right|\\
    \begin{aligned}
        &\approx\frac{256\pi^2\Big(1-\rme^{-\frac{\kappa^2a_e^2}{L^2}|\vect y_1-\vect y_2|^2}\Big)}{\log^2\Big(1+\frac{a^2}{4L^2}\Big)}\left|\int_{\Omega(\vect x)}\wG_0(\vect x,\vect y_1)\overline{\wG_0(\vect x,\vect y_2)}\,\rmd\vect x\right|\\
        &\ \ \ \ \ \ \times\left|\int_{B(\vect0,l)}\wG_0(\vect x',\vect y_1)\overline{\wG_0(\vect x',\vect y_2)}\,\rmd\vect x'\right|\\
        &\le\frac{\Big(1-\rme^{-\frac{\kappa^2a_e^2}{L^2}|\vect y_1-\vect y_2|^2}\Big)}{\log^2\Big(1+\frac{a^2}{4L^2}\Big)}\int_{B(\vect0,l)}\frac{\rmd\vect x}{|\vect x-\vect y|^2}\int_{\Omega(\vect x)}\frac{\rmd\vect x}{|\vect x-\vect y|^2}\\
        &\approx\frac{l^2\Big(1-\rme^{-\frac{\kappa^2a_e^2}{L^2}|\vect y_1-\vect y_2|^2}\Big)}{L^2\log\Big(1+\frac{a^2}{4L^2}\Big)}.
    \end{aligned}
\end{multline*}
When size of array $a$ increases, we have
$\frac{\mE|\vect\wg^\ast(\vect y_1)\vect\wg(\vect y_2)-\mE(\vect\wg^\ast(\vect y_1)\vect\wg(\vect y_2))|^2}
{\|\vect\wg(\vect y_1)\|^2\|\vect\wg(\vect y_2)\|^2}$ goes to $0$ logarithmically.
\end{proof}

\begin{remark}\label{rem:paraxial}
    {\rm The result in Proposition~\ref{prop:statistical stability} holds for any regime no matter $a\ll L$ or $L\lessapprox a$.
When $a\ll L$, i.e. in the paraxial regime, we can use parabolic approximation to compute the
approximate value for the ratio on the lefthand side of \eqref{eq:statistical stability}.
Let $\vect x=[x_1,x_2,0]^T$.
Then we have 
\[|\vect x-\vect y_1|=\sqrt{x_1^2+x_2^2+L^2}\approx L+\frac{x_1^2+x_2^2}{2L}\] and
\[|\vect x-\vect y_2|=\sqrt{L^2+(x_1-\xi)^2+(x_2-\eta)^2}=L+\frac{(x_1-\xi)^2+(x_2-\eta)^2}{2L}.\]
Since $|\vect y_1-\vect y_2|\le a\ll L$, we approximate in denominator
\[|\vect x-\vect y_1|\approx|\vect x-\vect y_2|\approx L,\] 
and for the two integrals in \eqref{variance of random green vector},
\begin{eqnarray*}
\int_{\Omega(\vect x)}\wG_0(\vect x,\vect y_1)\overline{\wG_0(\vect x,\vect y_2)}\,\rmd\vect x&\approx&
\frac{\rme^{-\rmi\kappa\frac{\xi^2+\eta^2}{2L}}}{L^2}\int_{\Omega(\vect x)}\rme^{\rmi\kappa\frac{x_1\xi+x_2\eta}{L}}\,\rmd\vect x\\
&=&\frac{\rme^{-\rmi\kappa\frac{\xi^2+\eta^2}{2L}}}{L^2}\int_{-a/2}^{a/2}\rme^{\rmi\kappa\frac{x_1\xi}{L}}\,\rmd x_1 \int_{-a/2}^{a/2}\rme^{\rmi\kappa\frac{x_2\eta}{L}}\,\rmd x_2\\
&=&\frac{a^2\rme^{-\rmi\kappa\frac{\xi^2+\eta^2}{2L}}}{\pi L^2}\sinc\Big(\frac{a\xi}{\lambda L}\Big)\sinc\Big(\frac{a\eta}{\lambda L}\Big),
\end{eqnarray*}
and similarly,
\begin{equation*}
\int_{B(\vect 0,l)}\wG_0(\vect x,\vect y_1)\overline{\wG_0(\vect x,\vect y_2)}\,\rmd\vect x\approx
\frac{a^2\rme^{-\rmi\kappa\frac{\xi^2+\eta^2}{2L}}}{\pi L^2}\sinc\Big(\frac{l\xi}{\lambda L}\Big)\sinc\Big(\frac{l\eta}{\lambda L}\Big).
\end{equation*}
Note that when $a\ll L$, we can approximate linearly
$
    \log\Big(1+\frac{a^2}{4L^2}\Big)\approx\frac{a^2}{4L^2}.
$
Therefore, the lefthand side of \eqref{eq:statistical stability} is approximately equal to
\begin{multline}
    \frac{\mE\left|\vect\wg^\ast(\vect y_1)\vect\wg(\vect y_2)-\mE\left(\vect\wg^\ast(\vect y_1)\vect\wg(\vect y_2)\right)\right|^2}{\|\vect\wg(\vect y_1)\|^2\|\vect\wg(\vect y_2)\|^2}\numberthis\label{eq:paraxial approx}\\
    \begin{aligned}
    &\approx\frac{a^2l^2\Big(1-\rme^{-\frac{\kappa^2a^2_e}{L^2}|\vect y_1-\vect y_2|^2}\Big)\sinc\Big(\frac{\xi a}{\lambda L}\Big)\sinc\Big(\frac{\eta a}{\lambda L}\Big)\sinc\Big(\frac{\xi l}{\lambda L}\Big)\sinc\Big(\frac{\eta l}{\lambda L}\Big)}{\frac{L^4\pi^2}{256\pi^2}\frac{a^2}{4L^2}\frac{a^2}{4L^2}}\nonumber\\
    &\approx256\pi^2\Big(1-\rme^{-\frac{\kappa^2a_e^2}{L^2}|\vect y_1-\vect y_2|^2}\Big)\Big(\frac{l}{a}\Big)^2\sinc\Big(\frac{\xi a}{\lambda L}\Big)\sinc\Big(\frac{\eta a}{\lambda L}\Big)\sinc\Big(\frac{\xi l}{\lambda L}\Big)\sinc\Big(\frac{\eta l}{\lambda L}\Big).
    \end{aligned}
\end{multline}
This matches the result of logarithmic decay bound in general.
}
\end{remark}

\begin{proof}[Proof of Proposition~\ref{prop:asymp orthogonality}]
    We first look at \eqref{eq:o1} for back-propagation in the true random media scenario.
It has been shown in Proposition~\ref{prop:statistical stability}
    that when size of array is large, the single realization will approach the average value.
    Therefore, it is enough to show the average value goes to zero as points $\vect y_1$ and
    $\vect y_2$ are far apart and the result is true due to the Chebyshev inequality under
    probability measure $\mathbb{P}$ induced from random field $\mu(\cdot)$. Using the moment formula, we have
\begin{equation*}
\mE(\vect\wg^\ast(\vect y_1)\vect\wg(\vect y_2))
=\sum_{j=1}^{N}\overline{\wG_0(\vect x_j,\vect y_1)}\wG_0(\vect x_j,\vect y_2)
\rme^{-\frac{\kappa^2a_e^2}{2L^2}|\vect y_1-\vect y_2|^2}.
\end{equation*}
Since the multiplier factor $\rme^{-\frac{\kappa^2a_e^2}{2L^2}|\vect y_1-\vect y_2|^2}<1$ and
goes to zero as $\frac{|\vect y_1-\vect y_2|}{\lambda}$ increases,
we thus have
\[
    \mE(\vect\wg^\ast(\vect y_1)\vect\wg(\vect y_2))\rightarrow0,\qquad\frac{|\vect y_1-\vect y_2|}{\lambda}\rightarrow\infty
\]
especially when size of array $a$ becomes large. When $a_e$ is large in
random media, average value will decay to zero faster than that in homogeneous media. This 
is why in random media better resolution can be achieved.

Next, we look at \eqref{eq:o2}. First, we compute the expectation of $\nu^2(\vect x,\vect y)$
\begin{equation*}
\mE\nu^2(\vect x,\vect y)=2\int_0^1\int_0^sR\left(\frac{|\vect y-\vect x|}{l}(s-s')\right)\,\rmd s'\,\rmd s
=\frac{2l}{|\vect y-\vect x|}\int_0^1\int_0^{\frac{|\vect y-\vect x|}{l}s}R(s')\,\rmd s'\,\rmd s
\end{equation*}
Under $l\ll L$ and symmetric extension of $R(\cdot)$ to negative real line, we have 
\[\mE\nu^2(\vect x,\vect y)\approx\frac{2l}{|\vect y-\vect x|}\int_0^\infty R(s)\,\rmd s=\frac{l\widehat R(0)}{|\vect y-\vect x|}.\]
Then it is easy to calculate the expectation of mixed inner product as
\begin{eqnarray*}
\mE(\vect\wg^\ast_0(\vect y_1)\vect\wg(\vect y_2))&=&\sum_{j=1}^{N}\overline{\wG_0(\vect x_j,\vect y_1)}
\wG_0(\vect x_j,\vect y_2)\mE\rme^{\rmi\kappa\sigma|\vect x_j-\vect y_2|\nu(\vect x_j,\vect y_2)}\\
&=&\sum_{j=1}^{N}\overline{\wG_0(\vect x_j,\vect y_1)}\wG_0(\vect x_j,\vect y_2)\rme^{-\frac{1}{2}\kappa^2\sigma^2_0|\vect x_j-\vect y_2|^2\mE\nu^2(\vect x_j,\vect y_2)}\\
&\approx&\sum_{j=1}^{N}\overline{\wG_0(\vect x_j,\vect y_1)}\wG_0(\vect x_j,\vect y_2)\rme^{-\frac{1}{2}\kappa^2\sigma^2_0l\widehat R(0)|\vect x_j-\vect y_2|}.
\end{eqnarray*}
Because $\widehat R(0)\ge0$, it implies $\mE(\vect\wg^\ast_0(\vect y_1)\vect\wg(\vect y_2))$
is approximately equal to $\vect\wg_0^\ast(\vect y_1)\vect\wg_0(\vect y_2)$ multiplied 
by a factor less that $1$ and therefore 
\[\mE(\vect\wg^\ast_0(\vect y_1)\vect\wg(\vect y_2))\rightarrow0,\quad\frac{|\vect y_1-\vect y_2|}{\lambda}\rightarrow\infty.\]

To show the statistical stability of the mixed inner product, we first compute the bound of the numerator
similar to that in the proof of Proposition~\ref{prop:asymp orthogonality}.
\begin{equation*}
    \begin{split}
        \mE\Big|\vect\wg^\ast_0(\vect y_1)\vect\wg(\vect y_2)&-\mE\vect\wg_0^\ast(\vect y_1)\vect\wg(\vect y_2)\Big|^2\\
                                                             &\approx\mE\left|\sum_{j=1}^{N}\overline{\wG_0(\vect x_j,\vect y_1)}\wG_0(\vect x_j,\vect y_2)\Big(\rme^{\rmi\kappa\sigma\vartheta(\vect x_j,\vect y_2)}-\rme^{-\frac{1}{2}\kappa^2\sigma^2l\widehat R(0)|\vect x_j-\vect y_2|}\Big)\right|^2\\
                                                             &=\sum_{j,j'=1}^{N}\overline{\wG_0(\vect x_j,\vect y_1)}\wG_0(\vect x_j,\vect y_2)\wG_0(\xjp,\vect y_1)\overline{\wG_0(\xjp,\vect y_2)}\\
                                                             &\qquad\times\rme^{-\frac{1}{2}\kappa^2\sigma^2l\widehat R(0)(|\vect x_j-\vect y_2|+|\xjp-\vect y_2|)}\left(\rme^{\kappa^2\sigma^2\mE\vartheta(\vect x,\vect y_2)\vartheta(\xjp,\vect y_2)}-1\right).
    \end{split}
\end{equation*}
The correlation term $\vartheta(\vect x_j,\vect y_2)\vartheta(\xjp,\vect y_2)$ is nonzero only when the paths
connecting $\vect x_j$, $\xjp$ with $\vect y_2$ are within the correlation length $l$. Also according to
Cauchy-Schwartz inequality, 
\begin{equation*}
    \mE(\vartheta(\vect x_j,\vect y_2)\vartheta(\xjp,\vect y_2))\le\frac{1}{2}\left(\mE\vartheta^2(\vect x_j,\vect y_2)+\mE\vartheta^2(\xjp,\vect y_2)\right).
\end{equation*}
Therefore the numerator is approximately bounded by continuous limit
\begin{equation*}
    \begin{split}
        \mE\Big|&\vect\wg^\ast_0(\vect y_1)\vect\wg(\vect y_2)-\mE\vect\wg_0^\ast(\vect y_1)\vect\wg(\vect y_2)\Big|^2\\
                                                             &\lessapprox\frac{1}{h^4}\Big(1-\rme^{-\kappa^2\sigma^2Ll\widehat R(0)}\Big)\left|\int_{\Omega(\vect x)}\overline{\wG_0(\vect x,\vect y_1)}\wG_0(\vect x,\vect y_2)\,\rmd\vect x\right|\left|\int_{B(\vect 0,l)}\overline{\wG_0(\vect x,\vect y_1)}\wG_0(\vect x,\vect y_2)\,\rmd\vect x\right|\\
                                                             &\lessapprox\frac{1}{256\pi^2h^4}\Big(1-\rme^{-\kappa^2\sigma^2Ll\widehat R(0)}\Big)\log\Big(1+\frac{a^2}{4L^2}\Big)\log\Big(1+\frac{l^2}{4L^2}\Big).
    \end{split}
\end{equation*}
The denominator is norm of the Green function vector which has been calculated in the proof of 
Proposition~\ref{prop:statistical stability}. Thus the ratio in \eqref{eq:mixed stability} is bounded by
\[\frac{\mE\left|\vect\wg_0^\ast(\vect y_1)\vect\wg(\vect y_2)-\mE(\vect\wg_0^\ast(\vect y_1)\vect\wg(\vect y_2))\right|^2}
{\|\vect\wg_0(\vect y_1)\|^2\|\vect\wg(\vect y_2)\|^2}\lessapprox\Big(1-\rme^{-\kappa^2\sigma^2Ll\widehat R(0)}\Big)
\frac{\log\Big(1+\frac{l^2}{4L^2}\Big)}{\log\Big(1+\frac{a^2}{4L^2}\Big)}.\]
When $a\rightarrow\infty$, the right handside goes to zero. The decay rate is again controlled by
the logrithmic factor of $a$ and when $a\ll L$, the decay rate is quadratic of $a$ which
is the same as that in Proposition~\ref{prop:statistical stability}. Due to the statistical stability,
the mixed inner product satisfies \eqref{eq:o2}.
\end{proof}

\end{document}